\documentclass{article}%
\usepackage{amsmath,amssymb,amsthm,amsfonts,verbatim}
\usepackage{enumerate}
\usepackage{microtype}
\usepackage{url}
\usepackage[all,arc]{xy}
\usepackage{mathrsfs}
\usepackage{listings, color} 
\usepackage[pdftex]{graphicx}
\usepackage[section]{placeins}
\usepackage{soul}
\usepackage[utf8]{inputenc}
\usepackage[english]{babel}
\usepackage{color}
\usepackage{xcolor}
\usepackage{mathtools}
\usepackage{hyperref}
\usepackage{listings}
\usepackage{color}
\usepackage{xcolor}

\newtheoremstyle{quest}{\topsep}{\topsep}{}{}{\bfseries}{}{ }{\thmname{#1}\thmnote{ #3}.}
\theoremstyle{quest}

\theoremstyle{plain}
\theoremstyle{definition}
\newtheorem{theorem}{Theorem}[section]
\newtheorem{corollary}[theorem]{Corollary}
\newtheorem{proposition}[theorem]{Proposition}

\newtheorem{lemma}[theorem]{Lemma}

\newtheorem{definition}[theorem]{Definition}

\newtheorem{remark}[theorem]{Remark}

\definecolor{dkgreen}{rgb}{0,0.6,0}
\definecolor{gray}{rgb}{0.5,0.5,0.5}
\definecolor{mauve}{rgb}{0.58,0,0.82}
 
 \allowdisplaybreaks
 \numberwithin{equation}{section} 
 
\newcommand{\p}{\ensuremath{\partial}}
\newcommand{\n}{\ensuremath{\nonumber}}
\newcommand{\eps}{\ensuremath{\epsilon}}
 \newcommand{\bigO}{\ensuremath{\mathcal{O}}}
  \newcommand{\X}{\ensuremath{\mathcal{X}}}

\title{\vspace{-50pt} Steady Prandtl Layers over a Moving Boundary: Non-Shear Euler Flows}
\author{ \Large Sameer Iyer \footnote{\url{sameer_iyer@brown.edu}. Division of Applied Mathematics, Brown University, 182 George Street, Providence, RI 02912, USA. Partially supported by NSF grant DMS-1611695. }}
\date{May 16, 2017}

\DeclareMathOperator{\supp}{\text{supp}}

\newcommand{\ud}{\,\mathrm{d}}

\begin{document}
\maketitle
\vspace{-30pt}
\begin{center}
\end{center}

\begin{abstract}
In this article we establish the validity of Prandtl layer expansions around Euler flows which are not shear. The presence of non-shear flows at the leading order creates a singularity of $\bigO(\frac{1}{\sqrt{\eps}})$. A new $y$-weighted positivity estimate is developed to control this leading-order growth at the far field.
\end{abstract}

\section{Introduction}

We consider the steady, incompressible Navier-Stokes equations on the domain $\Omega = (0,L) \times (0,\infty)$. The boundary consists of three components, $Y = 0, x = 0$, and $x = L$. The system reads: 
\begin{align} \label{sys.main}
 \left.\begin{aligned}
	&U^{NS}U^{NS}_x + V^{NS}U^{NS}_Y + P^{NS}_x = \epsilon \Delta U^{NS}  \\
	&U^{NS} V^{NS}_x + V^{NS} V^{NS}_Y + P^{NS}_Y = \epsilon \Delta V^{NS} \\ 
	& U^{NS}_x + V^{NS}_Y = 0.
       \end{aligned}
 \right\}
  \qquad \text{in $\Omega$}
\end{align}

The system above is taken together with the no-slip boundary condition on $Y = 0$, which in addition is assumed to be moving with velocity $u_b > 0$. The boundary conditions at $x = 0, L$ are inflow and outflow conditions, to be prescribed specifically in the article. 

We are interested in the asymptotic behavior of solutions to (\ref{sys.main}) as $\eps \rightarrow 0$. Such asymptotics must capture the formation of boundary layers, which we now describe in generality. Suppose an outer Euler flow is prescribed: 
\begin{align} \label{Euler.leading.order}
[u^0_e(x,Y), v^0_e(x,Y), P^0_e(x,Y)],
\end{align}

satisfying the Euler equations: 
\begin{align} \label{euler.equation.0}
 \left.\begin{aligned}
	&u^0_e u^{0}_{ex} + v^{0}_eu^{0}_{eY} + P^{0}_{ex} = 0  \\
	&u^0_e v^0_{ex} + v^0_e v^0_{eY} + P^0_{eY} = 0 \\ 
	& u^{0}_{ex} + v^{0}_{eY} = 0,
       \end{aligned}
 \right\}
  \qquad \text{in $\Omega$}
\end{align}

together with the no penetration boundary conditions at $Y = 0, Y \rightarrow \infty$: 
\begin{align}
v^0_e|_{Y = 0} = v^0_e|_{Y \rightarrow \infty} = 0. 
\end{align}

Generically there is a mismatch between the boundary velocity $u^0_e(x,0)$ and $u_b$, indicating that one should not expect solutions of (\ref{sys.main}) to converge to $[u^0_e, v^0_e]$ in the $L^\infty$ norm. Rather, it was proposed in 1904 by Ludwig Prandtl that one should expect the formation of boundary layers, which can be expressed mathematically as an asymptotic expansion: 
\begin{align} \label{exp.0}
&U^{NS}(x,Y) = u^0_e(x,Y) + u^0_p(x, \frac{Y}{\sqrt{\eps}}) + \bigO(\sqrt{\eps}), \\  \n
& V^{NS}(x,Y) = v^0_e(x,Y) + \sqrt{\eps}v^0_p(x, \frac{Y}{\sqrt{\eps}}) + \sqrt{\eps}v^1_e(x,Y) + \bigO(\eps), \\ \n
&P^{NS}(x,Y) = P^0_e(x,Y) + P^0_p(x, \frac{Y}{\sqrt{\eps}}) + \bigO(\sqrt{\eps}). 
\end{align}

The flows considered under the present setup are elliptic. Thus, a mathematical formulation of validating the expansion (\ref{exp.0}) is to assume boundary data are prescribed so that the expansions (\ref{exp.0}) are valid at the boundaries, $x = 0, L$, and to then prove that they must be valid in the interior of the domain, $\Omega$. 

Under the setup described above, (\ref{exp.0}) has been justified rigorously for shear flows in \cite{GN}. Our aim in this article is to generalize the results to non-shear flows that are ``sufficiently close to shear", to be made rigorous by assumption (\ref{euler.as.2}) in our main result. As is evident from (\ref{exp.0}), such a generalization is a \textit{leading order effect}, which when scaled to Prandtl variables creates a singularity of $\bigO(\frac{1}{\sqrt{\eps}})$; this is evident in the specification of (\ref{defn.vs}) below.

Let us briefly highlight the physical importance of developing a method to handle non-shear Eulerian flows. A classical setup from fluid mechanics deals with horizontal flows past a rotating disk, see for instance \cite{Schlicting}. Such a flow is non-shear, as in the set-up considered here. In the simpler case when the flows are actually circular (and therefore shear), as opposed to horizontal, in the presence of a rotating disk, the article of \cite{Iyer} develops machinery to handle the geometry of the boundary. The present article can be viewed as a first step in studying non-shear flows, without adding the complexities of a curved boundary.  

\subsection*{Boundary Layer Expansions}

We will work with scaled, boundary layer variables $y = \frac{Y}{\sqrt{\eps}}$, and consider the scaled Navier-Stokes unknowns: 
\begin{align}
U^\eps(x,y) = U^{NS}(x,Y), \hspace{3 mm} V^\eps(x,y) = \frac{V^{NS}(x,Y)}{\sqrt{\eps}} \hspace{3 mm} P^\eps(x,y) = P^{NS}(x,Y). 
\end{align}

In the new unknowns, the system (\ref{sys.main}) becomes: 
\begin{align}
&U^\eps U^\eps_x + V^\eps U^\eps_y + P^\eps_x = U^\eps_{yy} + \eps U^\eps_{xx}, \\
& U^\eps V^\eps_x + V^\eps V^\eps_y + \frac{P^\eps_y}{\eps} = V^\eps_{yy} + \eps V^\eps_{xx} \\
& U^\eps_x + V^\eps_y.
\end{align}

We start with the following expansions: 
\begin{align} \label{exp.1}
&U^\eps = u^0_e + u^0_p + \sqrt{\eps}u^1_e + \sqrt{\eps}u^1_p + \eps^{\frac{1}{2}+\gamma}u := u_s + \eps^{\frac{1}{2}+\gamma}u, \\ \label{exp.2}
&V^\eps = \frac{v^0_e}{\sqrt{\eps}} + v^0_p + v^1_e + \sqrt{\eps}v^1_p + \eps^{\frac{1}{2}+\gamma}v = v_s + \eps^{\frac{1}{2}+\gamma}v, \\ \label{exp.3}
&P^\eps = P^0_e + P^0_p + \sqrt{\eps}P^1_e + \sqrt{\eps}P^1_p + \eps P^2_p + \eps^{\frac{1}{2}+\gamma}P = P_s + \eps^{\frac{1}{2}+\gamma}P.
\end{align}

We are prescribed the Euler flow: 
\begin{align}
[u^0_e, v^0_e, P^0_e]. 
\end{align}

Importantly, the fact that $u^0_e$ is not shear means that it can have an $x$-dependence. This in turn implies that $v^0_e$ and $P^0_e$ are nonzero. Our analysis does not assume a sign condition for $\p_x P^0_e$. Due to the $x$-dependence of $u^0_e$, it is natural that in the scaled, Prandtl variable, there is a singularity of $\bigO(\frac{1}{\sqrt{\eps}})$, (see below, equation (\ref{defn.vs})).

We will construct the remaining terms in $[u_s, v_s, P_s]$, as defined by (\ref{exp.1}) - (\ref{exp.3}), in Appendix \ref{app.construct}. We will specify the particular equations satisfied by each of the terms in $[u_s, v_s]$ in Appendix \ref{app.construct}. Let us explicitly write the form of $v_s$: 
\begin{align} \label{defn.vs}
v_s = \frac{v^0_e}{\sqrt{\eps}} + v^0_p + v^1_e + \sqrt{\eps}v^1_p.
\end{align}

As can be seen from above, the presence of nonzero $v^0_e$ creates a leading order singularity of $\bigO(\frac{1}{\sqrt{\eps}})$, which is the main difficulty that must be addressed by our analysis. 

The main part of the article will be to construct and control the final term in the expansion, $[u, v, P]$, which we term the ``remainders".  The equations satisfied by the remainders $[u,v,P]$ are specified in (\ref{lin.1}) - (\ref{lin.3}). 

We now discuss the boundary data of each term above. The key point is that the no slip condition on $Y = 0$ must be enforced at each order in the expansion: 
\begin{align}
&u^0_e(x,0) + u^0_p(x,0) = u_b, \hspace{3 mm} u^1_p(x,0) = -u^1_e(x,0), \hspace{3 mm} u(x,0) = 0 \\
&v^0_e(x,0) = 0, \hspace{3 mm} v^1_e(x,0) = -v^0_p(x,0), \hspace{3 mm} v^1_p(x,0) = 0, \hspace{3 mm} v(x,0) = 0.
\end{align}

The boundary data at $x = 0$ must be specified for the Prandtl layers as follows: 
\begin{align}
u^0_p(x,0) = u^0_{p0}(y), \hspace{3 mm} u^1_p(x,0) = u^1_{p0}(y).
\end{align}

The equations for $u^i_p$ are diffusion equations, and so need to only be prescribed initial data at $x = 0$. $v^i_p$ are then recovered via the divergence free condition, and therefore do not need in-flow boundary conditions. We will assume that $u^i_{p0}$ are smooth and exponentially decaying. 

In contrast, the Euler layers, $[u^1_e, v^1_e]$ satisfy an elliptic system, and we must prescribe boundary data at both $x = 0, L$. We do so at the level of the stream function, where $\nabla^\perp \phi^1 = [u^1_e, v^1_e]$:
\begin{align}
\phi^1(0,Y) = \phi^1_0(Y), \hspace{3 mm} \phi^1(L,Y) = \phi^1_L(Y).
\end{align}

These are also assumed smooth and rapidly decaying, and in addition must satisfy a compatibility condition which we call ``well-prepared" boundary data defined in Definition \ref{defn.WP}. 

Finally, we can describe the boundary data for the remainders, $[u,v,P]$: 
\begin{align} \label{str.free.IN}
[u,v]|_{x = 0} = [a_0(y), b_0(y)], \hspace{3 mm} [u,v]|_{y = 0} = [u,v]|_{y \rightarrow \infty} = 0, \\ \label{str.free}
P - 2\eps u_x|_{x = L} = a_L(y), \hspace{3 mm} u_y + \eps v_x|_{x = L} = b_L(y).
\end{align}

The boundary condition at $x = 0$ allows the prescription of in-flow data. The boundary conditions at $x = L$ in (\ref{str.free}) is known as the (inhomogeneous) stress-free boundary condition, and corresponds to evaluating the Cauchy stress tensor at the boundary $x = L$. We will provide assumptions on the boundary data: 
\begin{align} \label{as.aL.bL}
&|\p_y^k a_L| \lesssim \sqrt{\eps} \langle y \rangle^{-N}, \hspace{3 mm} |\p_y^k \{a_0, b_0, b_L\}| \lesssim \langle y \rangle^{-N}, \hspace{3 mm} \supp\{a_0, b_0, a_L, b_L \} \subset \{ y \ge 1\}.
\end{align}

for sufficiently large $k, N$. 

\subsection*{Main Theorem}

In order to state our result, we must introduce the norm in which will control the solution. Define our $\X$ norm to be: 
\begin{align} \n
||u,v||_{\X} &:= ||u_y \cdot y||_{L^2} + ||\sqrt{\eps}u_x \cdot y ||_{L^2} + ||v_y, \sqrt{\eps}v_x||_{L^2} \\ \n
& + ||\Big\{u_{yy}, \sqrt{\eps}u_{xy}, \eps u_{xx} \Big\} \cdot y ||_{L^2} + \eps^{\frac{\gamma}{2}}||u, \sqrt{\eps}v||_{L^\infty} \\ \label{norm.S}
& + ||u,v||_{B}, 
\end{align}

where the boundary norm is given by: 
\begin{align} \label{norm.B}
||u,v||_B := ||u_y \cdot y, \sqrt{\eps}u_x \cdot y||_{L^2(x = L)} + ||\sqrt{\eps}u_x||_{L^2(x = L)}.
\end{align}

We will also have to define the space, $\X$, for which we refer the reader to Appendix B, equation \ref{defn.S}.

\begin{theorem} \label{th.main} Consider an Euler flow $[u^0_e(x,Y), v^0_e(x,Y)]$ satisfying the following hypothesis:
\begin{align} \label{euler.as.1}
& 0 < c_0 \le u^0_e \le C_0 < \infty, \\ \label{euler.as.2}
&||\frac{v^0_e}{Y}||_{L^\infty}  << 1,  \text{ and} \\ \label{euler.as.3}
&||Y^k \nabla^m v^0_e||_{L^\infty} < \infty \text{ for sufficiently large } k,m \ge 0, \\ \label{euler.as.4}
& ||Y^k \nabla^m u^0_e||_{L^\infty} < \infty \text{ for sufficiently large } k \ge 0, m \ge 1.  
\end{align}

Let the interval $L$ be sufficiently small relative to universal constants. Suppose in addition that the boundary data described above are prescribed, assumed to be smooth and rapidly decaying in their arguments, satisfy the assumptions (\ref{as.aL.bL}), and satisfy the compatibility conditions given in Definition \ref{defn.WP}. Then the remainder solutions $[u,v,P]$ exist in the space $\X$ and satisfy the estimate: 
\begin{align}
||u,v||_{\X} \lesssim 1. 
\end{align}

\end{theorem}

\begin{corollary} In the inviscid limit, we have the convergence: 
\begin{align}
||U^{NS} - u^0_e - u^0_p||_{L^\infty} + ||V^{NS} - v^0_e||_{L^\infty} \le \sqrt{\eps}.
\end{align}
\end{corollary}

\begin{remark} The Euler flows which satisfy the assumptions of (\ref{euler.as.1}) - (\ref{euler.as.4}) are plentiful, see Proposition \ref{eul.exist} in Appendix \ref{app.construct}.
\end{remark}

Let us place this result in the context of recent developments in the boundary layer theory. We will restrict to stationary, two dimensional flows. A central task in this setting is to establish validity of an expansion of the type (\ref{exp.0}), and this is considered to be one of the most challenging open problems in fluid mechanics. It has been achieved in the setting of a moving boundary in \cite{GN}, \cite{Iyer}, \cite{Iyer2}. The method introduced by \cite{GN} relies on establishing a crucial positivity estimate which gives $o(1)$ control over the remainder quantity $||v_y, \sqrt{\eps}v_x||_{L^2}$. The flows considered in those works were all shear flows, and the aim of the present result is to generalize (in particular the result of \cite{GN}) to the case of non-shear flows. As can be seen in the expansion (\ref{exp.2}), this is a leading order effect, and therefore requires a new $y$-weighted estimate. 

Within the stationary, two dimensional setting, the recent work of \cite{MD} addresses the related question of blowup of the Prandtl equation in the presence of an unfavorable pressure gradient. For unsteady flows, the validity of an asymptotic expansion of the form (\ref{exp.0}) has been established in the analyticity framework, \cite{Asano}, \cite{Caflisch1}, \cite{Caflisch2}, in the Gevrey setting in \cite{DMM}, for initial vorticity bounded away from the origin in \cite{Mae}, and for special flows in \cite{Taylor}. Giving a more exhaustive survey of results in the unsteady setting would lead us astray, and so we refer the reader to the review articles of \cite{E}, \cite{Temam}, and \cite{MM} and the references therein. 

\subsection*{Overview of Proof}

Let us introduce the system satisfied by the remainders. For our discussion, we will consider the linearized version of system (\ref{NSR.1}) - (\ref{NSR.3}) and regard $f, g$ and generic elements of $L^2$. 
\begin{align} \label{lin.1}
&-\Delta_\eps u + S_u + P_x = f, \\ \label{lin.2}
&-\Delta_\eps v + S_v + \frac{P_y}{\eps} = g, \\ \label{lin.3}
&u_x + v_y = 0,
\end{align}

together with the inhomogeneous boundary conditions: 
\begin{align} \label{BC.dirichlet}
&[u,v]|_{y = 0} = [u,v]|_{x = 0} = [u,v]|_{y \rightarrow \infty} = 0, \\ \label{BC.stress.free}
& P - 2\eps u_x|_{x = L} = a_L(y), \hspace{3 mm} \{u_y + \eps v_x\}|_{x = L} = b_L(y).
\end{align}

In actuality, $f, g$ contain the nonlinear components. Also note that we can, up to redefining $a_L, b_L$, reduce the boundary data from (\ref{str.free.IN}) - (\ref{str.free}) to the homogenized boundary data (\ref{BC.dirichlet}) - (\ref{BC.stress.free}). This is proven in Lemma \ref{lemma.WLOG}. We provide relevant definitions below: 
\begin{align} \label{nl.spec.1}
&S^u = u_s u_x + u_{sx}u + v_s u_y + u_{sy}v, \hspace{3 mm} S^v = u_s v_x + v_{sx}u + v_s v_y + v_{sy}v, \\ \label{nl.spec.2}
& N^u(u,v) = \eps^{\frac{1}{2}+\gamma}\Big[ uu_x + vu_y \Big], \hspace{3 mm} N^v(u,v) = \eps^{\frac{1}{2}+\gamma}\Big[ uv_x + vv_y \Big], \\ \label{nl.spec.3}
& f = \eps^{-\frac{1}{2}-\gamma} R^{u,1} + N^u + L^b_1, \hspace{3 mm} g = \eps^{-\frac{1}{2}-\gamma}R^{v,1} + N^v + L^b_2.  
\end{align}

Here, $R^u, R^v$ are high order profile remainders which are defined specifically in (\ref{defn.Ru1}), (\ref{defn.Rv1}) and estimated in (\ref{sigma}), and $L^b_1, L^b_2$ arise from homogenizing the boundary data, are defined in (\ref{defn.Lb}) and estimated in (\ref{est.Lb}). The important consideration for the purposes of this discussion is the rough specification of $[u_s, v_s]$, which we can write: 
\begin{align}
&u_s \approx u^0_e + u^0_p + \bigO(\sqrt{\eps}), \\ \label{vs.heur}
&v_s \approx \frac{v^0_e}{\sqrt{\eps}} + \bigO(1). 
\end{align}

Here, the Prandtl layer, $u^0_p$ is rapidly decaying in the Prandtl variable, $y$. The main idea is to close a $y$-weighted estimate which can control the $\bigO(\frac{1}{\sqrt{\eps}})$ contribution from $v_s$. The first estimate in our scheme is the basic energy estimate: 
\begin{align} \label{overview.en}
 ||u_y||_{L^2}^2 \lesssim \bigO(L) ||v_y, \sqrt{\eps}v_x||_{L^2}^2 + ||f, \sqrt{\eps}g||_{L^2}^2 + C(a_0, b_0, a_L, b_L).
\end{align}

This estimate is standard, and is obtained by applying $(u,\eps v)$ to the system (\ref{lin.1}) - (\ref{lin.2}). The main coercive term is the $\Delta_\eps$, which yields control over $||u_y, \sqrt{\eps}v_y, \eps v_x||_{L^2}^2$. The $\bigO(\frac{1}{\sqrt{\eps}})$ singular term from $v_s$ does not play a role at the level of energy estimates because a factor of $\p_y$ hits each instance of $v_s$ upon applying the multiplier $(u,\eps v)$. Nevertheless, the energy estimate is too weak to close as a standalone estimate due to large convective terms. For instance, one considers: 
\begin{align}
|\int u_{sy}uv| = |\int \Big( \sqrt{\eps}u^0_{eY} + u^0_{py} + \bigO(\sqrt{\eps}) \Big) uv| \le \bigO(L)||u_x||_{L^2}||v_y||_{L^2}.
\end{align}

Thus, it is required that $v_y$ be controlled at $\bigO(1)$, which is a famous difficulty in the boundary layer theory. This is the content of the next step, which generates the following positivity estimate:
\begin{align} \n
 ||v_y, \sqrt{\eps} v_x||_{L^2}^2 + ||\sqrt{\eps}u_x||_{L^2(x = L)} \lesssim &||u_y||_{L^2}^2 + ||\frac{v^0_e}{Y}||_{L^\infty} ||u_y \cdot y, \sqrt{\eps}v_y \cdot y ||_{L^2}^2 \\ \label{overview.pos} & + ||f, \sqrt{\eps}g||_{L^2}^2 + C(a_0, b_0, a_L, b_L).
\end{align}

The above estimate is generating by applying $[\p_y \frac{v}{u_s}, -\eps \p_x \frac{v}{u_s}]$ to the system (\ref{lin.1}) - (\ref{lin.3}). Here, $\bigO(v^0_e)$ is a constant that can be made small according to the assumption in (\ref{euler.as.2}). This estimate was introduced in the context of shear flows by \cite{GN}, and crucially utilizes the multiplier $\frac{v}{u_s}$ which is able to generate coercivity over $||v_y, \sqrt{\eps}v_x||_{L^2}^2$. We refer the reader to the article of \cite{GN} for more details, but emphasize that the significant difference when addressing non-shear flows is the term $||\frac{v^0_e}{Y}||_{L^\infty} ||u_y \cdot y, \sqrt{\eps}v_y \cdot y||_{L^2}^2$. \textit{The key difficulty} for our analysis is the loss of one $y$-weight on the right-hand side due to this term, which is the leading order effect of the non-shear flow. Specifically, consider the term $v_s u_y$ appearing in $S^u$, as seen from definition (\ref{nl.spec.1}), and recall that according to (\ref{vs.heur}) the leading order of $v_s \approx \frac{v^0_e}{\sqrt{\eps}}$. The outcome then becomes: 
\begin{align} \n
\int v_y^2 &\lesssim \int \frac{v^0_e}{\sqrt{\eps}} u_y v_y = \int \frac{v^0_e}{\sqrt{\eps}y} y u_y v_y \\
& \le ||\frac{v^0_e}{Y}||_{L^\infty} ||u_y \cdot y||_{L^2}||v_y||_{L^2}.
\end{align}

Our first main contribution of this paper is to develop the following $y$-weighted estimate which controls the term $||u_y \cdot y||_{L^2}$ term from (\ref{overview.pos}):
\begin{align} \n
&||\{u_{yy}, \sqrt{\eps}u_{xy}, \eps u_{xx}\} \cdot y||_{L^2}^2 + ||\{u_y, \sqrt{\eps}u_x \} \cdot y ||_{L^2}^2 + ||\{ u_y, \sqrt{\eps}u_x \} \cdot y||_{L^2(x = L)} \\  \n
& \hspace{20 mm} \lesssim ||\sqrt{\eps}u_x||_{L^2(x = L)} + ||u_y||_{L^2}^2 + ||v_y, \sqrt{\eps}v_x||_{L^2}^2 \\ \label{overview.weight} 
& \hspace{20  mm} + \text{Forcing Terms }+ C(a_0, b_0, a_L, b_L).
\end{align}

The key idea is to first apply $\p_y$ to the system (\ref{nl.spec.1}). Let us extract the main terms coming from $S^u$: 
\begin{align}
\p_y S^u = u_s u_{xy} + v_s u_{yy} + u_{syy}v.
\end{align}

We now introduce a mixed weight multiplier $u_y y^2 \cdot 1-x$, where the $1-x$ can take advantage of $\p_x$ integrating by parts: 
\begin{align} \n
\int u_s u_{xy} \cdot u_y y^2 (1-x) &= - \int \frac{u_{sx}}{2} u_y^2 y^2 (1-x) + \int \frac{u_s}{2} u_y^2 y^2 \\ \label{pos.co.1}
& + \int_{x = L} \frac{u_s}{2} u_y^2 y^2 (1-L), \\ \label{pos.co.2}
\int v_s u_{yy} \cdot u_y y^2 (1-x) &= - \int \frac{v_{sy}}{2} u_y^2 y^2 (1-x) - \int v_s y u_y^2 (1-x).
\end{align}

Summing (\ref{pos.co.1}) - (\ref{pos.co.2}), using $u_{sx} + v_{sy} = 0$, and the smallness given in (\ref{euler.as.2}), we have: 
\begin{align} \n
(\ref{pos.co.1}) + (\ref{pos.co.2}) &\gtrsim \int \frac{u_s}{2}u_y^2 y^2 - \int v_s y u_y^2 (1-x) + \int_{x = L} \frac{u_s}{2} u_y^2 y^2 (1-L) \\
& \gtrsim \int \frac{u_s}{2}u_y^2 y^2 + \int_{x = L} \frac{u_s}{2} u_y^2 y^2 (1-L).
\end{align}

At the level of the convection, the additional $\p_y$ is necessary to generate an additional factor of $\sqrt{\eps}$: 
\begin{align} 
\int u_{syy}v \cdot u_y y^2 (1-x) &\approx \int \eps u^0_{eYY} v \cdot u_y y^2 (1-x) \\ \n
& \le ||Y^2 u^0_{eYY}||_{L^\infty} ||\frac{v}{y}||_{L^2} ||u_y \cdot y||_{L^2} \lesssim ||v_y||_{L^2}||u_y y||_{L^2}.
\end{align}

The above series of estimates closes by using the smallness of $L$ and $v^0_e$. Let us make a few remarks. Although the purpose of the weighted estimate, (\ref{overview.weight}), is to capture behavior for large $y$, we cannot introduce a cut-off function that avoids the $y = 0$ boundary into the multiplier for instance by selecting $u_y y^2 (1-x) \chi(y)$. This is because the higher-order terms arising from $\p_y \Delta_\eps$ will generate local terms which cannot be controlled. 

Apart from from the weighted estimate, (\ref{overview.weight}), a second novelty of our analysis is that we treat a large class of inhomogeneous boundary data at $x = 0, x = L$, as is shown in (\ref{str.free.IN}) - (\ref{str.free}). This level of generality is important and very physical: it corresponds to taking measurements of the fluid at the inflow and outflow edges, $x = 0$ and $x = L$, and taking these values as inputs. The technique for treating these boundary conditions is based on Lemma \ref{lemma.WLOG}, proved in Appendix \ref{app.construct}. We first construct an auxiliary divergence free vector field which attains the boundary data from (\ref{str.free.IN}) at $\{x = 0\}$. Using this auxiliary vector field to homogenize then creates a boundary contribution at $x = L$, which cannot be removed by a further homogenization due to the need to preserve the divergence-free condition. This has the effect of contributing several new boundary terms from $x = L$ into the positivity estimate, (\ref{overview.pos}), which must then be controlled. 

A third novelty of our analysis is to develop a scaled, weighted version of Korn's inequality to close the above scheme of estimates. Such an estimate is needed due to the higher order contributions which are created in order to perform estimate (\ref{overview.weight}). In particular, the estimate we prove is a coercivity estimate of the form: 
\begin{align} 
\int \Big[u_{yy}^2 &+ 4\eps u_{xy}^2 + \eps^2 u_{xx}^2 - 2 \eps u_{yy}u_{xx} \Big] y^2 \cdot (1-x) \\ \n
& \gtrsim \int \Big[ u_{yy}^2 + \eps u_{xy}^2 + \eps^2 u_{xx}^2 \Big] y^2 \cdot (1-x) - \text{Acceptable Contributions}.
\end{align}

\subsection*{Notation}

Within lemmas, we will use $X \sim \bigO(\text{LHS})$ and $ X \sim \bigO(\text{RHS})$ to mean $X$ can be controlled, up to a universal constant, by the left-hand side (or right-hand side, respectively) of the lemma we are proving. Quantities denoted by $\bigO(L)$ refer to those which can be made small by making $L$ small, and quantities denoted by $\bigO(v^0_e)$ refer to those which can be made small according to the smallness assumptions in (\ref{euler.as.2}).

\section{Energy Estimate}

We will now give the basic energy estimate. The reader should recall the properties of the profiles, given in Appendix \ref{app.construct}, in particular Lemma \ref{Lemma.Unif}, and the space $\X$, as defined by the norm introduced in (\ref{norm.S}), and the definition in (\ref{defn.S}).
\begin{proposition} \label{prop.energy} Solutions $[u,v,P] \in \X$, as defined by (\ref{defn.S}), to the system (\ref{lin.1}) - (\ref{lin.3}), with the boundary conditions (\ref{BC.dirichlet}) - (\ref{BC.stress.free}), satisfy the following estimate:
\begin{align} \label{en.est.1}
||u_y||_{L^2}^2 + \int_{x = L} \frac{u_s}{2} \Big(u^2 + \eps v^2 \Big) \lesssim \bigO(L) ||v_y, \sqrt{\eps}v_x||_{L^2}^2 + \mathcal{R}_1 + ||a_L, b_L||_{L^2}^2, 
\end{align}
where: 
\begin{align} \label{R.1}
\mathcal{R}_1 := \int f \cdot u + \eps g \cdot v. 
\end{align}
\end{proposition}
\begin{proof}

This follows upon applying $(u, \eps v)$ to the system (\ref{lin.1}) - (\ref{lin.3}). First, we will write the $\Delta_\eps$ terms in the following way: 
\begin{align}
\Delta_\eps u = u_{yy} + 2\eps u_{xx} + \eps v_{xy}, \hspace{5 mm} \Delta_\eps v = 2v_{yy} + \eps \p_x \{u_y + \eps v_x \}.
\end{align}
 
 Using the above representation, we now integrate by parts: 
\begin{align} \n
&- \int u_{yy} \cdot u - \int 2\eps u_{xx}\cdot u - \int \eps v_{xy} u \\ \label{bd.1.3}
& \hspace{20 mm} = \int u_y^2 + \int 2\eps u_x^2 + \int \eps v_x u_y - \int_{x = L} 2 \eps u_xu, \\ \n
&-\int 2v_{yy} \cdot \eps v - \int \eps \p_x \{u_y + \eps v_x \} \cdot v \\   \label{bd11}
& \hspace{20 mm}  = + \int 2\eps v_y^2 + \int \eps^2 v_x^2 + \int \eps u_y v_x - \int_{x = L} \eps v b_L(y), \\ \label{pres.bd.1}
&\int P_x \cdot u + \int P_y v = - \int_{x = L} P u,
\end{align}

For (\ref{bd.1.3}) and (\ref{bd11}) we have used the boundary conditions from (\ref{BC.stress.free}). First, we will estimate the interior term from (\ref{bd11}):
\begin{align}
|\int \eps u_y v_x| \le \sqrt{\eps} ||\sqrt{\eps}v_x||_{L^2} ||u_y||_{L^2} \le \sqrt{\eps} \Big[ ||u_y||_{L^2}^2 + ||\sqrt{\eps}v_x||_{L^2}^2 \Big].
\end{align}

Next, the boundary term from (\ref{bd11}):
\begin{align}
|\int_{x = L} \eps v b_L| \le ||\eps v||_{L^2(x = L)} ||b_L||_{L^2(x = L)} \le \bigO(L) \eps ||\sqrt{\eps}v_x||_{L^2}^2 + ||b_L||_{L^2}^2.
\end{align}

We combine the boundary term from (\ref{bd.1.3}) and (\ref{pres.bd.1}) by invoking the stress free boundary condition, in (\ref{BC.stress.free}): 
\begin{align} \n
-\int_{x = L}\{ P - 2\eps u_x \} \cdot u& = -\int_{x = L} a_L(y) \cdot u \\ \n
& \le ||a_L||_{L^2(x = L)} ||u||_{L^2(x = L)} \\
& \le ||a_L||_{L^2(x = L)}^2 + \bigO(L)||u_x||_{L^2}^2.
\end{align}

We will now move to the terms from $S^u$, as defined in (\ref{nl.spec.1}):
\begin{align}
\int S^u \cdot u = \int \Big[ u_s u_x + u_{sx}u + v_s u_y + u_{sy}v \Big] \cdot u
\end{align}

The most difficult convective term from $S^u$ is: 
\begin{align} \n
|\int u_{sy}uv| &\le \bigO(L) ||u_{sy}\cdot y||_{L^\infty} ||v_y, \sqrt{\eps}v_x||_{L^2}^2.
\end{align}

We have used the estimate (\ref{Lemma.Unif.1}) with $k = 1$. The remaining profile terms: 
\begin{align} \n
\int \{u_s u_x + u_{sx}u + v_s u_y \} u&= \int_{x = L} \frac{u_s}{2}u^2 + \int u_{sx}u^2 \\
& \gtrsim \int_{x = L} \frac{u_s}{2} u^2 - \bigO(L) ||u_{sx}||_{L^\infty} ||u_x||_{L^2}^2.
\end{align}

Notice that crucially, the $v_s \sim \bigO(\frac{1}{\sqrt{\eps}}) v^0_e$ singular term is accompanied by a factor of $\p_y$ which cancels the singularity.  We now move to the profile terms from $S^v$, as defined in (\ref{nl.spec.1}):
\begin{align} \n
\int\{ u_s v_x + v_{sx}u &+ v_s v_y + v_{sy}v \} \eps v \\ \n
&= +\int_{x = L } \frac{1}{2}u_s v^2 + \int \eps v_{sx}uv + \int \eps v_{sy}v^2 \\  \n
& \gtrsim +\int_{x = L } \frac{1}{2}u_s v^2 - ||\sqrt{\eps} v_{sx}||_{L^\infty} \bigO(L) ||u_x||_2 ||\sqrt{\eps}v_x||_2 \\
& + ||v_{sy}||_{L^\infty} \bigO(L) ||\sqrt{\eps}v_x||_{L^2}^2. 
\end{align}

Above, we have again used that $\p_y v_s \bigO(1)$, according to (\ref{Lemma.Unif.1}) with $k = 1$. This concludes the proof.

\end{proof}

\section{Positivity Estimate}

For the positivity estimate, we must work with the new unknown: 
\begin{align} \label{defn.beta}
\beta = \frac{v}{u_s}.
\end{align}

Note that this quantity is well-defined because $u_s > 0$, according to (\ref{gtr}). We first establish the equivalence: 
\begin{lemma} For any function $v$ satisfying $v|_{y = 0} = 0 = v|_{x = 0} = 0$, and $\beta$ defined through (\ref{defn.beta}), the following estimate is valid: 
\begin{align}
||v_y, \sqrt{\eps}v_x||_{L^2}^2 \lesssim ||\beta_y, \sqrt{\eps} \beta_x||_{L^2}^2
\end{align}
\end{lemma}
\begin{proof}

The proof forwards directly from: 
\begin{align} \n
\int v_y^2 &= \int |\p_y\{u_s \beta \}|^2 = \int \Big( u_{sy} \beta + u_s \beta_y \Big)^2 \\
& \lesssim \int u_{sy}^2 \beta^2 + \int u_s^2 \beta_y^2 \lesssim ||y u_{sy}||_{L^\infty}^2 \int u_s^2 \beta_y^2.
\end{align}

We have used above that $\beta|_{y = 0} = 0$, according to the assumptions of the lemma. We have also used estimate (\ref{Lemma.Unif.1}) with $k = 1$. Next, 
\begin{align} \n
\int v_x^2 &= \int |\p_x \{u_s \beta \}|^2 = \int \Big( u_{sx} \beta + u_s \beta_x \Big)^2 \\
& \lesssim \int u_{sx}^2 \beta^2 + \int u_s^2 \beta_x^2 \lesssim \int \beta_x^2.
\end{align}

Above, we have used that $\beta|_{x = 0} = 0$, according to the assumptions of the lemma. This concludes the proof.

\end{proof}

According to the above lemma, it suffices to control $||\nabla_\eps \beta||_{L^2}$, to which we now turn: 

\begin{proposition}[Positivity Estimate] Solutions $[u,v,P] \in \X$, defined in (\ref{norm.S}), (\ref{defn.S}), to the system (\ref{lin.1}) - (\ref{lin.3}), with the boundary conditions (\ref{BC.dirichlet}) - (\ref{BC.stress.free}) satisfy: 
\begin{align} \n
||\beta_y, \sqrt{\eps}\beta_x||_{L^2}^2 + \int_{x = L} \frac{1}{u_s} \eps v_y^2 \lesssim &||u_y||_{L^2}^2 + \bigO(v^0_e) ||u_y \cdot y, \sqrt{\eps}v_y \cdot y||_{L^2}^2 \\ \label{pos.est.1}
& + \mathcal{R}_2 + ||b_L, \p_y b_L, \frac{a_L}{\sqrt{\eps}}||_{L^2(x = L)}^2. 
\end{align}
where: 
\begin{align} \label{R.2}
\mathcal{R}_2 := -\int f \cdot \beta_y + \eps g \cdot \beta_x.  
\end{align}
\end{proposition}
\begin{proof}

We will apply to the system (\ref{lin.1}) - (\ref{lin.3}) the multiplier:
\begin{align}
[-\beta_y, +\eps \beta_x].
\end{align}

Our analysis consists of a series of steps, which we now detail: 

\vspace{2 mm}

\textit{Step 1: $S_u$ Profile Terms}

\vspace{2 mm}

Referring to the definition of $S^u$ in (\ref{nl.spec.1}), we have via the divergence-free condition: 
\begin{align} \label{su.simp}
S^u = -u_s v_y + u_{sy}v + v_s u_y + u_{sx}u = -u_s^2 \beta_y + v_s u_y + u_{sx}u.
\end{align}

We gain: 
\begin{align}
\int -u_s^2 \beta_y \cdot -\beta_y = \int u_s^2 \beta_y^2.
\end{align}

Referring to the definition of $v_s$ in (\ref{exp.2}), the main convective term is: 
\begin{align}
-\int v_s u_y \cdot \beta_y = \int \Big(\frac{v^0_e}{\sqrt{\eps}} \Big) u_y \beta_y + \int \{v^0_p + v^1_e + \sqrt{\eps}v^1_p\} u_y \beta_y.
\end{align}

The lowest order term is the most dangerous: 
\begin{align} \n
|\int \frac{v^0_e}{\sqrt{\eps}y} y u_y \beta_y | &\le ||\frac{v^0_e}{Y}||_{L^\infty} ||u_y \cdot y||_{L^2} ||\beta_y||_{L^2} \\ \label{danger.1.1}
& \le \bigO(v^0_e) \Big[ ||u_y \cdot y||_{L^2}^2 + ||\beta_y||_{L^2}^2 \Big].
\end{align}

Here we need the small parameter $||\frac{v^0_e}{Y}||_{L^\infty}$. For the higher-order contributions: 
\begin{align} \n
|\int\Big( v_s - \frac{v^0_e}{\sqrt{\eps}} \Big) u_y \beta_y| &\le ||v_s - \frac{v^0_e}{\sqrt{\eps}} ||_{L^\infty}||u_y||_{L^2}||\beta_y||_{L^2} \\
& \le \delta ||\beta_y||_{L^2}^2 + N_\delta ||u_y||_{L^2}^2.
\end{align}

Finally, the last term from (\ref{su.simp})
\begin{align} 
|\int u_{sx}u \cdot -\beta_y| &\le \bigO(L) ||u_{sx}||_{L^\infty} ||u_x||_{L^2} ||\beta_y||_{L^2} \lesssim \bigO(L) ||\beta_y||_{L^2}^2.
\end{align}

\vspace{2 mm}

\textit{Step 2: $S_v$ Profile Terms}

\vspace{2 mm}

Referring to the definition of $S^v$ in (\ref{nl.spec.1}), here we will be treating: 
\begin{align}
\int S_v \cdot - \eps \p_x \{u y^2 w \} = \int \Big( u_s v_x + v_{sx}u + v_s v_y + v_{sy}v \Big) \cdot - \eps \p_x \{u y^2 w \}
\end{align}

First: 
\begin{align} \n
&\int u_s v_x \cdot \eps \beta_x = \int \eps u_s^2 \beta_x^2 + \int \eps u_s u_{sx} \beta \beta_x \\
& \hspace{20 mm} \gtrsim \int \eps u_s^2 \beta_x^2 -  \bigO(L) \int \eps u_s^2 \beta_x^2 \gtrsim \int \eps u_s^2 \beta_x^2, \\
&|\int v_{sx}u \cdot \eps \beta_x| \le \bigO(L) ||\sqrt{\eps}v_{sx}||_{L^\infty} ||u_x||_{L^2} ||\sqrt{\eps}\beta_x||_{L^2}.
\end{align}

Next, we will use the smallness of $||\frac{v^0_e}{Y}||_{L^\infty}$:
\begin{align} \n
&|\int \frac{v^0_e}{\sqrt{\eps}}v_y \cdot \eps \beta_x| \le || \frac{v^0_e}{\sqrt{\eps}y}||_{L^\infty} ||\sqrt{\eps} v_y \cdot y||_{L^2} ||\sqrt{\eps}\beta_x||_{L^2} \\
& \hspace{23 mm} \le \bigO(v^0_e) \Big[ ||\sqrt{\eps}v_y \cdot y||_{L^2}^2 + \bigO(\text{LHS}) \Big], \\ \n
& | \int \{v_s - \frac{v^0_e}{\sqrt{\eps}} \} v_y \cdot \eps \beta_x| \le \sqrt{\eps} ||v_s - \frac{v^0_e}{\sqrt{\eps}}||_{L^\infty} ||v_y ||_{L^2} ||\sqrt{\eps} \beta_x||_{L^2} \\ 
& \hspace{23 mm} \le \sqrt{\eps} \bigO(\text{LHS}), \\
&|\int v_{sy}v \cdot \eps \beta_x | \le \bigO(L) ||v_{sy}||_{L^\infty} ||\sqrt{\eps}v_x||_{L^2} ||\sqrt{\eps} \beta_x||_{L^2}.
\end{align}

\vspace{2 mm}

\textit{Step 3: Pressure Terms}
\begin{align} \n
\int P_x \cdot -\beta_y &+ \int P_y \cdot \beta_x = -\int_{x = L} P \beta_y = -\int_{x = L} 2 \eps u_x \beta_y - \int_{x =L} a_L \beta_y \\ \n
& = +\int_{x =L} 2\eps \frac{v_y^2}{u_s} - \int_{x = L} 2\eps u_x v \p_y \{  \frac{1}{u_s}\}-  \int_{x =L} a_L \beta_y \\ \n
& = +\int_{x =L} 2\eps \frac{v_y^2}{u_s} - \bigO(L) ||u_{sy}||_{L^\infty} ||\sqrt{\eps}v_x||_{L^2} ||\sqrt{\eps}v_y||_{L^2(x = L)} \\  \label{bd.crucial}
& \hspace{25 mm} - \int_{x = L}  a_L \beta_y.
\end{align}

The above term crucially yields control over the boundary term appearing in (\ref{pos.est.1}). We must estimate the contribution: 
\begin{align} \n
|\int_{x = L} a_L \beta_y| &\le ||\frac{a_L}{\sqrt{\eps}}||_{L^2(x = L)} ||\sqrt{\eps} \beta_y||_{L^2(x = L)}  \\
& \lesssim N_\delta  ||\frac{a_L}{\sqrt{\eps}}||_{L^2(x = L)}^2 + \delta ||\sqrt{\eps} \beta_y||_{L^2(x = L)}^2, 
\end{align}

the latter of which can be absorbed into (\ref{bd.crucial}).

\vspace{2 mm}

\textit{Step 4: Vorticity Terms}

\vspace{2 mm}

We will now move to the vorticity terms from (\ref{lin.1}) - (\ref{lin.3}), where the stress-free boundary condition shown in (\ref{BC.stress.free}) will be used repeatedly.  
\begin{align}  \n
+\int u_{yy} \beta_y &= \int u_{yy} \cdot \frac{v_y}{u_s} - u_{yy}v \frac{u_{sy}}{u_s^2} \\  \n
& = -\int u_y \p_y \{ \frac{v_y}{u_s} \} + u_y \p_y \{ v \frac{u_{sy}}{u_s^2} \} \\  \n
& = -\int u_y \frac{v_{yy}}{u_s} - \int 2 u_y v_y \p_y \{ \frac{1}{u_s} \} + \int u_y v \p_y \frac{u_{sy}}{u_s^2} \\  \n
& = - \int \frac{u_y^2}{2} \p_x \frac{1}{u_s} + \int_{x =L} \frac{u_y^2}{2u_s} - \int 2u_y v_y \p_y \{ \frac{1}{u_s} \} \\  \n
&+ \int u_y v \p_y \{ \frac{u_{sy}}{u_s^2} \} \\ \label{posvort.1}
& \gtrsim \int_{x = L} \frac{u_y^2}{2u_s} - ||u_{sx}, u_{sy}, y u_{sy}^2, y u_{syy}||_{L^\infty} \Big[ N_\delta ||u_{y}||_{L^2}^2 + \delta ||v_y||_{L^2}^2 \Big].
\end{align}

The boundary term above, as with all boundary terms from this set of calculations, will be put into (\ref{put}), and subsequently estimated. Next: 
\begin{align}  \n
+\int \eps u_{xx} \beta_y &= - \int \eps u_x \beta_{xy} + \int_{x = L} \eps u_x \beta_y \\  \n
& = -\int \eps u_x \p_x \{ \frac{v_y}{u_s} - v \frac{u_{sy}}{u_s^2} \} + \int_{x = L} \eps u_x \beta_y \\  \n
& = - \int \frac{\eps}{2} u_x^2 \p_x \{ \frac{1}{u_s} \} + \int_{x = L} \frac{\eps}{2 u_s} u_x^2  - \int \eps u_x v_y \p_x \frac{1}{u_s} \\  \label{posvort.2}
& + \int \eps u_x v_x \frac{u_{sy}}{u_s^2}  + \int \eps u_x v \p_x \{ \frac{u_{sy}}{u_s^2} \} + \int_{x = L} \eps u_x \beta_y.
\end{align}

The boundary terms are estimated as in: 
\begin{align}  \n
+ \int_{x = L} \frac{\eps}{2u_s} u_x^2 &+ \int_{x = L} \eps u_x \beta_y \\  \n
&= - \int_{x = L} \eps u_x^2 \frac{1}{2u_s} + \int_{x = L} \eps u_x v \frac{u_{sy}}{u_s^2} \\ \label{posvort.3}
& \le - \int_{x = L} \eps u_x^2 \frac{1}{2 u_s} + \bigO(L) ||u_{sy}||_{L^\infty} ||\sqrt{\eps}v_x||_{L^2} ||\sqrt{\eps }u_x||_{L^2(x = L)},
\end{align}

the final term above being absorbed into (\ref{bd.crucial}) using the smallness of $L$. The bulk terms are estimated via: 
\begin{align}  \n
|\int \frac{\eps}{2} u_x^2 \p_x \{ \frac{1}{u_s} \}| &+ |\int \eps u_x v_y \p_x \frac{1}{u_s}| + |\int \eps u_x v_x \frac{u_{sy}}{u_s^2}| + |\int \eps u_x v \p_x \frac{u_{sy}}{u_s^2}| \\ \label{posvort.4}
& \le \sqrt{\eps} ||u_{sx}, u_{sy}, u_{sxy}||_{L^\infty} \Big[  ||u_x||_{L^2}^2 + ||\sqrt{\eps}v_x||_{L^2}^2 \Big].
\end{align}

Next: 
\begin{align}  \n
-\int \eps v_{yy} \beta_x &= + \int \eps v_y \beta_{xy} \\  \n
& = +\int \eps v_y \p_y \{ \frac{v_x}{u_s} - \frac{u_{sx}}{u_s^2}v \} \\  \n
& = + \int \eps v_y \Big( \frac{v_{xy}}{u_s} - v_x \frac{u_{sy}}{u_s^2} - \p_y \{ \frac{u_{sx}}{u_s^2} \}v - \frac{u_{sx}}{u_s^2} v_y \Big) \\  \n
& = +\int \frac{\eps}{2} \frac{u_{sx}}{u_s^2} v_y^2 + \int_{x = L} \frac{\eps}{2 u_s} v_y^2 - \int \eps v_x v_y \frac{u_{sy}}{u_s^2} \\  \n
& - \int \eps vv_y \p_y \{ \frac{u_{sx}}{u_s^2} \} - \int \eps v_y^2 \frac{u_{sx}}{u_s^2} \\  \label{posvort.5}
& \gtrsim \int_{x = L} \frac{\eps}{2u_s} v_y^2 - ||u_{sx}, u_{sy}, u_{sxy}||_{L^\infty}| \Big[ \bigO(L) ||\sqrt{\eps}v_x||_{L^2}^2 + \sqrt{\eps}||v_y||_{L^2}^2. \Big]
\end{align}

Finally: 
\begin{align} \n
-\int \eps^2 v_{xx} \beta_x &= - \int \eps^2 v_{xx} \frac{v_x}{u_s} - \int \eps^2 v_{xx} v \p_x \frac{1}{u_s} \\  \n
& = + \int \frac{\eps^2}{2} \p_x \frac{1}{u_s} v_x^2 - \int_{x = L} \frac{\eps^2}{2u_s}v_x^2 \\  \n &+ \int_{x = 0} \frac{\eps^2}{2u_s} v_x^2  + \int \eps^2 v_x^2 \p_x \frac{1}{u_s} \\  \n & 
+ \int \eps^2 vv_x \p_{xx} \{ \frac{1}{u_s} \} - \int_{x = L} \eps^2 vv_x \p_x \{ \frac{1}{u_s} \} \\  \n
& \gtrsim \int_{x = 0} \frac{\eps^2}{2u_s} v_x^2 - \int_{x = L} \frac{\eps^2}{2u_s} v_x^2 - \int_{x = L} \eps^2 vv_x \p_x \{ \frac{1}{u_s} \} \\  \label{posvort.6}
& - \eps ||u_{sx}, u_{sxx}||_{L^\infty} \bigO(L) ||\sqrt{\eps}v_x||_{L^2}^2.
\end{align}

Collecting the highest order $x = L$ boundary contributions from (\ref{posvort.1}), (\ref{posvort.3}), (\ref{posvort.5}), (\ref{posvort.6}):
\begin{align} \n
&+ \int_{x = L} \frac{u_y^2}{2u_s} - \int_{x = L} \frac{\eps}{2 u_s} u_x^2 + \int_{x = L} \frac{\eps}{2u_s}v_y^2  - \int_{x = L} \frac{\eps^2}{2u_s}v_x^2 \\ \n
& = + \int_{x = L} \frac{u_y^2}{2u_s}- \int_{x = L} \frac{\eps^2}{2u_s}v_x^2 \\ \n
& = + \int_{x = L}  \frac{u_y^2}{2u_s} - \int_{x = L} \frac{(u_y + \eps v_x - u_y)^2}{2u_s} \\ \n
& = + \int_{x = L}  \frac{u_y^2}{2u_s} - \int_{x = L} \frac{(b_L- u_y)^2}{2u_s} \\ \n
& = + \int_{x = L}  \frac{u_y^2}{2u_s} - \int_{x = L} \frac{b_L^2}{2u_s} - \int_{x = L} \frac{u_y^2}{2u_s} + \int_{x = L} \frac{1}{u_s}b_L u_y \\ \n
& = -\int_{x = L} \frac{b_L^2}{2u_s} - \int_{x = L} u \p_y \{ \frac{b_L}{u_s} \} \\ \n
& \lesssim ||b_L, \p_y b_L||_{L^2(x = L)}^2 + ||u||_{L^2(x = L)}^2 \\ \label{put}
& \lesssim ||b_L, \p_y b_L||_{L^2(x = L)}^2 + \bigO(L) ||u_x||_{L^2}^2.
\end{align}

The final boundary term from (\ref{posvort.6}) can be estimated via: 
\begin{align} \n
|\int_{x = L} \eps^2 vv_x \p_x \{\frac{1}{u_s} \}| &= |\int \eps v u_y \p_x\{ \frac{1}{u_s} \}| \\ \n
& = |\int \eps v_y u \p_x \{ \frac{1}{u_s} \}| + |\int \eps v u \p_{xy} \{ \frac{1}{u_s} \}| \\ \label{bd.2}
& \le \sqrt{\eps} ||u_{sx},y \p_{xy} \{ \frac{1}{u_s} \}||_{L^\infty} || \sqrt{\eps}v_y||_{L^2(x = L)} \bigO(L) ||u_x||_{L^2}
\end{align}

The boundary contribution from (\ref{bd.2}) can be absorbed into (\ref{bd.crucial}). This concludes the proof.

\end{proof}

\section{Weighted Estimates}

In this section, we will bootstrap to the weighted estimates described in (\ref{overview.weight}). By differentiating the system (\ref{lin.1}) - (\ref{lin.3}), we have: 
\begin{align} \label{lin.diff.1}
&-\Delta_\eps u_y + P_{xy} + \p_y S_u = \p_y f \\ \label{lin.diff.2}
&-\Delta_\eps v_y + \frac{P_{yy}}{\eps} + \p_y S_v = \p_y g,
\end{align}

where: 
\begin{align} \label{pySu}
&\p_y S_u = u_s u_{xy} + v_s u_{yy} + u_{syy}v + u_{sxy}u \\ \label{pySv}
&\p_y S_v = u_s v_{xy} + v_s v_{yy} + u_{sy}v_x + v_{sxy}u + v_{sx}u_y + 2 v_{sy}v_y + v_{syy}v.
\end{align}

We will now prove the main weighted estimate. The reader should keep in mind Lemma \ref{Lemma.Unif} which will be in constant use. 

\begin{proposition} \label{weight.prop} Consider $[u,v,P] \in \X$ solutions to (\ref{lin.1}) - (\ref{lin.3}), with the boundary conditions (\ref{BC.dirichlet}) - (\ref{BC.stress.free}). Such a solution satisfies the following estimate:
\begin{align} \n
||\Big\{u_{yy}, &\sqrt{\eps}u_{xy}, \eps  u_{xx} \Big\} \cdot y ||_{L^2}^2 + ||\Big\{u_y, \sqrt{\eps}u_x \Big\} \cdot y ||_{L^2}^2 \\ \n
& + ||\Big\{u_y, \sqrt{\eps} u_x \Big\} \cdot y||_{L^2(x = L)}^2  \lesssim ||u_y||_{L^2}^2 + ||v_y, \sqrt{\eps}v_x||_{L^2}^2 \\ \label{weight.est.1}
& + ||\sqrt{\eps}u_x||_{L^2(x = L)}^2 + ||\{ a_L, \p_y a_L, b_L, \p_y b_L \} \langle y \rangle^2||_{L^2(x = L)}^2 + \mathcal{R}_3,
\end{align}

where: 
\begin{align} \label{R.3}
\mathcal{R}_3 := \int \p_y f \cdot \p_y \{u y^2 w \} - \eps \p_y g \cdot \p_x \{u y^2 w \}.
\end{align}
\end{proposition}

\begin{proof}

We will apply the weighted multiplier: 
\begin{align}
\Big[ \p_y\{u w(x)  y^2 \}, -\eps \p_x \{u w(x) y^2 \} \Big],
\end{align}

where $w(x) = 1-x$. The analysis proceeds in several steps which we will now detail. 

\vspace{2 mm}

\textit{Step 1: Positive Profile Terms}

\vspace{2 mm}

We will now generate the positive quantities on the left-hand side of (\ref{weight.est.1}), by considering from (\ref{pySu}) - (\ref{pySv}) the following terms: 
\begin{align} \label{positive.1}
\int \Big(u_s u_{xy} + v_s u_{yy} \Big) \cdot \p_y\{u w y^2 \} - \eps \int \Big( u_s v_{xy} + v_s v_{yy} \Big) \cdot \p_x \{u w y^2 \}.
\end{align}

First from (\ref{positive.1}): 
\begin{align} \n
\int u_s u_{xy} \cdot \p_y \{ u y^2 w \} &=\int  u_s u_{xy} u_y y^2 w + \int 2u_s u_{xy} u yw \\ \n
& = - \int \frac{u_{sx}}{2} u_y^2 y^2 w + \int \frac{u_s}{2} u_y^2 y^2 + \int_{x = L} \frac{u_s}{2} u_y^2 y^2 w \\ \n
& \hspace{25 mm} - \int 2 u_x \p_y \{ u_s u y \}w \\ \n
& = - \int \frac{u_{sx}}{2} u_y^2 y^2 w + \int \frac{u_s}{2} u_y^2 y^2 + \int_{x = L} \frac{u_s}{2} u_y^2 y^2 w \\ \label{pos.1}
& - \int 2 u_x u_s u_y y w - \int 2 u_x u_{sy}u y w - \int 2 u_x u_s u w.
\end{align}

The final three terms above are estimated: 
\begin{align}
&|\int 2u_s u_x u_y yw| \le \delta ||u_y y||_{L^2}^2 + N_\delta ||u_x||_{L^2}^2, \\
&|\int 2u_{sy} u_x uy w| \le \bigO(L) ||u_{sy} y||_{L^\infty} ||u_x||_{L^2}^2, \\
&|\int 2u_s uu_x w| \le \bigO(L) ||u_x||_{L^2}^2.
\end{align}

Next from (\ref{positive.1}):
\begin{align} \n
\int v_s u_{yy} \p_y \{u y^2 w \} &= \int v_s u_{yy} u_y y^2 w + \int 2 v_s u_{yy} u w y \\ \n
& = - \int \frac{v_{sy}}{2} u_y^2 y^2 w - \int v_s y u_y^2 w - \int 2 v_{sy}uu_y y w \\ \n
&  - \int 2w v_s uu_y - \int 2 v_s y u_y^2 w \\ \n
& = - \int \frac{v_{sy}}{2} u_y^2 y^2 w - \int 3 v_s y u_y^2 w \\ \label{pos.2}
& - \int 2 v_{sy}y uu_y w + \int wv_{sy} u^2.
\end{align}

The final two terms above are estimated: 
\begin{align}
&|\int 2 v_{sy} y uu_yw| \le \bigO(L) ||v_{sy}||_{L^\infty} ||yu_y||_{L^2} ||u_x||_{L^2}, \\
&|\int v_{sy} u^2 w| \le \bigO(L) ||v_{sy}||_{L^\infty} ||u_x||_{L^2}^2.
\end{align}

Summing (\ref{pos.1}) - (\ref{pos.2}):
\begin{align} \label{summary.1}
(\ref{pos.1}) + (\ref{pos.2}) &\gtrsim \int \{\frac{u_s}{2} y^2 - 3v_s y w \} u_y^2  + \int_{x= L} \frac{u_s}{2} u_y^2 y^2 w - \bigO(\text{RHS}). 
\end{align}

We will consider the $v_s$ term above. At leading order: 
\begin{align} \label{d.1}
-\int \frac{v^0_e}{\sqrt{\eps}} y w u_y^2 = - \int \frac{v^0_e}{Y} y^2 w u_y^2 \le_{|\cdot|} ||\frac{v^0_e}{Y}||_{L^\infty} ||u_y y||_{L^2}^2.
\end{align}

Here we use that $||\frac{v^0_e}{Y}||_{L^\infty}$ is taken sufficiently small by assumption (\ref{euler.as.2}) to absorb into the positive contribution from (\ref{summary.1}.) The higher order contributions can be estimated: 
\begin{align} \n
|\int \{v_s - \frac{v^0_e}{\sqrt{\eps}} \} y u_y^2 w| &\le ||v_s - \frac{v^0_e}{\sqrt{\eps}}||_{L^\infty} ||u_y \cdot y||_{L^2} ||u_y||_{L^2} \\
& \lesssim \delta ||u_y \cdot y||_{L^2}^2 + N_\delta ||u_y||_{L^2}^2.
\end{align}

Ultimately this yields: 
\begin{align}
(\ref{pos.1}) + (\ref{pos.2}) &\gtrsim \int \frac{u_s}{2} y^2 u_y^2 + \int_{x = L} \frac{u_s}{2} u_y^2 y^2 - \bigO(\text{RHS}).
\end{align}

We now move to the positive terms from $\p_y S_v$:
\begin{align} \n
-\int \eps u_s v_{xy} \cdot \p_x \{u y^2 w \} &= +\int \eps u_s v_{xy} v_y y^2 w + \int \eps u_s v_{xy} u y^2 \\ \n
& = -\int \frac{u_{sx}}{2} v_y^2 y^2 \eps w + \int \eps u_s \frac{v_y^2}{2}y^2 + \int_{x = L} \eps \frac{u_s}{2} v_y^2 y^2 w \\ \n
& - \int \eps v_y \p_x \{ u_s u y^2 \} + \int_{x = L} u_s v_y u y^2 \eps \\ \n
& = -\int \frac{u_{sx}}{2} v_y^2 y^2 \eps w + \int \eps u_s \frac{v_y^2}{2}y^2 + \int_{x = L} \eps \frac{u_s}{2} v_y^2 y^2 w \\ \n
& - \int \eps v_y u_{sx} u y^2 - \int \eps v_y u_s u_x y^2 + \int_{x = L} u_s v_y u y^2 \eps \\ \n
 &= -\int \frac{u_{sx}}{2} v_y^2 y^2 \eps w + \int \eps\frac{3}{2} u_s v_y^2 y^2 + \int_{x = L} \eps \frac{u_s}{2} v_y^2 y^2 w \\ \label{pos.4}
& - \int \eps v_y u_{sx} u y^2  + \int_{x = L} u_s v_y u y^2 \eps.
\end{align}

We will estimate the final two terms from (\ref{pos.4}):
\begin{align}
&|\int \eps u_{sx}v_y u y^2| \le \bigO(L) ||u_{sx}||_{L^\infty} ||\sqrt{\eps} u_x y||_{L^2}^2, \\
&|\int_{x = L} \eps u_s v_y u y^2 | \le ||\sqrt{\eps} v_y y||_{L^2(x = L)} ||\sqrt{\eps}u y||_{L^2(x = L)}.
\end{align} 

Finally, from above: 
\begin{align}
||\sqrt{\eps} uy||_{L^2(x = L)}^2 = \int_{x = L} \eps u^2 y^2 \le \bigO(L) ||\sqrt{\eps}u_x y||_{L^2}^2. 
\end{align}

Next: 
\begin{align} \n
-\int \eps v_s v_{yy} &\p_x \{u w y^2 \} = +\int \eps v_s v_{yy}v_y w y^2 + \int \eps v_s v_{yy} u y^2 \\ \n
&= - \int \eps \frac{v_y^2}{2} \p_y \{v_s w y^2 \} - \int \eps v_y \p_y \{ v_s u y^2 \} \\ \n
& = - \int \eps \frac{v_{sy}}{2} v_y^2 y^2 w - \int \eps v_s y w v_y^2 - \int \eps v_{sy} v_y u y^2 \\ \label{pos.3}
& - \int \eps v_s v_y u_y y^2 -  \int 2\eps v_s v_y u y. 
\end{align}

We will estimate three of the terms above: 
\begin{align}
&|\int \eps v_{sy}v_y u y^2 | \le \bigO(L) ||v_{sy}||_{L^\infty} ||\sqrt{\eps}v_y y||_{L^2}^2, \\ 
&|\int 2\eps v_s v_y u y| \le \bigO(L) ||\sqrt{\eps}v_s||_{L^\infty} ||\sqrt{\eps} v_y y||_{L^2} ||u_x||_{L^2}, \\ \label{small.3}
&|\int \eps v_s v_y u_y y^2| \le |\int \eps \frac{v^0_e}{\sqrt{\eps}} v_y u_y y^2| + |\int \eps \{v_s - \frac{v^0_e} {\sqrt{\eps}} \}v_y u_y y^2| \\ \n
& \hspace{23 mm} \le ||v^0_e \sqrt{\eps}y||_{L^\infty} ||v_y||_{L^2} ||u_y \cdot y||_{L^2} \\ \n
& \hspace{30 mm} + \sqrt{\eps} ||v_s - \frac{v^0_e}{\sqrt{\eps}}||_{L^\infty} ||\sqrt{\eps}v_y y||_{L^2}||u_y y||_{L^2} \\  \n
& \hspace{23 mm} \lesssim ||v^0_e \cdot Y||_{L^\infty}  ||v_y||_{L^2} ||u_y y||_{L^2} + \sqrt{\eps} \bigO(\text{LHS})  \\
& \hspace{23 mm} \lesssim \delta ||u_y \cdot y||_{L^2}^2 + N_\delta ||v_y||_{L^2}^2 + \sqrt{\eps}\bigO(\text{LHS}).
\end{align}

The $||u_y y||_{L^2}^2$ term can be absorbed into the positive contribution from (\ref{pos.4}), whereas the $||v_y||_{L^2}^2$ term is $\bigO(\text{RHS})$. Thus, summing (\ref{pos.3}) and (\ref{pos.4}) yields: 
\begin{align}
(\ref{pos.3}) + (\ref{pos.4}) \gtrsim \int \Big( \frac{3}{2}u_s y^2 - v_s yw \Big) \eps v_y^2 - \bigO(\text{RHS}).
\end{align}

We must now examine the $v_s$ term above: 
\begin{align}
&\int \{ v^0_p + \sqrt{\eps}v^1_p \} y w \eps v_y^2 \le ||v^0_p y, \sqrt{\eps} v^1_p y||_{L^\infty} ||\sqrt{\eps}v_y||_{L^2}^2, \\
&|\int \eps v^1_e y w v_y^2| \le ||v^1_e Y||_{\infty} \sqrt{\eps} ||v_y||_{L^2}^2, \\
&|\int \sqrt{\eps} v^0_e y w v_y^2| \le ||v^0_e \cdot Y||_{L^\infty} ||v_y||_{L^2}^2,
\end{align}

all of which are acceptable contributions according to the right-hand side of (\ref{weight.est.1}). Summarizing this set of calculations: 
\begin{align} \label{summar.1}
(\ref{positive.1}) \gtrsim \int \frac{u_s}{2} y^2 \Big(u_y^2 + \eps v_y^2 \Big) + \int_{x = L} \frac{u_s}{2} \Big( u_y^2 y^2 + \eps v_y^2 \Big) - \bigO(\text{RHS}).
\end{align}

\vspace{2 mm}

\textit{Step 2: Remaining Profile Terms}

\vspace{2 mm}

We now extract the remaining terms from (\ref{pySu}) - (\ref{pySv}):
\begin{align} \n
\int u_{syy}v \cdot \p_y \{u y^2 w \} &= \int u_{syy}v u_y y^2 w + \int u_{syy} v u 2y w \\ \n
& \le ||u_{syy} y^2||_{L^\infty} ||\frac{v}{y}||_{L^2} ||u_y y||_{L^2} \\ \n
& \hspace{10 mm} + \bigO(L) ||u_{syy} y^2||_{L^\infty} ||\frac{v}{y}||_{L^2}||u_x||_{L^2} \\ \n
& \le ||u_{syy} y^2||_{L^\infty} ||v_y||_{L^2} ||u_y y||_{L^2} \\ \n
& \hspace{10 mm}  + \bigO(L) ||u_{syy} y^2||_{L^\infty} ||v_y||_{L^2}||u_x||_{L^2}, \\ \n
& \le ||u_{syy} y^2||_{L^\infty} \Big[ N_\delta ||v_y||_{L^2}^2 + \delta ||u_y y||_{L^2}\Big] \\ 
& \hspace{10 mm}  + \bigO(L) ||u_{syy} y^2||_{L^\infty} ||v_y||_{L^2}||u_x||_{L^2},
\end{align}

all of which are acceptable contributions. Note that we have used estimate (\ref{Lemma.Unif.1}) to absorb $y^2$ into $u_{syy}$. Next:
\begin{align} \n
\int u_{sxy}u \cdot \p_y \{ u y^2 w \} &= \int u_{sxy}u u_y y^2 w + \int u_{sxy} u^2 2y w \\
& \le_{|\cdot|} ||u_{sxy}y||_{L^\infty} \bigO(L) \Big[ ||u_x||_{L^2}^2 + ||u_y y||_{L^2} \Big],
\end{align}

which is an acceptable contribution by taking $L << 1$. Next, we move to the terms from $\p_y S_v$ according to (\ref{pySv}), starting with: 
\begin{align} \n
-\int \eps u_{sy}v_x \p_x \{u w y^2 \} &= -\int \eps u_{sy} v_x u_x w y^2 + \int \eps u_{sy} v_x u y^2 \\ \n
&  \le ||u_{sy}y||_{L^\infty} ||\sqrt{\eps} u_x y||_{L^2} ||\sqrt{\eps}v_x||_{L^2} \\
& \le ||u_{sy}y||_{L^\infty} \Big[ \delta ||\sqrt{\eps} u_x y||_{L^2}^2 +  N_\delta ||\sqrt{\eps}v_x||_{L^2}^2  \Big],
\end{align}

which is seen to be an acceptable contribution according to (\ref{summar.1}). Next: 
\begin{align}
&-\int \eps v_{sxy}u \Big[ u_x w y^2 - uy^2 \Big] \le \bigO(L) ||v_{sxy}||_{L^\infty} ||\sqrt{\eps}u_x y||_{L^2}^2, \\ 
&-\int \eps v_{sx}u_y \Big[ u_x w y^2 - uy^2 \Big] \le ||\sqrt{\eps} v_{sx} Y||_{L^\infty} ||u_y y||_{L^2} ||u_x||_{L^2}, \\ \n
& \hspace{45 mm} \lesssim  ||\sqrt{\eps} v_{sx} Y||_{L^\infty} \Big[ \delta ||u_{y}y||_{L^2}^2 + N_\delta ||u_x||_{L^2}^2 \Big], \\
&-\int 2\eps v_{sy}v_y \Big[u_x w y^2 - uy^2 \Big] \lesssim ||v_{sy} Y^2||_{L^\infty} ||u_x||_{L^2}^2, \\
&-\int \eps v_{syy} v \Big[ u_x w y^2 - u y^2 \Big] \le \bigO(L) ||v_{syy}y||_{L^\infty} \Big[ ||\sqrt{\eps}v_x||_{L^2}^2 + ||\sqrt{\eps}u_x y||_{L^2}^2 \Big].
\end{align}

\vspace{2 mm}

\textit{Step 3: Vorticity Terms}

\vspace{2 mm}

We record the following identities: 
\begin{align} \label{id.1}
&-\Delta_\eps u_y = -u_{yyy} - 2\eps u_{xxy} - \eps v_{xyy}, \\ \label{id.2}
&-\Delta_\eps v_y = - 2v_{yyy} - \eps \p_x \{u_{yy} + \eps v_{xy} \}.
\end{align}

In the forthcoming calculations, we provide estimates on the vorticity terms: 
\begin{align} \label{lap.lap.1}
&-\int \Delta_\eps u_y \cdot \p_y \{u y^2 w \} = \int \Big\{ - u_{yyy} - 2 \eps u_{xxy} - \eps v_{xyy} \Big\} \cdot \p_y \{u  y^2 w \}, \\ \label{lap.lap.2}
&+\int \Delta_\eps v_y \cdot \eps \p_x \{u y^2 w \} = \int \Big\{ 2 v_{yyy} + \eps \p_x \{u_{yy} + \eps v_{xy} \} \Big\} \cdot \eps \p_x \{u y^2 w \}.
\end{align}

Starting with the first term from (\ref{lap.lap.1}):
\begin{align} \label{sing.1}
-\int u_{yyy} \cdot \p_y \{u w y^2 \} &= +\int u_{yy} \p_y^2 \{ u y^2 w \} \\ \n
& = \int u_{yy}^2 y^2 w + \int 4 u_{yy} u_y y w + \int 2 u_{yy} u w \\  \n
&  = + \int u_{yy}^2 y^2 w - 4 \int u_y^2 w \\ \label{vort.2.a}
& \gtrsim + \int u_{yy}^2 y^2  - \bigO(\text{RHS}).
\end{align}

We must provide the rigorous justification of the integration found in (\ref{sing.1}). The delicate calculation occurs near $x = L, y = 0$ corner, for which we use the regularity theory in \cite{OS}, which yields the asymptotic behavior: \footnote{One applies Theorem 4.1 in \cite{OS} with $\beta = 1+\delta, q = q_1 = 2, h = -\delta$ and $h_1 = \frac{1}{2}+$ to obtain $\beta_1 = \frac{1}{2}-$. Theorem 4.1 gives $||r^{-\frac{3}{2}}u, r^{-\frac{1}{2}}Du, r^{\frac{1}{2}}D^2 u||_{L^2} < \infty$. One can then bootstrap this regularity to obtain $||r^{\frac{1}{2}+k} D^{2+k}u||_{L^2} < \infty$. Standard Sobolev embedding arguments give the pointwise asymptotics in (\ref{sing.asy}). } 
\begin{align} \label{sing.asy}
|u| \lesssim r^{\frac{1}{2}}, \hspace{3 mm} |u_x, u_y| \lesssim r^{-\frac{1}{2}}, \hspace{3 mm}  |D^2 u| \lesssim r^{-\frac{3}{2}}, 
\end{align}

where $r$ is the distance to the corner. Defining $C_r$ to be a solid ball of radius $r$ around the corner, we have: 
\begin{align} \label{sing.2}
-\int u_{yyy} \cdot \p_y \{u w y^2 \} = - \int_{\Omega - C_r} u_{yyy} \cdot \p_y \{u w y^2 \} - \int_{C_r} u_{yyy} \cdot \p_y \{u w y^2 \}.
\end{align} 

First, the expansions in \cite{OS} show that $u_{yyy}r^{\frac{3}{2}} \in L^2$. Therefore, taking limit as $r \rightarrow 0$, the latter term in (\ref{sing.2}) vanishes, and it remains to treat the former term: 
\begin{align}
 - \int_{\Omega - C_r} u_{yyy} \cdot \p_y \{u w y^2 \} = +\int_{\Omega - C_r} u_{yy} \cdot \p_{yy} \{u w y^2 \} - \int_{\p C_r} u_{yy} \cdot \p_y \{u w y^2 \} \ud S. 
\end{align}

For the surface integral, we use the expansions from (\ref{sing.asy}), and that $y \le r$:
\begin{align}
-\int_{\p C_r} u_{yy} \p_y \{u wy^2 \} \ud S \le \int_{\p C_r} r^{-\frac{3}{2}} r^{-\frac{1}{2}} y^2 \le \int_{\p C_r} 1 \rightarrow 0.
\end{align}

We now move to the second term from (\ref{lap.lap.1}):
\begin{align} \n
-\int 2 \eps u_{xxy} \cdot \p_y \{ u y^2 w \} &= +\int 2\eps u_{xy} \p_{xy} \{ u y^2 w \} - \int_{x = L} 2\eps u_{xy} \p_y \{u y^2 w \} \\ \n 
& = +\int 2 \eps u_{xy}^2 y^2 + \int 4 \eps u_{xy}u_x y w \\  \n
& - \int 2 \eps u_{xy}u_y y^2 - \int 4\eps u_{xy} u y - \int_{x = L} 2 \eps u_{xy} \p_y \{u y^2w \} \\ \n
& \gtrsim \int 2\eps u_{xy}^2 y^2 - \int_{x = L} 2\eps u_{xy} \p_y \{ u y^2 w \} \\ \label{vort.2.b}
& \hspace{30  mm} - \bigO(\text{RHS}) - \eps \bigO(\text{LHS}),
\end{align}

where we have used the following estimates: 
\begin{align}
&\int 4\eps u_{xy}u_x y w = -\int 2\eps u_x^2 w, \\
&-\int 2 \eps u_{xy}u_y y^2 = -\int_{x = L} \eps u_y^2 y^2 \le \eps \cdot (\ref{summar.1}) \le \eps \bigO(\text{LHS}), \\
&-\int 4 \eps u_{xy}uy = +\int 4 \eps u_x u + \int 4 \eps u_x u_y y \le \eps \Big[ ||u_x||_{L^2}^2 + ||u_y y||_{L^2}^2 \Big].
\end{align}

Next, the third term from (\ref{lap.lap.1}):
\begin{align} \n
-\int \eps v_{xyy} \p_y \{u y^2 w \} &= +\int \eps v_{xy} \p_{yy} \{ u y^2 w \} \\ \n
& = +\int \eps v_{xy} \Big[ u_{yy} y^2 w + 2 uw + 4 u_y y w \Big] \\ \n
& = - \int \eps u_{xx}u_{yy} y^2 w - \int 2\eps v_x u_y w - \int 4 \eps u_{xx}u_y y w \\ \n
& = -\int \eps u_{xx} u_{yy} y^2 w - \int 2\eps v_x u_y w + \int 4 \eps u_x u_{xy} y w \\ \n
& - \int_{x = L} 4\eps u_x u_y y w \\ \label{vort.2.c}
& = -\int \eps u_{xx}u_{yy}y^2 w + \bigO(\text{RHS}) + \eps \bigO(\text{LHS}).
\end{align}

where we have estimated: 
\begin{align}
&|\int_{x = L} 4 \eps u_x u_y yw| \lesssim \sqrt{\eps} ||\sqrt{\eps}u_x||_{L^2(x = L)} ||u_y y||_{L^2(x = L)}, \\
&|\int 2\eps v_x u_y w| \le \sqrt{\eps} ||u_y||_{L^2} ||\sqrt{\eps}v_x||_{L^2}, \\
&\int 4\eps u_{xy}u_x y w = - \int 2 \eps u_x^2 w.
\end{align}

We now come to the first term from (\ref{lap.lap.2}):
\begin{align} \n
\int \eps \p_x \{u_{yy} &+ \eps v_{xy} \} \cdot \p_x \{u y^2 w \} \\ \n
& = - \int \{u_{yy} + \eps v_{xy} \} \cdot \p_{xx} \{u y^2 w \}  +\int_{x = L} \eps \{u_{yy} + \eps v_{xy} \} \cdot \p_x \{u y^2 w \} \\ \n
& = - \int \{u_{yy} + \eps v_{xy} \} \cdot \Big[ u_{xx}y^2w - 2u_x y^2 \Big] + \int_{x= L} \eps \p_y b_L \cdot \p_x \{u y^2 w \} \\ \n
& = - \int \eps u_{yy} u_{xx} w y^2 + \int \eps^2 u_{xx}^2 y^2 w + \int 2 \eps u_{yy} u_x y^2 + \int 2 \eps^2 v_{xy} u_x y^2 \\ \n 
& \hspace{30 mm} + \int_{x= L} \eps \p_y b_L \cdot \p_x \{u y^2 w \} \\ \label{vort.2.d}
& = -\int \eps u_{yy}u_{xx} w y^2 + \int \eps^2 u_{xx}^2 y^2 w + \eps \bigO(\text{RHS}) + \eps \bigO(\text{LHS}) \\ \n
& \hspace{30 mm}  + ||\p_y b_L \langle y \rangle^2||_{L^2(x = L)}^2 + \eps ||u, \sqrt{\eps}u_x||_{L^2(x = L)}^2.
\end{align}

We have estimated: 
\begin{align} \n
&+\int 2 \eps u_{yy} u_x y^2 = -\int 2 \eps u_y u_{xy} y^2 - \int 4 \eps u_{y} u_x y \\ \n
&\hspace{20 mm} = - \int_{x = L} \eps u_y^2 y^2 - \int 4\eps u_y u_x y, \\ \n
&\hspace{20 mm} \le \eps ||u_y y||_{L^2(x= L)}^2 +  \eps ||u_y y||_{L^2} ||u_x||_{L^2} \\
& \hspace{20 mm} \lesssim \eps \Big[ \bigO(\text{LHS}) + \bigO(\text{RHS})\Big], \\
&+\int 2 \eps^2 v_{xy} u_x y^2 = - \int 2\eps^2 u_{xx} u_x y^2 = -\int_{x = L} \eps^2 u_x^2 y^2 \le \eps \bigO(\text{LHS}).
\end{align}

Next from (\ref{id.2}):
\begin{align} \n
+\int 2 v_{yyy} \cdot \eps \p_x \{u w y^2 \} &= - \int 2 \eps v_{yyy} v_y w y^2 - \int 2 \eps v_{yyy} u y^2 \\ \n
& = +\int 2 \eps v_{yy}^2 y^2 + \int 4 \eps v_{yy}v_y w y + \int 2 \eps v_{yy} u_y y^2 + \int 4 \eps v_{yy}uy \\ \n
& = \int 2\eps v_{yy}^2 y^2 w - \int 2 \eps v_y^2 w - \int 2\eps u_{xy}u_y y^2 \\  \n
& \hspace{20 mm} - \int 4\eps v_y u_y y - \int 4 \eps v_y u \\ \label{vort.2.e}
& \gtrsim \int 2 \eps v_{yy}^2 y^2 w - \eps \Big[ \bigO(\text{LHS}) +\bigO(\text{RHS})   \Big],
\end{align}

where we have estimated the following terms: 
\begin{align}
&-\int 2 \eps u_{xy} u_y y^2 = - \int_{x = L} \eps u_y^2 y^2, \\
&|\int 4 \eps v_y u_y y| \le \eps ||u_y y||_{L^2} ||v_y||_{L^2}, \\
&|\int 4 \eps v_y u| \le \eps \bigO(L) ||u_x||_{L^2}^2.
\end{align}

We can now collect the estimates from  (\ref{vort.2.a}),  (\ref{vort.2.b}),  (\ref{vort.2.c}), (\ref{vort.2.d}), (\ref{vort.2.e}) to get: 
\begin{align} \n
- \int \Delta_\eps u_y &\cdot \p_y \{u y^2 w \} + \int \eps v_y \cdot \p_x \{u y^2 w\} \\ \n
& \gtrsim \int \Big[ u_{yy}^2 + 4 \eps u_{xy}^2 + \eps^2 u_{xx}^2 - 2\eps u_{xx}u_{yy} \Big] y^2 w - \int_{x = L} 2\eps u_{xy} \p_y \{u y^2 w \} \\ \label{summer.3}
& - \bigO(\text{RHS}) - \eps \bigO(\text{LHS}).
\end{align}

We now have the Pressure contributions: 
\begin{align} \label{pressure.1}
\int P_{yx} \cdot \p_y \{u y^2 w \} + \int P_{yy} \cdot \p_x \{u y^2 w \} = \int_{x = L} P_y \cdot \p_y \{ u y^2 w \}.
\end{align}

Using $P_y - 2\eps u_{xy} = a_L(y)$ on $x = L$, the boundary term above can be combined with that in (\ref{summer.3}) yielding:
\begin{align} \n
\int_{x = L}\Big( P_y - 2\eps u_{xy} \Big)& \cdot \p_y \{u y^2 w \} = \int_{x = L} \p_y a_L \cdot \p_y \{ u y^2 w \} \\
& \le ||\p_y a_L \cdot \langle y \rangle^2 ||_{L^2(x = L)} ||u_y y, u||_{L^2(x = L)} \\
& \lesssim N_\delta ||\p_y a_L \cdot \langle y \rangle^2 ||_{L^2(x = L)}^2 + \delta ||u_y y, u||_{L^2(x = L)}^2,
\end{align}

the latter of which can be absorbed into the left-hand side of our estimate, specifically the positive contribution of (\ref{summar.1}). Thus, summing (\ref{pressure.1}) with (\ref{summer.3}) yields: 
\begin{align} \n
&- \int \Delta_\eps u_y \cdot \p_y \{u y^2 w \} + \int \eps v_y \cdot \p_x \{u y^2 w\} \\ \n
& \hspace{30 mm} + \int P_{yx} \cdot \p_y \{ u y^2 w \} + \int P_{yy} \cdot \p_x \{u y^2 w \} \\ \n
& \gtrsim \int \Big\{u_{yy}^2 + 4\eps u_{xy}^2 + \eps^2 u_{xx}^2 - 2\eps u_{xx}u_{yy} \Big\} y^2 w - \bigO(\text{RHS}) - \eps \bigO(\text{LHS}) \\
& \gtrsim \int \Big\{u_{yy}^2 + \eps u_{xy}^2 + \eps^2 u_{xx}^2  \Big\} y^2,
\end{align}

where the final inequality follows from (\ref{korn.thm.1}). This concludes the proof.

\end{proof}

\subsection{The Korn's Inequality}

\begin{lemma} For any functions $[u,v] \in \X$, the following estimate is valid: 
\begin{align} \n
&\int \Big[ u_{yy}^2  + 4\eps u_{xy}^2 + \eps^2 u_{xx}^2 - 2\eps u_{yy} u_{xx}\Big] y^2  w(x) \ud x \ud y \\ \n
& \hspace{10 mm}  \gtrsim \int  \Big[ u_{yy}^2  + \eps u_{xy}^2 + \eps^2 u_{xx}^2 \Big] y^2  w(x) \ud x \ud y \\ \label{korn.thm.1}
& \hspace{10 mm} - \eps ||u_y y, \sqrt{\eps}u_x y||_{L^2}^2 - ||u_y, u_x||_{L^2}^2.
\end{align}
\end{lemma}

\begin{proof} 
We would like to apply the Korn inequality to generate positive terms: 
\begin{align} \label{korn.1}
&\int \Big[ u_{yy}^2  + 4\eps u_{xy}^2 + \eps^2 u_{xx}^2 - 2\eps u_{yy} u_{xx}\Big] y^2  w(x) \ud x \ud y.
\end{align}

We will first rescale to original Eulerian coordinates, so as to ensure all estimates are independent of $L$: 
\begin{align}
X = \frac{x}{L}, \hspace{3 mm} Y = \frac{\sqrt{\eps}}{L}y, \hspace{3 mm}  U(X,Y) = u(x,y), \hspace{3 mm} V(X,Y) = \sqrt{\eps}v(x,y).
\end{align}

Define also $w_L(X) = 1 - LX$. This gives the following relations: 
\begin{align}
&U_X = Lu_x, U_Y = \frac{L}{\sqrt{\eps}}u_y, U_{YY} = \frac{L^2}{\eps}u_{yy}, \\
&V_X = \sqrt{\eps}L v_x, V_Y = L v_y, V_{XX} = L^2 \sqrt{\eps}v_{xx}, V_{YY} = \frac{L^2}{\sqrt{\eps}}v_{yy}.
\end{align}

It is clear that:
\begin{align}
(\ref{korn.1}) = \sqrt{\eps} \int \Big[ U_{YY}^2 + 4 U_{XY}^2 + U_{XX}^2 - 2U_{YY}U_{XX}  \Big] Y^2 w_L(X) \ud X \ud Y.
\end{align}

We will define: 
\begin{align}
U^{(1)} := U_Y Y \sqrt{w_L}, \hspace{5 mm} V^{(1)} := V_Y Y \sqrt{w_L}.
\end{align}
\begin{align}
&U^{(1)}_Y = U_{YY}Y \sqrt{w_L} + U_Y \sqrt{w_L}, \\
&U^{(1)}_X = U_{XY} Y \sqrt{w_L} + U_Y Y \frac{1}{2\sqrt{w_L}} L, \\
&V^{(1)}_Y = V_{YY} Y \sqrt{w_L} + V_Y \sqrt{w_L} \\
& V^{(1)}_X = V_{XY} Y \sqrt{w_L} + V_Y Y \frac{1}{2\sqrt{w_L}} L.
\end{align}

We will now calculate: 
\begin{align}
&\sqrt{\eps} \int U_Y^2 w_L \ud X \ud Y = \int u_y^2 w \ud x \ud y, \\
& \sqrt{\eps} \int U_Y^2 Y^2 \frac{L^2}{w_L} \ud X \ud Y =  \eps \int u_y^2 y^2 \ud x \ud y, \\
& \sqrt{\eps} \int V_Y^2 w_L \ud X \ud Y = \int \eps v_y^2 w \ud x \ud y, \\
& \sqrt{\eps} \int V_Y^2 Y^2 \frac{L^2}{w_L} \ud X \ud Y = \eps^2 \int v_y^2 y^2 \ud x \ud y, \\ \label{l.o.1}
& \sqrt{\eps} \int |U^{(1)}|^2  \ud X \ud Y = \frac{\eps}{L^2} \int u_y^2 y^2 w \ud x \ud y, \\ \label{l.o.2}
& \sqrt{\eps} \int |V^{(1)}|^2 \ud X \ud Y = \frac{\eps^2}{L^2} \int v_y^2 y^2 w \ud x \ud y,
\end{align}

Thus: 
\begin{align*}
&\sqrt{\eps} \int U_{YY}^2 Y^2 w_L = \sqrt{\eps}\int |U^{(1)}_Y|^2 + \mathcal{C}, \\
&\sqrt{\eps} \int U_{XY}^2 Y^2 w_L = \sqrt{\eps} \int 4 |U^{(1)}_X|^2 + \mathcal{C}, \\
&\sqrt{\eps} \int U_{XX}^2 Y^2 w_L =\sqrt{\eps}  \int |V^{(1)}_X|^2 + \mathcal{C}, \\
&-2\sqrt{\eps} \int U_{YY}U_{XX} Y^2 w_L = -\sqrt{\eps} \int 2 U^{(1)}_Y V^{(1)}_X + \mathcal{C},
\end{align*}

where: 
\begin{align} \n
|\mathcal{C}| \lesssim &N_\delta \cdot ||u_y, u_x||_{L^2} + \eps \cdot||u_y y, \sqrt{\eps}u_x y||_{L^2}^2 \\
& + \delta \sqrt{\eps} \int |U^{(1)}_Y|^2 +| V^{(1)}_X|^2 + |U^{(1)}_X|^2.
\end{align}

According to this, we can write: 
\begin{align}
(\ref{korn.1}) \gtrsim \sqrt{\eps} \Big[ \int |U^{(1)}_Y|^2 + 4 |U^{(1)}_X|^2 + |V^{(1)}_X|^2 - 2 U^{(1)}_Y V^{(1)}_X \Big] - |\mathcal{C}|. 
\end{align}

By adding and subtracting (\ref{l.o.1}) - (\ref{l.o.2}) and up to redefining $\mathcal{C}$, we have: 
\begin{align*}
(\ref{korn.1}) \gtrsim &\sqrt{\eps} \int \Big[ |U^{(1)}_Y|^2 + 4 |U^{(1)}_X|^2 + |V^{(1)}_X|^2 - 2U^{(1)}_Y V^{(1)}_X \Big] \ud X \ud Y \\
& + \sqrt{\eps} \int \Big[ |U^{(1)}|^2 + |V^{(1)}|^2 \Big]\ud X \ud Y - |\mathcal{C}|.
\end{align*}

An application of Korn's inequality yields: 
\begin{align} \n
(\ref{korn.1}) &\gtrsim \sqrt{\eps} \int |U^{(1)}_Y|^2 + |U^{(1)}_X|^2 + |V^{(1)}_X|^2  \\ \n
& \hspace{30 mm} -  ||u_y, u_x||_{L^2}^2 + \eps ||u_y y, \sqrt{\eps}u_x y||_{L^2}^2 \\ \n
& \gtrsim \int \Big\{ u_{yy}^2 + \eps u_{xy}^2 + \eps^2 u_{xx}^2 \Big\} y^2 w(x) \ud y \ud x \\
& \hspace{30 mm} - ||u_y, u_x||_{L^2}^2 + \eps ||u_y y, \sqrt{\eps}u_x y||_{L^2}^2.
\end{align}

This concludes the proof. 

\end{proof}

\subsection{Summary of $L^2$ Estimates:}

Let us now consolidate the $L^2$-based estimates, by combining (\ref{en.est.1}), (\ref{pos.est.1}), (\ref{weight.est.1}). First, we will define the following $L^2$ based norm: 
\begin{align} \n
||u,v||_{X_1} &:= ||u_y \cdot y||_{L^2} + ||\sqrt{\eps}u_x \cdot y ||_{L^2} + ||v_y, \sqrt{\eps}v_x||_{L^2} \\ \label{norm.X}
& + ||\Big\{u_{yy}, \sqrt{\eps}u_{xy}, \eps u_{xx} \Big\} \cdot y ||_{L^2}. 
\end{align}

Recalling the boundary norm given in (\ref{norm.B}), accumulating estimates (\ref{weight.est.1}), (\ref{en.est.1}), and (\ref{pos.est.1}), and taking $0 < L << ||\frac{v^0_e}{Y}||_{L^\infty} << 1$ gives: 
\begin{align} \label{lin.complete}
||u,v||_{X_1}^2 + ||u,v||_B^2 \lesssim \mathcal{R}_1 + \mathcal{R}_2 + \mathcal{R}_3. 
\end{align}

\section{Uniform Estimates}

We will now obtain $L^\infty$ estimates for solutions $[u,v]$ to the system (\ref{lin.1}) - (\ref{lin.3}), which are based on bootstrapping estimates that are valid for the Stokes operator. 

\begin{lemma} Solutions $[u,v] \in \mathcal{X}$ to the system (\ref{lin.1}) - (\ref{lin.3}), with the boundary conditions (\ref{BC.dirichlet}) - (\ref{BC.stress.free}) satisfy the following uniform estimate: 
\begin{align} \n
\eps^{\frac{\gamma}{4}} ||u, \sqrt{\eps}v||_{L^\infty(\bar{\Omega})} \lesssim &C(\gamma, L) \Big\{ ||u, \sqrt{\eps}v||_{H^1}  + C(a_L, b_L)  \\  \label{unifo.1}
& + ||S_u, \sqrt{\eps}S_v||_{L^2} + ||f, \sqrt{\eps}g||_{L^2} \Big\}.
\end{align}
\end{lemma}
\begin{proof}

The proof follows from \cite{GN}, Lemma 4.1. Note that the estimate up to the boundary, $L^\infty(\bar{\Omega})$, is guaranteed according to \cite{Adams}, P. 98, Equation 9, as our domain $\Omega$ satisfies the strong local Lipschitz property, as defined by \cite{Adams}, P. 66.

\end{proof}

We emphasize that for our analysis, it is important to obtain the uniform control on the boundary $x = L$, due for instance, to the nonlinear contributions from (\ref{NL.tough}). We now relate the right-hand side above to our norms. 

\begin{lemma} For any functions, $[u,v] \in \mathcal{X}$, the following estimate holds: 
\begin{align} \n
||u, \sqrt{\eps}v||_{H^1} &+ ||S_u, \sqrt{\eps}S_v||_{L^2}  + ||\eps^{-\frac{1}{2}-\gamma}\{R^u, \sqrt{\eps}R^v\}||_{L^2} + ||L^b_1, \sqrt{\eps}L^b_2||_{L^2} \\ \label{unifo.2} &+ ||N^u(\bar{u}, \bar{v}), \sqrt{\eps}N^v(\bar{u},\bar{v})||_{L^2}  \lesssim 1 + ||u,v||_{X_1}  + ||\bar{u}, \bar{v}||_{\X}^2.
\end{align}
\end{lemma}
\begin{proof}

The estimates on $||u, \sqrt{\eps}v||_{H^1}, || \eps^{-\frac{1}{2}-\gamma}\{ R^u, \sqrt{\eps}R^v\}||_{L^2}$ follow trivially, the latter from (\ref{sigma}). The estimates on $||L^b_1, \sqrt{\eps} L^b_2||_{L^2}$, as defined in (\ref{defn.Lb.1}) - (\ref{defn.Lb}), follow from (\ref{est.Lb}). Next, referring to the definition of $S_u$ in (\ref{nl.spec.1}), and the estimates in (\ref{Lemma.Unif.1}), 
\begin{align} \n
||S_u||_{L^2} &= ||u_s u_x + u_{sx}u + v_s u_y + u_{sy}v||_{L^2} \\
& \le ||u_s, u_{sx}, \frac{v^0_e}{Y}, v_s - \frac{v^0_e}{\sqrt{\eps}}, u_{sy}y||_{L^2} ||u_x, u_y y||_{L^2}.
\end{align}

Similarly, referring to the definition of $S_v$ given in (\ref{nl.spec.1}) and the estimates (\ref{Lemma.Unif.1}):
\begin{align} \n
||\sqrt{\eps}S_v||_{L^2} &\le ||\sqrt{\eps} \Big( u_s v_x + v_{sx}u + v_s v_y + v_{sy}v \Big)||_{L^2} \\
& \le ||u_s, \sqrt{\eps}v_{sx}, \sqrt{\eps}v_s, v_{sy}||_{L^2} ||\sqrt{\eps}v_x||_{L^2}.
\end{align}

Referring to the definitions of the nonlinearities given in (\ref{nl.spec.2}):
\begin{align} \n
||N^u(\bar{u}, \bar{v})||_{L^2} &= \eps^{\frac{1}{2}+\gamma} ||\bar{u} \bar{u}_x + \bar{v} \bar{u}_y||_{L^2} \\
& \le \eps^{\frac{\gamma}{2}} ||\eps^{\frac{\gamma}{2}} \Big\{ \bar{u}, \sqrt{\eps} \bar{v} \Big\}||_{L^\infty} ||\bar{u}_x, \bar{u}_y||_{L^2}, \\ \n
||\sqrt{\eps} N^v(\bar{u}, \bar{v}) ||_{L^2} &\le ||\eps^{1+\gamma} \Big(\bar{u} \bar{v}_x + \bar{v} \bar{v}_y \Big) ||_{L^2} \\
& \le \eps^{\frac{1}{2}+\frac{\gamma}{2}}||\eps^{\frac{\gamma}{2}}\bar{u}||_{L^\infty} ||\sqrt{\eps} \bar{v}_x||_{L^2} + \eps^{\frac{1}{2}+\frac{\gamma}{2}} ||\eps^{\frac{1}{2}+\frac{\gamma}{2}}\bar{v}||_{L^\infty} || \bar{v}_y||_{L^2}.
\end{align}

The above estimates imply the result.
\end{proof}

Combining (\ref{unifo.1}), (\ref{unifo.2}) with the definition of $(f, g)$ given in (\ref{nl.spec.2}) - (\ref{nl.spec.3}), together with relevant definitions in (\ref{defn.Ru1}), (\ref{defn.Rv1}), and (\ref{defn.Lb}) gives the following: 

\begin{corollary} Solutions $[u,v] \in \mathcal{X}$ to the system (\ref{lin.1}) - (\ref{lin.3}), with the boundary conditions (\ref{BC.dirichlet}) - (\ref{BC.stress.free}) satisfy the following uniform estimate: 
\begin{align} \label{unif.cor}
\eps^{\frac{\gamma}{2}}||u, \sqrt{\eps}v||_{L^\infty} \lesssim \eps^{\frac{\gamma}{4}} + \eps^{\frac{\gamma}{4}} \Big[||u,v||_{X_1}+  ||\bar{u}, \bar{v}||_{\X}^2 \Big].
\end{align}
\end{corollary}

Combining with (\ref{lin.complete}), we have now controlled the full $\X$ norm:
\begin{corollary} Solutions $[u,v] \in \mathcal{X}$ to the system (\ref{lin.1}) - (\ref{lin.3}), with the boundary conditions (\ref{BC.dirichlet}) - (\ref{BC.stress.free}) satisfy the following estimate:
\begin{align}
||u,v||_{\X}^2 \lesssim \eps^{\frac{\gamma}{2}} + \mathcal{R}_1 +  \mathcal{R}_2 +  \mathcal{R}_3 + \eps^{\frac{\gamma}{2}} ||\bar{u}, \bar{v}||_{\X}^4. 
\end{align}
\end{corollary}

It remains to control $\mathcal{R}_i$, which we now expand by recalling (\ref{R.1}), (\ref{R.2}), (\ref{R.3}), and (\ref{nl.spec.3}): 
\begin{align} \n
\mathcal{R}_1 = &\int \Big[ \eps^{-\frac{1}{2}-\gamma} R^{u,1} + N^u + L^b_1 \Big] \cdot u \\ \label{R.1.e}
& + \int \eps \Big[ \eps^{-\frac{1}{2}-\gamma} R^{v,1} + N^v + L^b_2 \Big] v, \\ \n
\mathcal{R}_2  = &- \int \Big[ \eps^{-\frac{1}{2}-\gamma} R^{u,1} + N^u + L^b_1 \Big] \cdot \beta_y \\ \label{R.2.e}
& + \int \eps \Big[ \eps^{-\frac{1}{2}-\gamma} R^{v,1} + N^v + L^b_2 \Big] \cdot \beta_x \\ \n
\mathcal{R}_3 = &\int \Big[ \eps^{-\frac{1}{2}-\gamma} \p_y R^{u,1} + \p_y N^u + \p_y L^b_1 \Big] \cdot \p_y \{u y^2 w \} \\ \label{R.3.e} 
& - \int \eps \Big[ \eps^{-\frac{1}{2}-\gamma} \p_y R^{v,1} + \p_y N^v + \p_y L^b_2 \Big] \cdot \p_x \{u y^2 w \}.
\end{align}

We now turn to controlling these quantities. 

\section{Nonlinearities}

We now provide estimates on $\mathcal{R}_1, \mathcal{R}_2, \mathcal{R}_3$, as displayed in (\ref{R.1.e}) - (\ref{R.3.e}). We will first estimate the nonlinear terms, $N^u, N^v$, which are in turn defined in (\ref{nl.spec.2}). Because we will eventually perform a contraction mapping argument, we will consider $N^u(\bar{u}, \bar{v}), N^v(\bar{u}, \bar{v})$, where $[\bar{u}, \bar{v}] \in \X$. We have: 
\begin{align}  \label{ken.1}
&\p_y N^u(\bar{u},\bar{v}) = \eps^{\frac{1}{2}+\gamma} \Big\{ \bar{u} \bar{u}_{xy} + \bar{v} \bar{u}_{yy} \Big\}, \\ \label{ken.2}
&\p_y N^v(\bar{u},\bar{v})  = \eps^{\frac{1}{2}+\gamma}\Big\{ \bar{u}_y \bar{v}_x + \bar{u}\bar{v}_{xy} + \bar{v}_y^2 + \bar{v} \bar{v}_{yy} \Big\}.
\end{align}

The first step is to provide estimates on the nonlinear contributions from $\mathcal{R}_3$, as defined in (\ref{R.3}). For this we have:

\begin{lemma} For any vector fields $[u,v], [\bar{u}, \bar{v}] \in \X$, the following estimate holds: 
\begin{align} \label{nl.R.1}
|\int \p_y N^u(\bar{u},\bar{v}) \cdot \p_y\{u w y^2 \}| + |\int \p_y N^v(\bar{u}, \bar{v}) \cdot \eps \p_x \{u w  y^2 \}|  \lesssim \eps^{\frac{\gamma}{2}} ||\bar{u}, \bar{v}||_{\X}^2 ||u, v||_{\X}.
\end{align}
\end{lemma}

\begin{proof}

Turning to the first term from (\ref{ken.1}), we will expand via the product rule: 
\begin{align} \n
\int \eps^{\frac{1}{2}+\gamma} &\bar{u} \bar{u}_{xy} \cdot \p_y \{ u y^2 w \} \\ \n
&= \int \eps^{\frac{1}{2}+\gamma} \bar{u} \bar{u}_{xy} \cdot u_y y^2 w +  \int \eps^{\frac{1}{2}+\gamma} \bar{u} \bar{u}_{xy} \cdot u 2y w \\ \n
& = \int \eps^{\frac{1}{2}+\gamma} \bar{u} \bar{u}_{xy} u_y y^2 w - \int \eps^{\frac{1}{2}+\gamma} \bar{u}_y \bar{u}_x u 2y w \\ \n
& \hspace{20 mm} - \int \eps^{\frac{1}{2}+\gamma} \bar{u} \bar{u}_x u_y 2yw - \int \eps^{\frac{1}{2}+\gamma} \bar{u} u \bar{u_x} 2 w \\ \n
& \le  \eps^{\frac{\gamma}{2}}||\eps^{\frac{\gamma}{2}} \bar{u}||_{L^\infty} || \sqrt{\eps} \bar{u}_{xy} y||_{L^2} ||u_y y||_{L^2} \\ \n
& \hspace{20 mm} + \eps^{\frac{1}{2}+\frac{\gamma}{2}} ||\eps^{\frac{\gamma}{2}}u||_{L^\infty} ||\bar{u}_y y||_{L^2} ||\bar{u}_x||_{L^2} \\ \n
& \hspace{20 mm} + \eps^{\frac{1}{2}+\frac{\gamma}{2}} ||\eps^{\frac{\gamma}{2}} \bar{u}||_{L^\infty} ||u_y y||_{L^2} ||\bar{u}_x||_{L^2} \\
& \hspace{20 mm} + \eps^{\frac{1}{2}+\frac{\gamma}{2}}||\eps^{\frac{\gamma}{2}}u||_{L^\infty} ||\bar{u}_x||_{L^2}^2.
\end{align}

Turning to the second term from (\ref{ken.1}):
\begin{align} \n
\int \eps^{\frac{1}{2}+\gamma} &\bar{v} \bar{u}_{yy} \cdot \p_y \{ u y^2 w \} \\ \n
& = \int \eps^{\frac{1}{2}+\gamma} \bar{v} \bar{u}_{yy} u_y y^2 w + \int \eps^{\frac{1}{2}+\gamma} \bar{v} \bar{u}_{yy} u 2y w \\
& \le \eps^{\frac{\gamma}{2}} ||\eps^{\frac{1}{2}+\frac{\gamma}{2}}\bar{v}||_{L^\infty} ||\bar{u}_{yy}y||_{L^2} \Big[ ||u_y y||_{L^2} + \bigO(L) ||u_x||_{L^2} \Big].
\end{align}

We will now turn to the first term from (\ref{ken.2}), which is the most delicate because $\bar{v}_x$ cannot accept any weights of $y$, according to our norm $\X$, (\ref{norm.S}). As a result, we must rely on an integration by parts in $x$: 
\begin{align} \n
\int \eps^{\frac{3}{2}+\gamma} \bar{u}_y \bar{v}_x \p_x \{ u y^2 w \} &= \int \eps^{\frac{3}{2}+\gamma} \bar{u}_y \bar{v}_x u_x y^2 w - \int \eps^{\frac{3}{2}+\gamma} \bar{u}_y \bar{v}_x u y^2 \\ \n
& = - \int \eps^{\frac{3}{2}+\gamma} \bar{v} \bar{u}_{xy} u_x y^2 w - \int \eps^{\frac{3}{2}+\gamma} \bar{v} \bar{u}_y u_{xx} y^2 w \\ \n
& \hspace{10 mm} - \int \eps^{\frac{3}{2}+\gamma} \bar{v} \bar{u}_y u_x y^2 - \int \eps^{\frac{3}{2}+\gamma} \bar{u}_y \bar{v}_x u y^2 \\ \n
&  \hspace{10 mm} + \int_{x = L} \eps^{\frac{3}{2}+\gamma} \bar{v} \bar{u}_y u_x y^2 \\ \n
& = - \int \eps^{\frac{3}{2}+\gamma} \bar{v} \bar{u}_{xy} u_x y^2 w - \int \eps^{\frac{3}{2}+\gamma} \bar{v} \bar{u}_y u_{xx} y^2 w \\ \n
& \hspace{10 mm} - \int \eps^{\frac{3}{2}+\gamma} \bar{v} \bar{u}_y u_x y^2 + \int \eps^{\frac{3}{2}+\gamma}  \bar{v} u_x \bar{u}_y y^2 + \int \eps^{\frac{3}{2}+\gamma}  \bar{v} u \bar{u}_{xy} y^2 \\ \n
&  \hspace{10 mm}- \int_{x = L} \eps^{\frac{3}{2}+\gamma} \bar{v} u \bar{u}_y y^2  + \int_{x = L} \eps^{\frac{3}{2}+\gamma} \bar{v} \bar{u}_y u_x y^2  \\ \n
& \le \eps^{\frac{\gamma}{2}} ||\eps^{\frac{1}{2}+\frac{\gamma}{2}} \bar{v}||_{L^\infty} ||\sqrt{\eps}u_{xy} y||_{L^2} ||\sqrt{\eps}u_x y||_{L^2} \\ \n
& \hspace{10 mm} + \eps^{\frac{\gamma}{2}} ||\eps^{\frac{1}{2}+\frac{\gamma}{2}} \bar{v}||_{L^\infty} ||\bar{u}_y y||_{L^2} ||\eps u_{xx} y||_{L^2} \\ \n
&\hspace{10 mm}   + \eps^{\frac{1}{2}+\frac{\gamma}{2}} ||\eps^{\frac{1}{2}+\frac{\gamma}{2}}\bar{v}||_{L^\infty} ||\bar{u}_y y||_{L^2} ||\sqrt{\eps}u_x y||_{L^2} \\ \n
& \hspace{10 mm}  + \eps^{\frac{1}{2}+\frac{\gamma}{2}}||\eps^{\frac{1}{2}+\frac{\gamma}{2}}\bar{v}||_{L^\infty} ||\sqrt{\eps}u_x y||_{L^2} ||\bar{u}_y y||_{L^2} \\ \n
& \hspace{10 mm}  + \eps^{\frac{\gamma}{2}} ||\eps^{\frac{1}{2}+\frac{\gamma}{2}}\bar{v}||_{L^\infty} ||\sqrt{\eps}u_x y||_{L^2} ||\sqrt{\eps}u_{xy} y||_{L^2} \\ \n
& \hspace{10 mm}  + \eps^{\frac{1}{2}+\frac{\gamma}{2}} ||\eps^{\frac{1}{2}+\frac{\gamma}{2}}\bar{v}||_{L^\infty} ||\bar{u}_y y||_{L^2(x = L)} ||\sqrt{\eps}uy||_{L^2(x = L)} \\ \label{NL.tough}
& \hspace{10 mm}  +  \eps^{\frac{1}{2}+\frac{\gamma}{2}} ||\eps^{\frac{1}{2}+\frac{\gamma}{2}}\bar{v}||_{L^\infty} ||\bar{u}_y y||_{L^2(x = L)} ||\sqrt{\eps}u_x y||_{L^2(x = L)}.
\end{align}

Note that for the above term, (\ref{NL.tough}), it is imperative to obtain control of $v$ on the boundary $x = L$, as shown in estimate (\ref{unifo.1}). We now move to the second term from (\ref{ken.2}):
\begin{align} \n
\int \eps^{\frac{3}{2}+\gamma} \bar{u} \bar{v}_{xy} \p_x \{ u y^2 w \} &= \int \eps^{\frac{3}{2}+\gamma} \bar{u} \bar{v}_{xy} u_x y^2 w - \int \eps^{\frac{3}{2}+\gamma} \bar{u} \bar{v}_{xy} u y^2 \\ \n
& \le \eps^{\frac{\gamma}{2}} ||\eps^{\frac{\gamma}{2}} \bar{u}||_{L^\infty} ||\eps \bar{u}_{xx}y||_{L^2} ||\sqrt{\eps}u_x y||_{L^2} \\
& \hspace{20 mm} + \eps^{\frac{\gamma}{2}} ||\eps^{\frac{\gamma}{2}}u||_{L^\infty} ||\eps u_{xx} y||_{L^2} ||\sqrt{\eps}u_x y||_{L^2}.
\end{align}

Now we turn to the third term from (\ref{ken.2}):
\begin{align} \n
\int \eps^{\frac{3}{2}+\gamma}  \bar{v}_y^2 \p_x \{u y^2 w \} &= \int \eps^{\frac{3}{2}+\gamma}  \bar{v}_y^2 u_x y^2 w - \int \eps^{\frac{3}{2}+\gamma}  \bar{v}_y^2 u y^2 \\ \n
& = - \int \eps^{\frac{3}{2}+\gamma}  \bar{v} \bar{v}_{yy} u_x y^2 w - \int \eps^{\frac{3}{2}+\gamma}  \bar{v} \bar{v}_y u_{xy} y^2 w \\ \n
& \hspace{20 mm} - \int \eps^{\frac{3}{2}+\gamma}  \bar{v} \bar{v}_y u_x 2y w - \int \eps^{\frac{3}{2}+\gamma}  \bar{v}_y^2 u y^2 \\ \n
& \le \eps^{\frac{\gamma}{2}} ||\eps^{\frac{1}{2}+\frac{\gamma}{2}} \bar{v}||_{L^\infty} || \sqrt{\eps} \bar{v}_{yy} y||_{L^2} ||\sqrt{\eps}u_x y||_{L^2} \\ \n
& \hspace{20 mm} + \eps^{\frac{\gamma}{2}} ||\eps^{\frac{1}{2}+\frac{\gamma}{2}} \bar{v}||_{L^\infty} ||\sqrt{\eps}\bar{u}_x y||_{L^2} ||\sqrt{\eps}u_{xy}y||_{L^2} \\ \n
& \hspace{20 mm} +\eps^{\frac{1}{2}+\frac{\gamma}{2}} ||\eps^{\frac{1}{2}+\frac{\gamma}{2}} \bar{v}||_{L^\infty} ||\sqrt{\eps}\bar{u}_x y||_{L^2} ||u_x||_{L^2} \\ 
& \hspace{20 mm} + \eps^{\frac{1}{2}+\frac{\gamma}{2}} ||\eps^{\frac{\gamma}{2}}u||_{L^\infty} ||\sqrt{\eps}\bar{u}_x y||_{L^2}^2.
\end{align}

Now we turn to the fourth, final term from (\ref{ken.2}):
\begin{align} \n
\int \eps^{\frac{3}{2}+\gamma} \bar{v} \bar{v}_{yy} \p_x \{u w y^2 \} &= \int \eps^{\frac{3}{2}+\gamma} \bar{v} \bar{v}_{yy} u_x w y^2 - \int \eps^{\frac{3}{2}+\gamma} \bar{v} \bar{v}_{yy} u y^2 \\
& \lesssim \eps^{\frac{\gamma}{2}} ||\eps^{\frac{1}{2}+\frac{\gamma}{2}} \bar{v}||_{L^\infty} ||\sqrt{\eps}\bar{u}_{xy} y||_{L^2} ||\sqrt{\eps}u_x y||_{L^2}. 
\end{align}

These estimates conclude the proof of the desired result, estimate (\ref{nl.R.1}).
\end{proof}

We will now come to the nonlinear contributions to the energy estimates, which are contained in (\ref{R.1.e}):
\begin{lemma} For any vector fields $[u,v], [\bar{u}, \bar{v}] \in \X$, the following estimate holds: 
\begin{align} \label{nl.R.2}
|\int N^{u}(\bar{u}, \bar{v}) \cdot u + \eps N^v(\bar{u}, \bar{v}) \cdot v| \le \eps^{\frac{\gamma}{2}}||\bar{u}, \bar{v}||_{\X}^2 ||u,v||_{\X}.
\end{align}
\end{lemma}
\begin{proof}

We turn to the definitions of $N^u, N^v$ which are given in (\ref{nl.spec.2}). From there, the following calculations follow: 
\begin{align}
&\eps^{\frac{1}{2}+\gamma} |\int \bar{u} \bar{u}_x \cdot u| \le \eps^{\frac{1}{2}+\frac{\gamma}{2}} ||\eps^{\frac{\gamma}{2}}\bar{u}||_{L^\infty} ||\bar{u}_x||_{L^2}^2 \lesssim  \eps^{\frac{1}{2}+\frac{\gamma}{2}} ||\bar{u}, \bar{v}||_{\X}^2 ||u,v||_{\X}, \\
& \eps^{\frac{1}{2}+\gamma}|\int \bar{v} \bar{u}_y \cdot u| \le \eps^{\frac{\gamma}{2}} ||\eps^{\frac{\gamma}{2}}\sqrt{\eps}\bar{v}||_{L^\infty} ||\bar{u}_x||_{L^2}||u_y||_{L^2} \lesssim \eps^{\frac{\gamma}{2}}||\bar{u}, \bar{v}||_{\X}^2||u,v||_{\X}, \\
& \eps^{\frac{1}{2}+\gamma}|\int \bar{u} \bar{v}_x \cdot \eps v| \le \eps^{\frac{1}{2}+\frac{\gamma}{2}} ||\eps^{\frac{\gamma}{2}}\bar{u}||_{L^\infty} ||\sqrt{\eps}\bar{v}_x||_{L^2} ||\sqrt{\eps}v_x||_{L^2} \lesssim \eps^{\frac{1}{2}+\frac{\gamma}{2}}  ||\bar{u}, \bar{v}||_{\X}^2 ||u,v||_{\X} , \\
& \eps^{\frac{1}{2}+\gamma}|\int \bar{v} \bar{v}_y \cdot \eps v| \le \eps^{\frac{1}{2}+\frac{\gamma}{2}} || \eps^{\frac{1}{2}+\frac{\gamma}{2}}\bar{v} ||_{L^\infty} ||\sqrt{\eps} v_x||_{L^2}||\bar{v}_y||_{L^2} \lesssim \eps^{\frac{1}{2}+\frac{\gamma}{2}} ||\bar{u}, \bar{v}||_{\X}^2 ||u,v||_{\X}.
\end{align}

The desired result follows from these calculations. 

\end{proof}

We will now provide nonlinear estimates arising from the positivity estimate, in particular we must evaluate the contributions of the nonlinearity in (\ref{R.2.e}):
\begin{lemma} For any vector fields $[u,v], [\bar{u}, \bar{v}] \in \X$, the following estimate holds: 
\begin{align} \label{nl.R.3}
|\int N^u(\bar{u}, \bar{v}) \cdot -\beta_y |+| \int N^v(\bar{u}, \bar{v} ) \cdot \eps \beta_x | \le \eps^{\frac{\gamma}{2}}||\bar{u},\bar{v}||_{\X}^2 ||u, v||_{\X}. 
\end{align}
\end{lemma}
\begin{proof}
We again turn to the definitions of $N^u, N^v$ from (\ref{nl.spec.2}): 
\begin{align}
&\eps^{\frac{1}{2}+\gamma} |\int \bar{u} \bar{u}_x \cdot \beta_y| \le \eps^{\frac{1}{2}+\frac{\gamma}{2}} ||\eps^{\frac{\gamma}{2}}\bar{u}||_{L^\infty} ||\bar{u}_x||_{L^2} ||\beta_y||_{L^2}, \\
& \eps^{\frac{1}{2}+\gamma}|\int \bar{v} \bar{u}_y \cdot \beta_y| \le \eps^{\frac{\gamma}{2}} ||\eps^{\frac{1}{2}+\frac{\gamma}{2}}\bar{v}||_{L^\infty} ||\bar{u}_y||_{L^2}||\beta_y||_{L^2} , \\
& \eps^{\frac{1}{2}+\gamma}|\int \bar{u} \bar{v}_x \cdot \eps \beta_x| \le \eps^{\frac{\gamma}{2}} ||\eps^{\frac{\gamma}{2}}\bar{u}||_{L^\infty} ||\sqrt{\eps}\bar{v}_x||_{L^2} ||\sqrt{\eps} \beta_x||_{L^2}, \\
& \eps^{\frac{1}{2}+\gamma}|\int \bar{v} \bar{v}_y \cdot \eps \beta_x| \le \eps^{\frac{\gamma}{2}} ||\eps^{\frac{1}{2}+\frac{\gamma}{2}}\bar{v}||_{L^\infty} ||\bar{v}_y||_{L^2}||\sqrt{\eps}\beta_x||_{L^2}.
\end{align}

This concludes the proof.
\end{proof}

\section{Forcing}

Recall the definitions given in (\ref{defn.Ru1}) and (\ref{defn.Rv1}), and the definitions given in (\ref{defn.Lb.1}) - (\ref{defn.Lb}).  The purpose of the following estimates is to estimate the contributions of the forcing terms $R^{u,1}, L^b_1, R^{v,1}, L^b_2$ into $\mathcal{R}_1, \mathcal{R}_2, \mathcal{R}_3$, as shown in (\ref{R.1.e}) - (\ref{R.3.e}) Thus, we will analyze the forcing contributions: 
\begin{lemma} For any vector fields $[u,v] \in \X$, the following estimates hold: 
\begin{align} \n
|\int \eps^{-\frac{1}{2}-\gamma} &\Big\{ R^{u,1} \cdot u + \eps R^{v,1} \cdot v \Big\}| + |\int \eps^{-\frac{1}{2}-\gamma} \Big\{ R^{u,1} \cdot -\beta_y + \eps R^{v,1} \beta_x \Big\}| \\ \n
& + |\int \eps^{-\frac{1}{2}-\gamma} \p_y R^{u,1} \cdot \p_y (u y^2 w)| + |\int \eps^{-\frac{1}{2}-\gamma} \eps \p_y R^{v,1} \p_x \{u y^2 w \}| \\ \n
&+ |\int L^b_1 \cdot u + \eps L^b_2 v| +| \int L^b_1 \cdot -\beta_y + \int \eps L^b_2 \beta_x| \\ \n
& + |\int \p_y L^b_1 \cdot \p_y\{u y^2 w \}| + |\int \eps \p_y L^b_2 \cdot \p_x \{u y^2 w \} | \\ \label{forcing.est}
& \lesssim (C(a_0, b_0, a_L, b_L) + \eps^{\frac{1}{4}-\gamma}) ||u,v||_{\X}.
\end{align}
\end{lemma}
\begin{proof}

We recall estimate (\ref{sigma}) from the Appendix, which we then directly use. First, we start with the contributions to $\mathcal{R}_1$, shown in (\ref{R.1.e}):
\begin{align} \n
\int \eps^{-\frac{1}{2}-\gamma} \Big\{ R^{u,1} \cdot u + \eps R^{v,1} \cdot v \Big\} &\le \eps^{-\frac{1}{2}-\gamma} ||R^{u,1}, \sqrt{\eps}R^{v,1}||_{L^2} ||u, \sqrt{\eps}v||_{L^2} \\
& \le \eps^{-\frac{1}{2}-\gamma} \eps^{\frac{3}{4}} \bigO(L) ||u_x, \sqrt{\eps}v_x||_{L^2}.
\end{align}

We now move the contributions from $\mathcal{R}_2$, shown in (\ref{R.2.e}):
\begin{align}
\int \eps^{-\frac{1}{2}-\gamma} \Big\{ R^{u,1} \cdot -\beta_y + \eps R^{v,1} \beta_x \Big\} \le \eps^{-\frac{1}{2}-\gamma}\eps^{\frac{3}{4}}  ||\beta_y, \sqrt{\eps} \beta_x||_{L^2}.
\end{align}

Next, we move to the higher order quantities from $\mathcal{R}_3$, shown in (\ref{R.3.e}):
\begin{align} \n
\int\eps^{-\frac{1}{2}-\gamma} \p_y R^{u,1} \cdot \p_y (u y^2 w) &= \int\eps^{-\frac{1}{2}-\gamma} \p_y R^{u,1} \cdot u_y y^2 w + \int\eps^{-\frac{1}{2}-\gamma} \p_y R^{u,1} \cdot u 2y w \\ \n
& \le \eps^{-\frac{1}{2}-\gamma} ||\p_y R^{u,1} y||_{L^2} \Big[ ||u_y y||_{L^2} + \bigO(L) ||u_x||_{L^2} \Big] \\
& \le \eps^{-\frac{1}{2}-\gamma}\eps^{\frac{3}{4}}  \Big[ ||u_y y||_{L^2} + \bigO(L) ||u_x||_{L^2} \Big], \\ \n
\int \eps^{-\frac{1}{2}-\gamma} \eps \p_y R^{v,1} \p_x \{u y^2 w \} &= \int \eps^{-\frac{1}{2}-\gamma} \eps \p_y R^{v,1} \Big[ u_x y^2 w - u y^2 \Big] \\ \n
& \le \eps^{-\frac{1}{2}-\gamma}  ||\sqrt{\eps} \p_y R^{v,1} y||_{L^2} ||\sqrt{\eps}u_x y||_{L^2} \\
& \le \eps^{-\frac{1}{2}-\gamma}\eps^{\frac{3}{4}}  ||\sqrt{\eps}u_x y||_{L^2}.
\end{align}

The estimates on $L^b_1, L^b_2$ contributions follow directly from estimate (\ref{est.Lb}). This concludes the proof. 

\end{proof}

Combining (\ref{nl.R.1}), (\ref{nl.R.2}), (\ref{nl.R.3}), and (\ref{forcing.est}):
\begin{corollary} For $\mathcal{R}_1, \mathcal{R}_2, \mathcal{R}_3$ defined as in (\ref{R.2}), (\ref{R.3}), (\ref{R.1}), we have: 
\begin{align}
|\mathcal{R}_1 + \mathcal{R}_2 + \mathcal{R}_3| \lesssim \Big[C(a_0, b_0, a_L, b_L) + \eps^{\frac{1}{4}-\gamma} \Big] ||u,v||_{\X} + \eps^{\frac{\gamma}{2}}||u,v||_{\X}^2 + \eps^{\frac{\gamma}{2}}||\bar{u}, \bar{v}||_{\X}^4. 
\end{align}
\end{corollary}

Combining the above estimate with (\ref{lin.complete}) and (\ref{unif.cor}), and performing Young's inequality for the product $C(a_0, b_0, a_L, b_L) ||u,v||_{\X}$ above to absorb $||u,v||_{\X}^2$ to the left-hand side of (\ref{z.estimate}), we have now established the main \textit{a-priori} estimate: 
\begin{theorem}[$\X$-Estimate] \label{th.z} Solutions $[u,v] \in \X$ to the system (\ref{lin.1}) - (\ref{lin.3}), with the boundary conditions (\ref{BC.dirichlet}) - (\ref{BC.stress.free}) satisfy the following estimate: 
\begin{align} \label{z.estimate}
||u,v||_{\X}^2 \lesssim C(a_0, b_0, a_L, b_L) +  \eps^{\frac{1}{4}-\gamma} + \eps^{\frac{\gamma}{2}} ||\bar{u}, \bar{v}||_{\X}^4.
\end{align}
\end{theorem}

With the main \textit{a-priori} estimate in hand, we give the formal arguments leading to existence of a solution in $\X$ in Appendix \ref{app.ex}. In particular, Theorem \ref{th.z} coupled with Proposition \ref{nl.exist} gives the main result, Theorem \ref{th.main}.

\appendix

\section{Construction of Profiles} \label{app.construct}

\subsection{Specification of $R^u$}

Define: 
\begin{align} \label{defn.RU}
R^u := U^\eps \p_x U^\eps + V^\eps \p_y U^\eps + \p_x P^\eps - \p_{yy} U^\eps - \eps \p_{xx} U^\eps, \\ \label{defn.RV}
R^v: = U^\eps \p_x V^\eps + V^\eps \p_y V^\eps + \frac{\p_y}{\eps} P^\eps - \p_{yy} V^\eps - \eps \p_{xx} V^\eps.
\end{align}

In this subsection, we will specify the equations we shall take for $R^u$. We will first expand the nonlinear terms in the following manner: 
\begin{align} \n
U^\eps \p_x U^\eps = &\Big( u^0_e + u^0_p + \sqrt{\eps}u^1_e + \sqrt{\eps} u^1_p + \eps^{\frac{1}{2}+\gamma}u \Big) \times \\ \n
& \hspace{10 mm} \Big( u^0_{ex} + u^0_{px} + \sqrt{\eps}u^1_{ex} + \sqrt{\eps} u^1_{px} + \eps^{\frac{1}{2}+\gamma}u_x \Big) \\ \n
& = \{u^0_e(x,0) + u^0_p\} u^0_{px} + u^0_{ex}(x,0) u^0_p \\ \n 
& \hspace{10 mm} + \{u^0_e - u^0_e(x,0) \} u^0_{px}  + \{ u^0_{ex} - u^0_{ex}(x,0) \} u^0_p \\ \n
&\hspace{10 mm} + \sqrt{\eps} \Big[ u^0_p u^1_{ex} + u^1_e u^0_{px} \Big] + \sqrt{\eps} \Big[ \{u^0_e -  u^0_e(x,0) \} u^1_{px} \\ \n
& \hspace{10 mm} \n + \sqrt{\eps} \Big[ \{u^0_{e}(x,0) + u^0_p \} u^1_{px} + \{ u^0_{ex}(x,0) + u^0_{px} \}u^1_p \Big] \\ \n
& \hspace{10 mm} + \{ u^0_{ex} - u^0_{ex}(x,0) \} u^1_p\Big]  + \eps \Big[ (u^1_e + u^1_p) u^1_{px} \\ \n 
&\hspace{10 mm}+ u^1_{ex}  u^1_p \Big]  + \Big[ u^0_e u^0_{ex} + \sqrt{\eps} \Big( u^1_{ex} u^0_e + u^1_e u^0_{ex} \Big) + \eps u^1_{e} u^1_{ex} \Big] \\
& \hspace{10 mm} + \eps^{\frac{1}{2}+\gamma}\{ u_s u_x + u_{sx}u \} + \eps^{1+2\gamma} uu_x.
\end{align}

\begin{align} \n
V^\eps \p_y U^\eps = &\Big( \frac{v^0_e}{\sqrt{\eps}} + v^0_p + v^1_e + \sqrt{\eps} v^1_p + \eps^{\frac{1}{2}+\gamma}v \Big) \\  \n
& \hspace{10 mm} \times \Big( \sqrt{\eps} u^0_{eY} + u^0_{py} + \eps u^1_{eY} + \sqrt{\eps} u^1_{py} + \eps^{\frac{1}{2}+\gamma}u_y \Big) \\\n
& = \Big( y v^0_{eY}(x,0) + v^0_p + v^1_e(x,0) \Big) u^0_{py}  \\ \n
& \hspace{10 mm} + \sqrt{\eps} \Big[ \{v^0_p + y v^0_{eY}(x,0) + v^1_e(x,0) \} u^1_{py} + u^0_{py}v^1_p  \Big] \\\n
&  \hspace{10 mm} + \sqrt{\eps} \Big[ v^0_p( u^0_{eY} + \sqrt{\eps} u^1_{eY} ) \Big] \\ \n
& \hspace{10 mm} + \Big[ \frac{v^0_e}{\sqrt{\eps}} - y v^0_{eY}(x,0) \Big] u^0_{py} + \eps v^1_p \Big( u^0_{eY} + \sqrt{\eps}u^1_{eY} \Big) + \eps v^1_p u^1_{py}  \\ \n
& \hspace{10 mm} + \Big[ v^0_e - Y v^0_{eY}(x,0) \Big] u^1_{py} + \sqrt{\eps} \Big[ v^1_e - v^1_e(x,0) \Big] u^1_{py} \\\n
& \hspace{10 mm} + \Big[ v^1_e - v^1_e(x,0) - Y v^1_{eY} \Big] u^0_{py} + \sqrt{\eps}y v^1_{eY} u^0_{py} \\ \n
& \hspace{10 mm} + \Big[ v^0_{e} u^0_{eY} + \sqrt{\eps} \Big( v^0_e u^1_{eY} + v^1_e u^0_{eY} \Big) + \eps v^1_e u^1_{eY} \Big] \\ 
& \hspace{10 mm} + \eps^{\frac{1}{2}+\gamma} \{u_{sy}v + v_s u_y \} + \eps^{1 + 2\gamma} vu_y.
\end{align}

Inserting into the system (\ref{defn.RU}) gives the following expansion: 
\begin{align} \n
R^u & = \Big\{ \{u^0_e(x,0) + u^0_p\} u^0_{px} + u^0_{ex}(x,0) u^0_p + \{y v^0_{eY}(x,0) + v^0_p + v^1_e(x,0) \} u^0_{py} \\ \label{sp.eq.pr.0}
& \hspace{10 mm} + P^0_{px} - u^0_{pyy} \Big\} \\ \n
& + \sqrt{\eps} \Big\{ \{u^0_e(x,0) + u^0_p \} u^1_{px} + \{ u^0_{ex}(x,0) + u^0_{px} \} u^1_p \\ \label{sp.eq.pr.1}
& \hspace{10 mm} + \{y v^0_{eY}(x,0) + v^0_p + v^1_e(x,0) \} u^1_{py}  + u^0_{py} v^1_p - u^1_{pyy} + P^1_{px} - F_1 \Big\} \\ \label{sp.eq.eul.1}
& + \sqrt{\eps} \Big[ u^0_e u^1_{ex} + u^0_{ex} u^1_e + v^0_e u^1_{eY} + v^1_e u^0_{eY} + P^1_{ex} \Big] \\ \label{final.line}
& + \tilde{R^{u,1}} + \eps P^2_{px} + \eps^{\frac{1}{2}+\gamma} \Big[-\Delta_\eps u + S^u(u,v) + P_x + N^u(u,v)\Big].
\end{align}

We will define: 
\begin{align} \n
F_1 := &v^0_p \{ u^0_{eY} + \sqrt{\eps} u^1_{eY} \} + \frac{1}{\sqrt{\eps}} \Big[ \frac{v^0_e}{\sqrt{\eps}} - y v^0_{eY}(x,0) \Big] u^0_{py} \\ \n
& + y v^1_{eY} u^0_{py} + \frac{1}{\sqrt{\eps}} \Big[ \{u^0_e - u^0_e(x,0) \} u^0_{px} + \{u^0_{ex} - u^0_{ex}(x,0) \} u^0_p \Big] \\ \label{defn.F1}
& + u^0_p u^1_{ex} + u^1_e u^0_{px}, \\ \n
\tilde{R^{u,1}} := &\eps^{\frac{3}{2}} v^1_p u^1_{eY} + \eps v^1_p u^1_{py} + \Big( v^0_e - Y v^0_{eY}(x,0) \Big) u^1_{py}  \\ \n
& + \Big( v^1_e - v^1_e(x,0) - Y v^1_{eY}(x,0) \Big) u^0_{py} + \sqrt{\eps} \Big[ v^1_e - v^1_e(x,0) \Big] u^1_{py} \\ \n
& + \sqrt{\eps} \{ u^0_e - u^0_e(x,0) \} u^1_{px} + \sqrt{\eps} u^1_p \{ u^0_{ex} - u^0_{ex}(x,0) \} \\ \n
& + \eps \Big[ (u^1_e + u^1_p) u^1_{px} + u^1_{ex} u^1_p + u^0_{eY}v^1_p \Big] \\  \label{tilde.R.def}
& + \eps u^0_{pxx} + \eps^{\frac{3}{2}} u^1_{pxx} + \eps \Big[ u^1_e u^1_{ex} + v^1_e u^1_{eY} + \Delta u^0_e + \sqrt{\eps} \Delta u^1_e \Big], \\ \label{Nu.defn}
N^u(u,v) & := \eps^{\frac{1}{2}+\gamma} \Big( uu_x + vu_y \Big), \\ \label{Nv.defn}
N^v(u,v) & := \eps^{\frac{1}{2}+\gamma} \Big( uv_x + vv_y \Big), \\ \label{Su.defn}
S^u(u,v) &:= u_s u_x + u_{sx}u + v_s u_y + u_{sy}v, \\ \label{Sv.defn}
S^v(u,v) &:= u_s v_x + v_{sx}u + v_s v_y + v_{sy}v.
\end{align}

Equations (\ref{sp.eq.pr.0}) - (\ref{sp.eq.eul.1}) define the equations for our approximate layers, as seen in (\ref{sys.pr.0}), (\ref{sys.eul.1}), and (\ref{sys.pr.1}), thereby contributing the final line, (\ref{final.line}) into the remainder equation, (\ref{NSR.1}). We must actually modify $\tilde{R}^{u,1}$ to $R^{u,1}$, which accounts for the fact that the layers $[u^1_p, v^1_p]$ are cutoff at $y \rightarrow \infty$:
\begin{align} \label{defn.Ru1}
R^{u,1} := \tilde{R}^{u,1} + \sqrt{\eps} R^u_p + \eps P^2_{px},
\end{align}

where:
\begin{align} \n
R^{u}_p &:= \{u^0_e(x,0) + u^0_p \} u^1_{px} + \{ u^0_{ex}(x,0) + u^0_{px} \} u^1_p + \{y v^0_{eY}(x,0) + v^0_p \} u^1_{py} \\
& \hspace{10 mm} + u^0_{py} v^1_p - u^1_{pyy} + P^1_{px} - F_1.
\end{align}

We are then left with: 
\begin{align} \label{rem.eq.u}
\eps^{\frac{1}{2}+\gamma} \Big[ -\Delta_\eps u + S^u(u,v) + P_x + N^u(u,v) \Big] = R^{u,1}.
\end{align}

\subsection{Specification of $R^v$}

We turn now to the simplification of (\ref{defn.RV}). 
\begin{align} \label{Rv.1}
R^v = &\frac{1}{\sqrt{\eps}} \Big[ u^0_e v^0_{ex} + v^0_e v^0_{eY} + P^0_{eY} \Big] + \frac{P^0_{py}}{\eps} + \frac{P^1_{py}}{\sqrt{\eps}} \\ \label{Rv.2}
& \hspace{10 mm} + \Big[ u^0_e v^1_{ex} + u^1_{e}v^0_{ex} + v^0_e v^1_{eY} + v^0_{eY}v^1_e + P^1_{eY}  \Big] \\ \n
& \hspace{10 mm} + v^0_{px} \Big( u^0_e + u^0_p + \sqrt{\eps}u^1_e + \sqrt{\eps}u^1_p \Big) \\ \n
& \hspace{10 mm} + \frac{v^0_{ex}}{\sqrt{\eps}} \Big( u^0_p + \sqrt{\eps}u^1_p \Big) + v^1_{ex} \Big( u^0_p + \sqrt{\eps} u^1_p \Big) \\ \n
& \hspace{10 mm} + \frac{v^0_e}{\sqrt{\eps}} v^0_{py} + v^0_e v^1_{py} + v^0_p \Big( v^0_{eY} + v^0_{py} + \sqrt{\eps}v^1_{eY} + \sqrt{\eps}v^1_{py} \Big) \\ \label{Rv.3}
& \hspace{10 mm} + v^1_e( v^0_{py} + \sqrt{\eps}v^1_{py}) + \Delta_\eps v^0_p + P^2_{py} \\ \n
& \hspace{10 mm} + \sqrt{\eps} v^1_{px} \Big( u^0_e + u^0_p + \sqrt{\eps} u^1_e + \sqrt{\eps} u^1_p \Big) + \Delta_\eps v^1_p \\ \n
& \hspace{10 mm} + \sqrt{\eps} v^1_p \Big( v^0_{eY} + v^0_{py} + \sqrt{\eps}v^1_{eY} + \sqrt{\eps} v^1_{py} \Big) \\ \n
& \hspace{10 mm} + \Big[ \sqrt{\eps} \Delta v^0_e + \eps \Delta v^1_e + \sqrt{\eps} u^1_e v^1_{ex} + \sqrt{\eps} v^1_e v^1_{eY}  \Big] \\ 
& \hspace{10 mm} + \eps^{\frac{1}{2}+\gamma} \Big[-\Delta_\eps v + S^v(u,v) + \frac{P_y}{\eps} + N^v(u,v) \Big]
\end{align}

We shall make the identifications so that (\ref{Rv.1}) and (\ref{Rv.2}) vanish by using these equations to define the construction of the approximate layers in (\ref{sys.pr.0}), (\ref{sys.eul.2}), and (\ref{sys.pr.1}). We then define $P^2_p$ via (\ref{Rv.3}): 
\begin{align} \n
P^2_{p} = &-\int_y^\infty v^0_{px} \Big( u^0_e + u^0_p + \sqrt{\eps}u^1_e + \sqrt{\eps}u^1_p \Big)  + \frac{v^0_{ex}}{\sqrt{\eps}} \Big( u^0_p + \sqrt{\eps}u^1_p \Big) \\ \n
& + v^1_{ex} \Big( u^0_p + \sqrt{\eps} u^1_p \Big) + \frac{v^0_e}{\sqrt{\eps}} v^0_{py} + v^0_e v^1_{py} + v^0_p \Big( v^0_{eY} + v^0_{py} + \sqrt{\eps}v^1_{eY} + \sqrt{\eps}v^1_{py} \Big) \\ \label{pressure.aux.1}
& + v^1_e( v^0_{py} + \sqrt{\eps}v^1_{py}) + \Delta_\eps v^0_p.
\end{align}

This choice enforces the vanishing of line (\ref{Rv.3}). We are then left with: 
\begin{align} \label{rem.eq.v}
\eps^{\frac{1}{2}+\gamma} \Big[ -\Delta_\eps v + S^v(u,v) + \frac{P_y}{\eps} + N^v(u,v) \Big] = R^{v,1},
\end{align}

where
\begin{align} \n
R^{v,1} := & \sqrt{\eps} v^1_{px} \Big( u^0_e + u^0_p + \sqrt{\eps} u^1_e + \sqrt{\eps} u^1_p \Big) + \Delta_\eps v^1_p \\ \n
 &+ \sqrt{\eps} v^1_p \Big( v^0_{eY} + v^0_{py} + \sqrt{\eps}v^1_{eY} + \sqrt{\eps} v^1_{py} \Big)  \\ \label{defn.Rv1}
& + \Big[ \sqrt{\eps} \Delta v^0_e + \eps \Delta v^1_e + \sqrt{\eps} u^1_e v^1_{ex} + \sqrt{\eps} v^1_e v^1_{eY}  \Big].
\end{align}

This defines the second equation for the remainder, as seen in (\ref{NSR.2}).

\subsection{Construction of Layers}

We are prescribed the Euler flow $[u^0_e, v^0_e, P^0_e]$. The first task is to verify that there exists Euler flows satisfying assumptions (\ref{euler.as.1}) - (\ref{euler.as.4}): 
\begin{proposition} \label{eul.exist} There exists a nontrivial set of Euler flows, $[u^0_e, v^0_e, P^0_e]$ satisfying assumptions (\ref{euler.as.1}) - (\ref{euler.as.4}).
\end{proposition}

We will start with the shear flow $U_0(Y)$, satisfying the following hypothesis:  
\begin{align} \label{shear.as.1}
&c_0 \le U_0 \le C_0, \\ \label{shear.as.2}
& U_0 \text{ smooth, with rapidly decaying derivatives, } \\ \label{shear.as.3}
& \p_Y U_0 \ge 0, \\ \label{shear.as.4}
& U_0 = 1 \text{ in a neighborhood of $0$}.
\end{align}

Such a shear has stream function $\phi_0(Y) = \int_0^Y U_0$. Such a stream function has the following asymptotics: 
\begin{align}
\phi_0|_{Y = 0} = 0, \hspace{3 mm} \phi_0|_{x = 0} = \phi_0|_{x = L} = \phi_0(Y), \hspace{3 mm} \lim_{Y \rightarrow \infty} \frac{\phi_0}{Y} = U_\infty \in (c_0, C_0). 
\end{align}

Note that assumption (\ref{shear.as.1}) implies $c_0 Y \le \phi_0 \le C_0 Y$. To define our final Euler flow, we must first solve for an perturbative stream function, $\psi$, using the following elliptic equation: 
\begin{align} \n
-\Delta \psi = \p_Y U_0 + f_e(\phi_0 + \psi), \hspace{3 mm} &\psi|_{x = 0} = A_0(Y), \hspace{3 mm} \psi_{x = L} = A_L(Y), \\ \label{ell.1} & \psi|_{Y = 0} = 0, \hspace{8 mm} \psi|_{Y \rightarrow \infty} = 0. 
\end{align}

We will assume the following conditions on $f_e$ and the boundary data $A_{0,L}$: 
\begin{align} \label{g.as.1}
&0 \le f_e \le \delta << 1, \\ \label{g.as.2}
&|\p^kf_e(x+a)| \lesssim |\p^k f_e(x)| \text{ for } a \ge 0, \\ \label{g.as.3}
&f_e \in C^\infty(\mathbb{R}), \text{ rapidly decaying in it's argument}, \\ \label{g.as.4}
&f_e \text{ supported in a neighborhood away from $0$ }, \\ \label{AS.as.1}
&0 \le A_0, A_L \le \delta \times L^{10}, \\ \label{AS.insert.1}
& |\p_Y^k \{A_0, A_L \}| \le \delta \times L^{10} \\  \label{AS.as.2} 
& A_0, A_L \in C^\infty(\mathbb{R}_+), \text{ rapidly decaying in it's argument}, \\ \label{AS.as.3}
& A_0, A_L \text{ supported in a neighborhood away from $0$. } \\ \label{AS.as.4}
& A_0 \neq A_L.
\end{align}

It is straightforward to see that the set of admissible $f_e, A_0, A_L$ is nonempty. First, via hypothesis (\ref{shear.as.3}) and (\ref{g.as.1}), we have $\Delta \psi \le 0$, so that via the maximum principle and assumption on the boundary data (\ref{AS.as.1}):
\begin{align} \label{max.nl}
\psi \ge 0. 
\end{align}

\begin{lemma} Assume (\ref{shear.as.1}) - (\ref{shear.as.4}) and the assumptions (\ref{g.as.1}) - (\ref{AS.as.4}) are satisfied. For $0 < L << \delta << 1$, the following energy estimate holds: 
\begin{align}
||Y^k \psi||_{H^1} \le C_k \bigO(\delta). 
\end{align}
\end{lemma}
\begin{proof}

Define: 
\begin{align}
B(x,Y) = \frac{L-x}{L} A_0(Y) + \frac{x}{L}A_L(Y). 
\end{align}

$B$ is smooth and all derivatives are order $\delta$ by the assumptions (\ref{AS.insert.1}) on $A_{0,L}$. Define now $\bar{\psi} = \psi - B$, which satisfies:
\begin{align} \label{B.h.1}
-\Delta \bar{\psi} = \Delta B + \p_Y U_0 + f_e(\phi_0 + \psi), \hspace{5 mm} \bar{\psi}|_{\p \Omega} = 0. 
\end{align}

An energy estimate coupled with Poincare's inequality gives: 
\begin{align} \n
\int |\nabla \bar{\psi}|^2 &= \int \Big( \Delta B + \p_Y U_0\Big) \cdot \bar{\psi} + \int f_e(\phi_0 + \psi) \cdot \bar{\psi} \\ \label{ident.1}
& \le \bigO(\delta, L) ||\bar{\psi}_x||_{L^2} + ||f_e(\phi_0 + \psi)||_{L^2} \bigO(L) ||\bar{\psi}_x||_{L^2}.
\end{align}

We now use (\ref{max.nl}) together with assumptions (\ref{g.as.2}) and (\ref{shear.as.1}) to estimate: 
\begin{align}
||f_e(\phi_0 + \psi)||_{L^2} \le ||f_e(\phi_0)||_{L^2} \le ||f_e(c_0 Y)||_{L^2} \le \bigO(\delta). 
\end{align}

This concludes the proof. 
\end{proof}

We now upgrade to weighted estimates, and higher regularity: 
\begin{lemma} Assume (\ref{shear.as.1}) - (\ref{shear.as.4}) and the assumptions (\ref{g.as.1}) - (\ref{AS.as.4}) are satisfied. For $0 < L << \delta << 1$, the following energy estimate holds: 
\begin{align} \label{hove}
||Y^m \p_x^j \p_Y^k \psi||_{L^2} \le C_{m, k, j} \delta \text{ for any $k,m, j \ge 0$.}
\end{align}
\end{lemma}
\begin{proof}

The first step is to differentiate (\ref{ell.1}) in $Y$. Defining $\psi^{(1)} := \p_Y \psi$, this produces: 
\begin{align} \n
-\Delta \psi^{(1)} = \p_{Y}^2 U_0 + f_e'(\phi_0 + \psi) (\p_Y \phi_0 + \psi^{(1)}), \\
\psi^{(1)}_Y|_{Y = 0} = 0, \hspace{3 mm} \psi^{(1)}|_{x = 0,L} = \p_Y A_{0,L},
\end{align}

where we have evaluated (\ref{ell.1}) using the condition (\ref{shear.as.4}), (\ref{g.as.4}), and (\ref{AS.as.3}) to obtain the Neumann boundary condition above. A homogenization procedure and energy estimate nearly identical to (\ref{B.h.1}) - (\ref{ident.1}) produces: 
\begin{align}
\int |\psi_{xY}|^2 + |\psi_{YY}|^2 \lesssim \bigO(\delta). 
\end{align}

By using now the equation(\ref{ell.1}), we also obtain $\psi_{xx}$ in $L^2$. Note crucially that $||\p_Y U_0||_{L^2} \le \bigO(L)$ due to the integration in the $x$-direction, which prevents us from requiring a smallness condition on $\p_Y U_0$. One can iterate this procedure for higher derivatives. It is also straightforward to obtain weighted in $Y$ estimates, using hypothesis (\ref{shear.as.2}), (\ref{g.as.3}), and (\ref{AS.as.2}) to absorb weights of $Y$. This concludes the proof of (\ref{hove}).

\end{proof}

\begin{proof}[Proof of Proposition \ref{eul.exist}]
If we define $\phi^E := \phi_0 + \psi$, then $\phi^E$ solves: 
\begin{align} \n
-\Delta \phi^E = f_e(\phi^E), \hspace{5 mm} &\phi^E(0,Y) = \phi_0 + A_0(Y), \hspace{5 mm} \phi^E(L,Y) = \phi_0 + A_L(Y), \\
& \phi^E(x,0) = 0, \hspace{5 mm} \frac{\phi^E(x,Y)}{Y} \xrightarrow{Y \rightarrow \infty} U_\infty.
\end{align}

Solutions to such elliptic equations solve the 2D Euler equations (see \cite{CS}) by setting: 
\begin{align}
u^0_e = \p_Y \phi^E, \hspace{3 mm} v^0_e = -\p_x \phi^E = - \p_x \psi, \hspace{3 mm} P^0_e = -\frac{1}{2}|\nabla \phi^E|^2 + F_e(\phi^E), \hspace{2 mm} F_e' = f_e. 
\end{align}

We view $\psi$ as a $\bigO(\delta)$-perturbation to the shear flow $(U_0(Y),0)$ for which $\phi^E = \phi_0$, which is therefore achieved by setting $f_e = A_0 = A_L = 0$. Note that the property (\ref{AS.as.4}) creates the $x$-dependence, for if $A_0 = A_L$, one could solve (\ref{ell.1}) for $\psi^1$ as just a function of $Y$, creating another shear flow. All properties (\ref{euler.as.1}) - (\ref{euler.as.4}) are easily verified, where the crucial smallness is obtained through the use of (\ref{hove}): 
\begin{align}
||\frac{v^0_E}{Y}||_{L^\infty} \le ||v^0_{eY}||_{L^\infty} = || \psi_{xY}||_{L^\infty} \le ||\psi_{xY}||_{H^2} \le \bigO(\delta). 
\end{align}

This concludes the proof of the proposition. 
\end{proof}

We will now abandon the particular construction of Proposition \ref{eul.exist}, and consider \textit{any} flow satisfying the assumptions of the paper, namely (\ref{euler.as.1}) - (\ref{euler.as.4}). Similar to the above considerations, there exists a function $f_e$ such that: 
\begin{align} \label{sys.fe}
u^0_{eY} - v^0_{ex} = w_e^0 = -\Delta \phi^0 = f_e(\phi^0), P^0_e = -\frac{1}{2}|\nabla \phi^0|^2 + F_e(\phi^0), \hspace{3 mm} F_e' = f_e.
\end{align}

Our assumptions (\ref{euler.as.1}) - (\ref{euler.as.4}) guarantee the following: 
\begin{align*}
c_0 \le u^0_e  \le C_0 \Rightarrow \phi^0 = \int_0^Y u^0_e \sim Y, 
\end{align*}

coupled with $w^0_e$ is bounded and decaying in $Y$ implies that $f_e$ together with derivatives are bounded and decaying, which we state now as a lemma: 
\begin{lemma} \label{lemma.a1} Define $f_e$ to satisfy the equalities in (\ref{sys.fe}). The assumptions on $[u^0_e, v^0_e]$ stated in (\ref{euler.as.1}) - (\ref{euler.as.4}) imply that $f_e$ together with sufficiently many derivatives is bounded and decaying in its argument. 
\end{lemma}

The above lemma is in spirit a converse to Proposition \ref{eul.exist}, which will be convenient for later constructions (see specifically equation (\ref{data.0})). 

In accordance with (\ref{sp.eq.pr.0}) and (\ref{Rv.1}), we will take the following system for the leading order Prandtl layer: 
\begin{align} \n
& \{u^0_e(x,0) + u^0_p\} u^0_{px} + u^0_{ex}(x,0) u^0_p + \{y v^0_{eY}(x,0) + v^0_p + v^1_e(x,0) \} u^0_{py} \\ \label{sys.pr.0}
& \hspace{10 mm} + P^0_{px} - u^0_{pyy}, \hspace{5 mm} P^0_{py} = 0, \\
& u^0_p(x,0) = u_b - u^0_e(x,0), \hspace{5 mm} u^0_p(0,y) = u^0_{p,0}(y), \hspace{5 mm} v^0_p(x,0) = -v^1_e(x,0). 
\end{align}

\begin{remark}
By rewriting the system (\ref{sys.pr.0}) for the unknowns: 
\begin{align}
\bar{u} := u^0_e(x,0) + u^0_(x,y), \hspace{5 mm} \bar{v} = yv^0_{eY}(x,0)+ v^0_p(x,y) + v^1_e(x,0), 
\end{align}
we obtain: 
\begin{align} \n
&\bar{u} \bar{u}_x + \bar{v} \bar{u}_y - \bar{u}_{yy} = u^0_e(x,0) u^0_{ex}(x,0), \hspace{5 mm} \bar{v} = - \int_y^\infty \bar{u}_x, \\ \label{pr.bar}
& \bar{u}|_{y = 0} = u_b, \hspace{5 mm} \bar{u}|_{y = \infty} = u^0_e(x,0). 
\end{align}

By evaluating equation (\ref{euler.equation.0}) at $Y = 0$, we see that $u^0_e u^0_{ex}|_{Y = 0} = - P^0_{ex}|_{Y = 0}$. Note that we do not demand any sign condition on this forcing term. 
\end{remark}

For the system (\ref{sys.pr.0}), we have:  

\begin{proposition} There exists a unique solution, $[u^0_p, v^0_p]$, to the system (\ref{sys.pr.0}), satisfying the following: 
\begin{align} \label{prandtl.estimate}
\sup_x ||y^M \p_x^{j} \p_y^{k} \{u^0_p, v^0_p \}||_{L^2_y} \lesssim C(M, k, j). 
\end{align} 

Moreover, the following profile is strictly positive: 
\begin{align} \label{gtr.0}
u^0_p + u^0_e \gtrsim 1.
\end{align}
\end{proposition}
\begin{proof}

The proof follows via an appropriate von-Mises transformation, an application of the standard parabolic maximum principle, and energy estimates in a very similar manner to \cite{GN}. We therefore omit the proof. 

\end{proof}

We will next move to the first Euler layer, which in accordance to (\ref{sp.eq.eul.1}) and (\ref{Rv.2}) is obtained via the following system: 
\begin{align} \label{sys.eul.1}
&u^0_e u^1_{ex} + u^0_{ex} u^1_e + v^0_e u^1_{eY} + v^1_e u^0_{eY} + P^1_{ex} = 0, \\ \label{sys.eul.2}
&u^0_e v^1_{ex} + u^1_{e} v^0_{ex} + v^0_e v^1_{eY} + v^0_{eY} v^1_e + P^1_{eY} = 0, \\ 
& u^1_{ex} + v^1_{eY} = 0. 
\end{align}

By going to the stream function formulation, where $\nabla^\perp \phi^1 = [u^1_e, v^1_e]$, we have: 
\begin{align} \label{data.0}
&- \Delta \phi^1 = f_e'(\phi^0) \phi^1, \\ \label{data.1}
&\phi_x^1(x,0) = -v^1_e(x,0) = v^0_p(x,0), \Rightarrow \phi^1(x,0) = 1 + \int_0^x v^0_p(x', 0) \ud x', \\ \label{data.2}
&\phi^1(0,y) = \phi^1_0(y), \hspace{5 mm} \phi^1(L, y) = \phi^1_L(y).
\end{align}

We assume the data in (\ref{data.1}), (\ref{data.2}) are well-prepared in the following sense: 
\begin{definition}[Well Prepared Boundary Data] \label{defn.WP} There exists a value of $\phi^1_{YY}|_{Y = 0}$ which is given by evaluating equation (\ref{data.0}) on $Y = 0$ and using (\ref{data.1}): $\phi^1_{YY}(x,0) = -\phi^1_{xx}(x,0) - f'_e(\phi^0) \phi^1(x,0)$. The value of $\phi^1_{YY}(x,0)|_{x = 0}$ should equal $\p_{YY}\phi^1_0|_{Y = 0}$. Similarly, $\phi^1_{YY}(x,0)|_{x = 0} = \p_{YY} \phi^1_L|_{Y = 0}$. If this is the case, we say the boundary data are well-prepared up to order $2$. The generalization to order $k$ is obtained by repeating the above procedure. 
\end{definition}

By standard elliptic regularity, one has: 
\begin{lemma} \label{lemma.a4} Assuming well-prepared boundary data, there exists a solution $\phi^1$ to the system (\ref{data.0}) -  (\ref{data.2}), satisfying the following estimate:
\begin{align} \label{euler.estimate}
||Y^m \phi^1||_{H^k} \lesssim_{k,m} 1.
\end{align}
\end{lemma}
\begin{proof}
Introduce the corrector: 
\begin{align} \label{BLC}
B(x,Y) = (1 - \frac{x}{L}) \phi^1_{0}(Y) \phi^1(x,0) + \frac{x}{L} \frac{\phi^1_{L}(Y)}{\phi^1(L,0)} \phi^1(x,0).
\end{align}

By definition, $B$ is regular and decays exponentially fast in $Y$. Homogenizing: 
\begin{align}
\bar{\phi} = \phi^1 - B, 
\end{align}

we have: 
\begin{align}
-\Delta \bar{\phi} - f'(\phi^0) \bar{\phi} = -\Delta B - f'(\phi^0) B, \hspace{5 mm} \bar{\phi}|_{\p \Omega} = 0.
\end{align}

As our boundary data are well-prepared according to Definition \ref{defn.WP}, we may take $\p_Y^2$ of the system and repeat the procedure. In particular: 
\begin{align}
-\Delta \p_{Y}^2 \phi^1 - f'_e(\phi^0) \p_Y^2 \phi^1 = 2f''(\phi^0) \phi^0_y \phi^1_Y + f'''(\phi^0) |\phi^0_Y|^2 \phi^1 + f''(\phi^0) \phi^0_{YY} \phi^1.
\end{align}

One may define the new corrector $B$ analogously to (\ref{BLC}) and perform standard elliptic estimates to conclude that: 
\begin{align}
||Y^m \{ \phi^1_{YYY}, \phi^1_{YYX}, \phi^1_{YY} \}||_{L^2} \lesssim 1. 
\end{align}

By Hardy inequality, as all derivatives of $\phi^1$ decay as $Y \rightarrow \infty$, we can conclude: 
\begin{align}
|| Y^m \phi^1_{xY}||_{L^2} \lesssim 1. 
\end{align}

From equation (\ref{data.0}), it is clear that: 
\begin{align}
||\phi^1_{xx} Y^m||_{L^2} \lesssim 1. 
\end{align}

We have thus obtained all $H^2$ quantities. Taking $\p_Y$ of (\ref{data.0}) enables us to estimate $\phi^1_{xxY}$ and taking $\p_x$ of (\ref{data.0}) enables us to estimate $\phi_{xxx}$, giving the full $H^3$ estimate. Next, we can conclude that:
\begin{align}
||\phi^1_Y||_{L^\infty([0,\infty])} \le ||\phi^1_Y||_{H^1((0,L))} \le ||\phi^1||_{H^3} \lesssim 1.
\end{align}

This enables us to iterate the procedure. 

\end{proof}

In accordance to the (\ref{sp.eq.pr.1}) and (\ref{Rv.1}), we will take the following system for the Prandtl-1 layer: 
\begin{align} \label{sys.pr.1}
&u^0 u_{px} + u^0_x u_p + v^0 u_{py} + u^0_y v_p - u_{pyy} = F_1, \hspace{5 mm} P^1_{py} = 0, \\
&u_{px} + v_{py} = 0,  \hspace{5 mm} u^1_p(x,0) = -u^1_e(x,0), \hspace{5 mm} u^1_p(0,y) = u^1_{p0}(y), \\
& u^0 := u^0_e(x,0) + u^0_p, \hspace{5 mm} v^0 := y v^0_{eY}(x,0) + v^0_p + v^1_e(x,0).
\end{align}

Here, $v_p$ will be recovered via: 
\begin{align}
v_p = \int_0^y u^1_{px} \ud y'.
\end{align}

We will homogenize in the following manner: define $\chi$ such that: 
\begin{align}
\chi(0) = 1, \hspace{5 mm} \p_y^k \chi(0) = 0 \text{ for } k \ge 1, \hspace{5 mm} \int_0^\infty \chi \ud y = 0. 
\end{align}

Then define: 
\begin{align} \label{aux.u.u}
u = u^1_p + \chi(y) u^1_e(x,0), \hspace{5 mm} v  = v^1_p  + u^1_{ex}(x,0) I_\chi(y), \hspace{5 mm} I_\chi(y) = \int_y^\infty \chi.
\end{align}

The new unknowns, $[u,v]$ satisfy the following system: 
\begin{align} \label{pr.1.hom}
&u^0 u_{x} + u^0_x u + v^0 u_{y} + u^0_y v - u_{yy} = F_1 + H_1, \\ \n
&H_1 := u^0 \chi u^1_{ex}(x,0) + u^0_x \chi u^1_e(x,0) \\
& \hspace{10  mm} + v^0 \chi' u^1_e(x,0) + u^0_{y} I_\chi u^1_{ex}(x,0) - \chi'' u^1_e(x,0) .
\end{align}

We recall the definition of $F_1$ given in (\ref{defn.F1}). Furthermore, we have the following estimate on the forcing: 
\begin{lemma} \label{lemma.a5} For any $m, k, j \ge 0$, the following estimate for $F_1$ holds: 
\begin{align}
||\langle y \rangle^m \p_y^k \p_x^j F_1||_{L^2}^2 \lesssim_{m,k,j} 1.
\end{align}
\end{lemma}
\begin{proof}
The proof follows directly due to the smoothness and rapid decay properties of $[u^0_p, v^0_p]$. 
\end{proof}

\begin{lemma}Solutions $[u,v]$ as defined in (\ref{aux.u.u}) to the problem (\ref{pr.1.hom}) satisfy the following estimate: 
\begin{align}
\sup_{x \in [0, L]} ||u||_{L^2_y}^2  + ||u_y||_{L^2}^2 \le C(u^1_{p0}) + \bigO(L)  ||u_x||_{L^2}^2. 
\end{align}
\end{lemma}

\begin{proof}

One applies $u$ to the above system, (\ref{pr.1.hom}), and integrates: 
\begin{align} 
\int \Big( u^0u_x + u^0_x u + v^0 u_y + u^0_y v \Big) \cdot u = \int \{F_1 + H_1\} \cdot u.
\end{align}

The result follows upon integrating in $x$ and estimating: 
\begin{align} \n
|\int \int u^0_y uv| &+ |\int \int u^0_x u^2| + |\int \int \{F_1 + H_1 \}u| \\
& \le C(u^1_{p0}) + \bigO(L) ||u^0_y \cdot y, u^0_x||_{L^\infty} \Big[||u_x||_{L^2}^2 + ||F_1, H_1||_{L^2}^2 \Big].
\end{align}

\end{proof}

\begin{lemma} Solutions $[u,v]$ as defined in (\ref{aux.u.u}) to the problem (\ref{pr.1.hom}) satisfy the following estimate:
\begin{align}
||u_x||_{L^2}^2 + \sup ||u_y||_{L^2_y}^2 \lesssim  C(u^1_{p0}) + ||u_y||_{L^2}^2 +  C(v^0_e) || y u_y||_{L^2}^2.
\end{align}
\end{lemma}

\begin{proof}

Introduce $v = \beta u^0$. Then the system becomes: 
\begin{align*}
- u^0 v_y &+ u^0_y v + u^0_x u + v^0 u_y - u_{yy} = -|u^0|^2 \beta_y + u^0_x u + v^0 u_y - u_{yy}.
\end{align*}

Multiplying by $-\beta_y$ and integrating in $y$ yields: 
\begin{align} \n
\int |u^0|^2 \beta_y^2 &+ \int u_{yy} \beta_y = \int |u^0|^2 \beta_y^2 - \int u_y \beta_{yy} \\ \n
& = \int |u^0|^2 \beta_y^2 + \int u_y \frac{u_{xy}}{u^0} - \int 2u_y v_y \p_y \frac{1}{u^0} - \int u_y v \p_y^2 \frac{1}{u^0} \\ \n
& = \int |u^0|^2 \beta_y^2 + \frac{\p_x}{2} \int \frac{1}{u^0} u_y^2 - \int u_y^2 \p_x \{ \frac{1}{u^0} \} - 2 \int u_y v_y \p_y \frac{1}{u^0} \\
& - \int u_y v \p_y^2 \frac{1}{u^0}.
\end{align}

Upon integrating further in $x$, the final three terms above are estimated: 
\begin{align} \n
|- \int \int u_y^2 \p_x \{ \frac{1}{u^0} \}& - 2 \int \int u_y v_y \p_y \frac{1}{u^0} - \int \int u_y v \p_y^2 \frac{1}{u^0}| \\
& \lesssim ||u_y||_{L^2}^2 + \delta ||v_y||_{L^2}^2. 
\end{align}

The remaining terms, upon integrating in $x$: 
\begin{align}
&|\int \int u^0_x u \cdot \beta_y| \le \bigO(L) ||\beta_y||_{L^2}^2, \\
&|\int \int v^0 u_y \cdot \beta_y| \le C(v^0_e) ||yu_y||_{L^2} ||\beta_y||_{L^2}.
\end{align}

The right-hand side is estimated simply using Holder's inequality. 

\end{proof}

\begin{lemma}[Weighted Estimates] Solutions $[u,v]$ as defined in (\ref{aux.u.u}) to the problem (\ref{pr.1.hom}) satisfy the following estimate:
\begin{align}
 || \{u_y, u_{yy} \} \cdot y \chi(y)||_{L^2}^2 \lesssim 1 + ||u_x||_{L^2}^2 + ||u_y||_{L^2}^2.
\end{align}
\end{lemma}
\begin{proof}

Applying $\p_y$ to the system gives: 
\begin{align}
u^0 u_{xy} + u^0_{xy} u + v^0 u_{yy} + u^0_{yy}v - u_{yyy} = \p_y \{F_1 + H_1 \}.
\end{align}

We apply the multiplier $u_y \chi y^2 \cdot (1-x)$. The main positive terms are: 
\begin{align} \n
\int u^0 u_{xy} &\cdot u_y \chi y^2 (1-x) + \int v^0 u_{yy} \cdot u_y \chi y^2 (1-x) \\ \n
& =  \frac{\p_x}{2} \int u^0 u_y^2 y^2 \chi (1-x) + \int \frac{u^0}{2} u_y^2 y^2 \chi - \int \frac{u^0_x}{2} u_y^2 y^2 \chi (1-x) \\ \n
& - \int \frac{v^0_y}{2} u_y^2 \chi y^2 (1-x) - \int u_y^2 v^0 y (1-x) - \int \frac{u_y^2}{2} y^2 (1-x) \chi' \\ \n
& =  \frac{\p_x}{2} \int u^0 u_y^2 y^2 \chi (1-x) + \int \frac{u^0}{2} u_y^2 y^2 \chi - \int u_y^2 v^0 y (1-x) \\ \label{int.1}
& - \int \frac{u_y^2}{2} y^2 (1-x) \chi'.
\end{align}

Upon taking integration in $x$ from $0$ to $X_\ast$, we obtain using the smallness of $v^0_{eY}(x, 0)$: 
\begin{align}
\int_0^{X_\ast} (\ref{int.1}) \gtrsim \int_{x = X_\ast} u^0 u_y^2 y^2 \chi (1-X_\ast) + ||u_y y \chi||_{L^2}^2 - ||u_y||_{L^2}^2. 
\end{align}

The remaining terms, upon integrating in $x$ from $[0, X_\ast]$: 
\begin{align}
&|\int u^0_{xy} u \cdot u_y \chi y^2 (1-x)| \le ||u^0_{xy} y^2||_{L^\infty} ||u_y||_{L^2}^2, \\
&|\int u^0_{yy}v \cdot u_y \chi y^2 (1-x)| \le ||u^0_{yy} y^2||_{L^\infty} ||u_y||_{L^2}||v_y||_{L^2}, \\
&-\int u_{yyy} u_y \chi y^2 (1-x) \gtrsim ||u_{yy} y \chi||_{L^2}^2 - ||u_y||_{L^2}^2, \\
&|\int \p_y \{F_1 + H_1 \} \cdot u_y \chi y^2 (1-x)| \lesssim 1 + ||u_y||_{L^2}^2.
\end{align}

\end{proof}

Placing the above series of estimates together closes the basic estimate for $u^1_p$. It is possible to take $\p_x^k$ and repeat with weights $y^m$. We omit these details. Summarizing: 
\begin{lemma} For any $k, m \ge 0$, solutions $[u,v]$ as defined in (\ref{aux.u.u}) to the problem (\ref{pr.1.hom}) satisfy the following estimate:
\begin{align} \label{est.pr.1}
\sup_x || y^m \p_x^k u_p ||_{L^2_y} + ||y^m \p_x^k u_{py}||_{L^2} + ||v_p||_{L^\infty} \lesssim C(k,m).
\end{align}
\end{lemma}
\begin{proof}
Only the $v_p$ estimate remains to be proven, for which we appeal to Hardy (as $v_p|_{y = 0} = 0$): 
\begin{align}
v_p^2 = \int_0^y v_p v_{py} \le ||v_{py}||_{L^2_y} ||y v_{py}||_{L^2} =  ||u_{px}||_{L^2_y} ||y u_{px}||_{L^2} \lesssim 1, 
\end{align}

the final estimate following from the $u_p$ estimates in (\ref{est.pr.1}).
\end{proof}

The final task is to cut-off the Prandtl-1 layer: 
\begin{align}
u^1_p = \chi(\sqrt{\eps}y) u_p - \sqrt{\eps} \chi'(\sqrt{\eps}y) \int_y^\infty u_p(x,s) \ud s , \hspace{5 mm} v^1_p = \chi(\sqrt{\eps}y) v_p
\end{align}

It is clear that the divergence free structure is preserved and the same estimates from (\ref{est.pr.1}) hold. The error created by such a cut-off layer is: 
\begin{align} \n
R^u_p &= (1-\chi) F_1 +\sqrt{\eps} u^0\chi' v_p - \sqrt{\eps} u^0_x \chi' \int_y^\infty u_p \\ \n
& + 2\sqrt{\eps}v^0 \chi' u_p - \eps v^0 \chi'' \int_y^\infty u_p + 3\sqrt{\eps} \chi' u_{py} \\ \label{defn.cutoff}
& + 3\eps \chi'' u_p - \eps^{\frac{3}{2}} \chi''' \int_y^\infty u_p.
\end{align}  

$R^u_p$ then contributes into $R^{u,1}$, according to (\ref{defn.Ru1}). 
\begin{lemma} The remainder $R^u_p$ defined in (\ref{defn.cutoff}) satisfies the following estimate: 
\begin{align} \label{Rupest}
||R^u_p||_{L^2} + ||y \p_y R^{u}_p||_{L^2} \lesssim \eps^{\frac{1}{4}}.
\end{align}
\end{lemma}
\begin{proof}
All follow via the estimates in (\ref{est.pr.1}) aside from $F_1$, for which we must use the rapid decay and that support of $1-\chi$ is on $y \ge \frac{1}{\sqrt{\eps}}$. Next, the term with $v_p$, we must use:
\begin{align}
||\sqrt{\eps}u^0 \chi' v_p||_{L^2} \le \sqrt{\eps} ||v_p||_{L^\infty} ||1 ||_{L^2(y \le \frac{1}{\sqrt{\eps}})} \le \eps^{\frac{1}{4}}.
\end{align}

Upon applying $y \p_y$, an identical calculation yields the desired result. 
\end{proof}

\subsection{Remainder System/}

Collecting the constructions above, according to (\ref{rem.eq.u}) and (\ref{rem.eq.v}), the remainders $[u,v, P]$ are to satisfy the following system: 
\begin{align} \label{NSR.1}
&-\Delta_\eps u + S^u(u,v) + P_x = N^u(u,v) + \eps^{-\frac{1}{2}-\gamma}R^{u,1} := f_0, \\ \label{NSR.2}
&-\Delta_\eps v + S^v(u,v) + \frac{P_y}{\eps} = N^v(u,v) +  \eps^{-\frac{1}{2}-\gamma}R^{v,1}: = g_0. \\ \label{NSR.3}
&u_x + v_y = 0,
\end{align}

together with the boundary conditions: 
\begin{align} \label{BC.1}
& [u,v]|_{x = 0} = [a_0(y), b_0(y)], [u,v]|_{y = 0} = [u,v]|_{y \rightarrow \infty} = 0, \\ \label{BC.2}
& \{u_y + \eps v_x\}|_{x = L} = b_L(y), \hspace{3 mm} \{P-2\eps u_x \}|_{x = L} = a_L(y).
\end{align}

The definitions of $S_u, S_v, N^u, N^v$ are given in (\ref{Nu.defn}) - (\ref{Sv.defn}). The assumptions on $[a_0, b_0, a_L, b_L]$ are given in (\ref{as.aL.bL}). Up to renaming $a_L, b_L$, it is possible to, without loss of generality, consider the following simplification: 
\begin{align} \label{NSR.1.bar}
&-\Delta_\eps \bar{u} + S^u(\bar{u},\bar{v}) + P_x = N^u(\bar{u}, \bar{v}) + \eps^{-\frac{1}{2}-\gamma}R^{u,1} := f, \\ \label{NSR.2.bar}
&-\Delta_\eps \bar{v} + S^v(\bar{u}, \bar{v}) + \frac{P_y}{\eps} = N^v(\bar{u}, \bar{v}) +  \eps^{-\frac{1}{2}-\gamma}R^{v,1}: = g. \\ \label{NSR.3.bar}
&\bar{u}_x + \bar{v}_y = 0,
\end{align}

together with homogenized boundary conditions: 
\begin{align}  \label{BC.3}
& [\bar{u},\bar{v}]|_{x = 0} = [0, 0], [\bar{u},\bar{v}]|_{y = 0} = [\bar{u},\bar{v}]|_{y \rightarrow \infty} = 0, \\ \label{BC.4}
& \{\bar{u}_y + \eps \bar{v}_x\}|_{x = L} = \bar{b}_L(y), \hspace{3 mm} \{P-2\eps \bar{u}_x \}|_{x = L} = \bar{a}_L(y).
\end{align}

The reason is: 
\begin{lemma} \label{lemma.WLOG} If the assumptions in (\ref{as.aL.bL}) are satisfied, $[u,v]$ solves the system (\ref{NSR.1}) - (\ref{NSR.3}) with boundary conditions (\ref{BC.1}) - (\ref{BC.2}) if and only if $[\bar{u}, \bar{v}] = [u - u_0, v - v_0]$ solves (\ref{NSR.1}) - (\ref{NSR.3}) with (\ref{BC.3}) - (\ref{BC.4}), and with modifying $[f_0, g_0]$ to $[f, g]$ as defined by: 
\begin{align} \label{defn.fg}
f := f_0 + L^b_1, \hspace{3 mm} g := g_0 + L^b_2, 
\end{align}
where: 
\begin{align} \n
L^b_1 := &u_s u_{0x} + u_{sx}u_0 + v_s u_{0y} + u_{sy}v_0 \\ \n
& + \eps^{\frac{1}{2}+\gamma} \Big( u_0 u_x + u u_{0x} + v_0 u_y + u_{0y}v \Big), \\ \label{defn.Lb.1}
& + \eps^{\frac{1}{2}+\gamma} \Big( u_0 u_{0x}  + v_0 u_{0y} \Big) \\ \n
L^b_2 := &u_s v_{0x} + v_{sx}u_0 + v_s v_{0y} + v_{sy}v_0 \\ \n
& + \eps^{\frac{1}{2}+\gamma} \Big( u_0 v_x + u v_{0x} + v_0 v_y + v_{0y}v \Big) \\ \label{defn.Lb} 
&  + \eps^{\frac{1}{2}+\gamma} \Big( u_0 v_{0x} + v_0 v_{0y} \Big). 
\end{align}
where $[u_0, v_0]$ are defined below in (\ref{defn.u0}). Finally, we have the following estimate: 
\begin{align} \label{est.Lb}
||\langle y \rangle^N \Big\{ L^b_1, \sqrt{\eps} L^b_2, \p_y L^b_1, \sqrt{\eps} \p_y L^b_2\Big\}||_{L^2} \lesssim 1. 
\end{align}
\end{lemma}
\begin{proof}

We will define the following auxiliary profiles: 
\begin{align} \label{defn.u0}
u_0 = a_0(y) -  x \p_y b_0(y),  \hspace{5 mm} v_0 = b_0(y).  
\end{align}

It is clear that $[u_0, v_0]$ is a divergence free vector field, that achieves the boundary conditions at $x = 0, y = 0, y \rightarrow \infty$. It is also clear that $[u_0, v_0]$ are order-1, and decay rapidly in $y$. Consider now the difference: 
\begin{align}
\bar{u} = u - u_0, \hspace{3 mm} \bar{v} = v - v_0. 
\end{align}

At $x = L$, the following boundary conditions are satisfied: 
\begin{align}
&\bar{a}_L := P - 2\eps \bar{u}_x|_{x = L} = P - 2\eps u_x - 2\eps \p_y b_0(y) = a_L(y) - 2\eps \p_y b_0(y), \\
& \bar{b}_L := \bar{u}_y + \eps \bar{v}_x|_{x = L} = \Big(u_y + \eps v_x\Big)|_{x = L} - \p_y u_0 = b_L - \p_y u_0. 
\end{align}

It is clear that $[\bar{u}, \bar{v}]$ will achieve the boundary conditions in (\ref{BC.2}), with $[a_L, b_L]$ replaced by $[\bar{a}_L, \bar{b}_L]$, and that $[\bar{a}_L, \bar{b}_L]$ satisfy the required assumptions, (\ref{as.aL.bL}). Finally, the new profiles $[\bar{u}, \bar{v}]$ satisfy the new system (\ref{NSR.1}) - (\ref{NSR.3}) with $f, g$ defined in (\ref{defn.fg}) according to a standard linearization. The estimate in (\ref{est.Lb}) follows from the definitions (\ref{defn.Lb.1}) - (\ref{defn.Lb}), together with (\ref{defn.u0}).
\end{proof}

\begin{remark}[Notation]Due to this lemma, we can restrict to considering (\ref{NSR.1.bar}) - (\ref{BC.4}), and we will rename $[\bar{u}, \bar{v}]$ to $[u,v]$ to help simplify notation. 
\end{remark}

\begin{proposition} For $[R^{u,1}, R^{v,1}]$ defined as in (\ref{defn.Ru1}), (\ref{defn.Rv1}), we have: 
\begin{align} \label{sigma}
||R^{u,1}, \sqrt{\eps} R^{v,1}||_{L^2} + ||\langle y \rangle \p_y \{R^{u,1}, \sqrt{\eps} R^{v,1} \}||_{L^2} \le \eps^{\frac{3}{4}}.
\end{align}
\end{proposition}

\begin{proof}

We will start with $R^{u,1}$, as defined in (\ref{defn.Ru1}). The estimate on $R^u_p$ follows from recalling the prefactor of $\sqrt{\eps}$ given in (\ref{defn.Ru1}) coupled with (\ref{Rupest}). The estimate on $\eps P^2_{px}$ follows upon noticing that each term in the definition (\ref{pressure.aux.1}) exhibits rapid decay, and therefore $|P^2_{p}| \le \langle y \rangle^{-M}$. We can thus move to the terms from $\tilde{R}^{u,1}$, as defined in (\ref{tilde.R.def}), and $R^{v,1}$ as defined in (\ref{defn.Rv1}). Combining estimates (\ref{prandtl.estimate}), (\ref{euler.estimate}), and (\ref{est.pr.1}) immediately implies the desired result. 
\end{proof}

We will also record here the following, which will be in constant use throughout the paper: 
\begin{lemma}[Uniform Estimates of Profiles] \label{Lemma.Unif} With $[u_s, v_s]$ defined as in (\ref{exp.1}) - (\ref{exp.2}), for any $k, j \ge 0$, we have: 
\begin{align} \label{Lemma.Unif.1}
||y^k \p_y^k \p_x^j u_s, y^k \p_y^{k+1} \p_x^j v_s||_{L^\infty} \lesssim 1.
\end{align}
Moreover, we have the strict positivity: 
\begin{align} \label{gtr}
u_s \gtrsim 1. 
\end{align}
\end{lemma}
\begin{proof}

Using the definitions provided in (\ref{exp.1}), we see that: 
\begin{align}
\p_y^k u_s &= \p_y^k \{u^0_e + u^0_p + \sqrt{\eps} u^1_e + \sqrt{\eps}u^1_p  \} \\
& = \eps^{\frac{k}{2}}\{\p_Y^k u^0_{e}, \sqrt{\eps}\p_Y^k u^1_e\} + \p_y^k\{u^0_p, \sqrt{\eps }u^1_p \}.
\end{align}

Multiplying by $y^k$ and using $\sqrt{\eps}^k y^k = Y^k$ gives the desired result. An analogous computation can be made using the definition (\ref{exp.2}), and finally iterates of $\p_x$ do not contribute factors of $\sqrt{\eps}$, which is why they do not enhance the weight of $y$.  

Finally, the positivity in (\ref{gtr}) follows from (\ref{gtr.0}) and the uniform estimates on $u^1_e, u^1_p$ found in (\ref{euler.estimate}), (\ref{est.pr.1}).

\end{proof}

\section{Existence and Uniqueness of Remainder} \label{app.ex}

The main result of this appendix is the following: 
\begin{proposition}[Linear Existence] \label{lin.exist} Given $(f,g) \in L^2$, and given $(a_L, b_L)$ satisfying the assumptions (\ref{as.aL.bL}), for $L$ sufficiently small, there exists a unique solution to the linear problem (\ref{lin.1}) - (\ref{lin.3}), together with boundary conditions (\ref{BC.dirichlet}) -(\ref{BC.stress.free}).
\end{proposition}

\begin{proposition}[Nonlinear Existence] \label{nl.exist} Given boundary data satisfying assumptions (\ref{as.aL.bL}), there exists a unique solution $[u,v] \in \X$ to the full nonlinear problem (\ref{NSR.1.bar}) - (\ref{BC.4}). 
\end{proposition}

We will define the operator: 
\begin{align} \n
S_{\alpha, m}[u,v,P] := &-\Delta_\eps u + P_x - 10 \alpha \p_y \{ \langle y \rangle^{2m} u_y \chi_1(y) \} \\ \n
& -\Delta_\eps v + \frac{P_y}{\eps} -2 \alpha \p_y \{ \langle y \rangle^{2m} v_y \chi_1(y) \} \\ \label{stokes.weak}
& \hspace{20 mm} - \alpha \p_x \{ \langle y \rangle^{2m} \{u_y + \eps v_x\} \chi_1(y) \}.
\end{align}

defined always on divergence free vector fields, together with the boundary conditions:  
\begin{align} \n
&[u, v]|_{x = 0} = [u,v]|_{y = 0} = [u,v]|_{y \rightarrow \infty} = 0, \\ \label{bc.bc.1}
&P - 2\eps u_x|_{x = L} =a_L(y), \hspace{5 mm} u_y + \eps v_x|_{x = L} = b_L(y). 
\end{align}

Here $\chi_1$ a cutoff function which is equal to $0$ on $[0,1)$ and $1$ on $(2,\infty)$. Strictly, $S_{\alpha, m}$ must return a four-tuple, with the first two components being (\ref{stokes.weak}), and the final two components including $(a_L, b_L)$. Fix another cut-off function $\chi_2(y) = 0$ on $[0,10)$ and $\chi_2 = 1$ on $(20,\infty)$. Define now the norms: 
\begin{align} \label{defn.h1m}
&||u,v||_{H^1_{m}}^2 := ||\Big\{u_y,  \sqrt{\eps}v_y, \eps v_x \Big\} \langle y \rangle^m||_{L^2}^2 , \\ \label{defn.h2m}
& ||u,v||_{H^2_{m}}^2 := ||\Big\{u_{yy},  \sqrt{\eps}v_{yy}, \eps v_{xx}\Big\} \langle y \rangle^m \chi_2 ||_{L^2}^2.
\end{align}

Notationally, we will refer to the $m = 0$ norm as simply $H^1$. Define now the space: 
\begin{align} \label{C0s}
C_{0,S} := \Big\{ (\varphi, \phi) \in C^\infty: &\text{ compactly supported in $y$, } \\ \n & \text{supported away from $x = 0$, and } \p_x \varphi + \p_y \phi = 0  \Big\}.
\end{align}

We will define: 
\begin{align} \label{density.2}
H^1_m := \overline{C_{0,S}}^{||\cdot||_{H^1_m}},
\end{align}

and the scaled, symmetric gradient via: 
\begin{align}
D_\eps = \left( \begin{array}{cc}
\sqrt{\eps} u_x & u_y + \eps v_x \\
u_y + \eps v_x & \sqrt{\eps} v_y \end{array} \right).
\end{align}

Recalling the definition of $||\cdot||_{\X}$ from (\ref{norm.S}), our ultimate space $\X$ is defined via: 
\begin{align}  \label{defn.S0}
&\X_0 := \overline{C_{0,s}}^{||\cdot||_{H^1}}, \\ \label{defn.S}
&\X := \{[u,v] \in \X_0: ||u,v||_{\X} < \infty \}.
\end{align}

Define the weak formulation of (\ref{stokes.weak}) to be: 
\begin{align}  \n 
\int D_\eps u &\cdot D_\eps \varphi + \int \eps D_\eps v \cdot D_\eps \phi  + \int_{x = L} a_L \varphi - \int_{x = L} \eps b_L \phi \\ \n
&+ \alpha \int \Big( 10 u_y \cdot \varphi_y + 2 \eps v_y \cdot \phi_y + \{\eps^2 v_x + \eps u_y\} \cdot \phi_x \Big) \chi_1(y) \langle y \rangle^{2m} \\ \label{weak.form.1}
& - \alpha \int_{x = L} \eps y^{2m} \chi_1 b_L(y) \phi = \int f \cdot \varphi + \eps g \cdot \phi,
\end{align}

for all $(\varphi, \phi) \in C_{0,S}$.

\begin{lemma} Given $(f,g) \in L^2$, and boundary values $(a_L, b_L)$ satisfying the assumptions (\ref{as.aL.bL}), there exists a weak solution $[u,v,P] \in H^1_m$ satisfying the estimate: 
\begin{align}
\alpha ||u,v||_{H^1_m}^2 \lesssim ||f,\sqrt{\eps}g||_{L^2}^2 + ||\Big\{\frac{a_L}{\sqrt{\eps}}, b_L \Big\}y^{2m}||_{L^2(x = L)}^2.
\end{align}
\end{lemma}
\begin{proof}
The existence of solutions follows directly from Lax-Milgram. We must verify the Bilinear form in (\ref{weak.form.1}) is coercive: 
\begin{align} 
B[(u,v), (\varphi, \phi)] &:= \int D_\eps u \cdot D_\eps \varphi + \int \eps D_\eps v \cdot D_\eps \phi   \\ \n
&+ \alpha \int \Big( 10 u_y \cdot \varphi_y + 2 \eps v_y \cdot \phi_y + \{\eps^2 v_x + \eps u_y\} \cdot \phi_x \Big) \chi_1(y) \langle y \rangle^{2m}.
\end{align}

This is immediate, apart from the cross term, to which we first appeal to the density: 
\begin{align} \n
\int \eps u_y \phi^{(n)}_x \chi_1 \langle y \rangle^{2m}& \xrightarrow{n \rightarrow \infty} \int \eps u_y v_x \chi_1 \langle y \rangle^{2m} \\
& \le_{|\cdot|} \frac{1}{2} \int \eps^2 v_x^2 \chi_1 \langle y \rangle^{2m} + \frac{1}{2} \int u_y^2 \chi_1 \langle y \rangle^{2m},
\end{align}

which explains the constants of 10 appearing in (\ref{stokes.weak}). We view the terms: 
\begin{align}
\int f \cdot \varphi + \eps g \cdot \phi - \int_{x = L} a_L \varphi + \int_{x = L} \eps b_L \phi + \alpha \int_{x = L} \eps y^{2m} \chi_1 b_L \phi,
\end{align}

as a functional on $H^{-1}$. We must thus estimate the following boundary term: 
\begin{align} \n
|\int_{x = L} \alpha \eps y^{2m} \chi_1 b_L \phi| &\le \alpha ||b_L \langle y \rangle^{2m}||_{L^2(x = L)} ||\eps \phi||_{L^2(x = L)} \\
&  \lesssim ||b_L \langle y \rangle^{2m}||_{L^2(x = L)}^2 + \alpha^2 \bigO(L) ||\eps \phi_x||_{L^2}^2,
\end{align}

the latter term being absorbed into the positive contributions from $\int |D_\eps v|^2$ using the smallness of $L$ and $\alpha$. Next, we must estimate the boundary terms: 
\begin{align} \n
|\int_{x = L} a_L \varphi| &\le ||\frac{a_L}{\sqrt{\eps}}||_{L^2(x = L)} ||\sqrt{\eps}\varphi||_{L^2(x = L)} \\
& \lesssim ||\frac{a_L}{\sqrt{\eps}}||_{L^2(x = L)}^2 + \bigO(L) ||\sqrt{\eps} \varphi_x||_{L^2}^2, \\ \n
|\int_{x = L} \eps b_L \phi| &\le ||b_L||_{L^2(x = L)} ||\eps \phi||_{L^2(x = L)} \\
& \le ||b_L||_{L^2(x = L)}^2 + \bigO(L) ||\eps \phi_x||_{L^2}^2. 
\end{align}

the latter terms in both of the above calculations can be absorbed into the positive contributions from $\int |D_\eps v|^2$.
\end{proof}

It is clear that each solution $[u,v] \in H^1_m$ is automatically in $\X_0$. We will now bootstrap to $H^2_{m}$ solutions. 
\begin{lemma} Solutions $[u,v] \in H^1_m$ to the system (\ref{stokes.weak}) satisfy: 
\begin{align} \label{h2.lax}
\alpha ||u,v||_{H^2_m} \lesssim ||f, \sqrt{\eps}g||_{L^2} +  ||\Big\{\frac{a_L}{\sqrt{\eps}}, b_L \Big\}y^{2m}||_{L^2(x = L)}^2.
\end{align}

Moreover, such solutions are strong solutions, which satisfy the boundary conditions of (\ref{bc.bc.1}).
\end{lemma}
\begin{proof}

This follows formally from differentiating (\ref{stokes.weak}) in $y$ and applying the multiplier $u_y$, with the help of the cut-off function $\chi_2$ in (\ref{defn.h2m}) to avoid the corners. Rigorously, one needs to work with difference quotients within the weak formulation (\ref{weak.form.1}). We demonstrate this now for the main weighted term. Given $[u,v] \in H^1_m$, there exists a sequence $\varphi^n, \phi^n$ such that: $[\varphi^n, \phi^n ] \xrightarrow{H^1_m} [u,v]$ by the density (\ref{density.2}). Denote by $D^h$ the difference quotient in the $y$-direction: $D^hu(x,y) = \frac{u(x, y+h) - u(x,y)}{h}$. We will select the multiplier $D^{-h} D^h \varphi$ to apply the weak formulation (\ref{weak.form.1}):
\begin{align} \label{limit.1}
\int y^m \chi_1(y) u_y D^{-h} D^h \varphi^{(n)}_y = \int D^h \Big\{ y^m \chi_1(y) u_y \Big\} D^h \varphi^{(n)}_y. 
\end{align}

By definition of difference quotient, for each fixed $h$, $\langle y \rangle^m D^h \varphi^{(n)}_y \xrightarrow{L^2}  \langle y \rangle^m  D^h u_y$. Similarly, for each fixed $h$, $\langle y \rangle^m D^h u_y \in L^2$. Thus, for each fixed $h$, we can take $n \rightarrow \infty$: 
\begin{align} \label{limit.2}
(\ref{limit.1}) \xrightarrow{n \rightarrow \infty} \int D^h \Big\{ y^m \chi_1(y) u_y \Big\} D^h u_y. 
\end{align}

Next taking limits in $h$ gives: 
\begin{align}
(\ref{limit.2}) \xrightarrow{h \rightarrow 0} \int \p_y \{ y^m \chi_1(y) u_y \} \cdot u_{yy}.
\end{align}

Performing similar calculations for each of the terms yields the desired result. The boundary conditions (\ref{bc.bc.1}) are satisfied by integrating by parts (\ref{stokes.weak}) against a test function, justified as $[u,v]$ are strong solutions, and comparing the boundary terms with (\ref{weak.form.1}).

\end{proof}

Near the boundary $y = 0$, standard Stokes theory (applicable due to cutoff $\chi_1(y)$) implies: 
\begin{lemma} Solutions $[u,v]$ to the system (\ref{stokes.weak}) satisfy the following estimate: 
\begin{align}
||u,v||_{H^{\frac{3}{2}}_{loc}}^2 \lesssim ||f,\sqrt{\eps} g||_{L^2}^2 + ||\Big\{\frac{a_L}{\sqrt{\eps}}, b_L \Big\}y^{2m}||_{L^2(x = L)}^2.
\end{align}
\end{lemma}

To summarize, we have established that: 
\begin{corollary} For $m, \alpha > 0$, the map $S_{\alpha,m }^{-1}: [L^2]^{\times 2} \times [L^2(x = L)]^{\times 2} \rightarrow [H^2_{m}  \cap H^{\frac{3}{2}}_{loc} \Big]^2$ is well defined, and returns a solution to the system (\ref{stokes.weak}) which satisfies the boundary conditions specified by the third and fourth inputs of $S_{\alpha, m}$.
\end{corollary}

We now define: 
\begin{align} \n
T[u,v] = &u_s u_x + u_{sx}u  + v_s u_y + u_{sy}v \\
& u_s v_x + v_{sx}u + v_s v_y + v_{sy}v.
\end{align}

We will study: 
\begin{align} \n
&S_{\alpha, m}[u,v] + T[u,v] = (f,g) \Rightarrow \\ \label{fred.1}
&[u,v] + S_{\alpha, m}^{-1}\Big[T[u,v], a_L, b_L\Big] = S_{\alpha, m}^{-1} \Big[f,g, a_L, b_L\Big]
\end{align}

as an equality in $H^1 \times H^1$. 
\begin{lemma} For $m > 0$, we have the following compact embedding:
\begin{align}
H^2_m \cap H^{\frac{3}{2}}_{loc} \subset \subset H^1.
\end{align}
\end{lemma}
\begin{proof}
The proof follows from a standard argument, see for instance \cite[P. 145, Lemma 13.1]{Iyer}. 
\end{proof}

As a direct consequence, $S_{\alpha,m}^{-1}T$ is a compact operator on $H^1$. An application of the Fredholm Alternative shows that to produce an $H^1$ solution of (\ref{fred.1}), we must rule out nontrivial solutions to the homogeneous problem, which occurs when $f = g = a_L = b_L = 0$. For this purpose, we give a-priori estimates of the problem (\ref{fred.1}), under the hypothesis that $[u,v] \in H^1$. For such functions, we automatically know that $[u,v] \in H^2_m$ due to (\ref{h2.lax}).

\begin{lemma}[Energy Estimates] Solutions $[u,v] \in H^2_m$ to the system (\ref{stokes.weak}) satisfy the following energy estimate: 
\begin{align}\n
& ||u_y, \sqrt{\eps}u_x, \eps v_x||_{L^2}^2 + \alpha ||\{u_y, \sqrt{\eps}v_y, \eps v_x \} \cdot \chi y^m||_{L^2}^2 \\  \label{alpha.est.1} & \hspace{5 mm} \lesssim \bigO(L) ||u_x, \sqrt{\eps}v_x||_{L^2}^2 + \mathcal{R}_1  + ||\Big\{\frac{a_L}{\sqrt{\eps}}, b_L \Big\}y^{2m}||_{L^2(x = L)}^2.
\end{align}
\end{lemma}
\begin{proof}
This follows upon testing the system (\ref{fred.1}) against $[\varphi^{(n)}, \phi^{(n)}]$, where the sequence $[\varphi^{(n)}, \phi^{(n)}] \xrightarrow{H^1_m} [u,v]$, and repeating the energy estimate in Proposition \ref{prop.energy}.  

\end{proof}

\begin{lemma}[Positivity Estimates] Let $m = 1$. Then solutions $[u,v] \in H^2_m$ to the system (\ref{stokes.weak}) satisfy the following estimate: 
\begin{align} \n
||v_y, \sqrt{\eps}v_x||_{L^2}^2 + ||\sqrt{\eps}u_x||_{L^2(x = L)}^2 &\lesssim ||u_y||_{L^2}^2 + \bigO(v^0_e) ||u_y \cdot y, \sqrt{\eps}v_y y||_{L^2}^2 \\ \n  &+\alpha ||\{u_y, \sqrt{\eps}v_y, \eps v_x \} \cdot \chi y^m||_{L^2}^2  + \mathcal{R}_2 \\  \label{lax.pos.1}
& + ||\Big\{\frac{a_L}{\sqrt{\eps}}, b_L, \p_y b_L \Big\}y^{2m}||_{L^2(x = L)}^2.
\end{align}
\end{lemma}
\begin{proof}

We must perform estimates on the new, weighted quantities appearing from (\ref{stokes.weak}). Temporarily omitting the prefactor of $10$, we have:
\begin{align} \label{lax.3}
+ \alpha \int \p_y \{u_y y^{2m} \} \cdot \p_y \frac{v}{u_s} &= - \alpha \int u_y y^{2m} \p_{y} \Big[ \frac{v_y}{u_s} - v \frac{u_{sy}}{u_s^2}\Big] \\ \n
& = - \alpha \int u_y y^{2m} \frac{v_{yy}}{u_s} + \alpha \int u_y y^{2m} v_y \frac{u_{sy}}{u_s^2} \\  \n
& - \alpha \int u_y y^{2m} v \p_y \{ \frac{u_{sy}}{u_s^2} \} \\ \n
& = + \alpha \int u_y y^{2m} \frac{u_{xy}}{u_s} + \alpha \int u_y y^{2m} v_y \frac{u_{sy}}{u_s^2} \\ \n
& - \alpha \int u_y y^{2m} v \p_y \{ \frac{u_{sy}}{u_s^2} \} 
\\ \n & = - \alpha \int u_y^2 y^{2m} \p_x \{ \frac{1}{u_s} \}  + \int_{x = L} \alpha u_y^2 y^{2m} \frac{1}{2u_s} \\
& + \alpha \int u_y y^{2m} v_y \frac{u_{sy}}{u_s^2} - \alpha \int u_y y^{2m} v \p_y \{ \frac{u_{sy}}{u_s^2} \} .
\end{align}

The boundary contribution above is positive, whereas the other terms can all be estimated by the $\alpha$ term in (\ref{lax.pos.1}). We need to justify the integration by parts leading to the equality in (\ref{lax.3}). For this we notice that our solution is in $H^2_m$, and so both the left and right-hand sides of (\ref{lax.3}) are in $L^1$. This then justifies the following limit: 
\begin{align} \n
\int \p_y \{ u_y y^{2m} \} \cdot \p_y \frac{v}{u_s} &= \lim_{M \rightarrow \infty} \int_{y = 0}^M  \p_y \{ u_y y^{2m} \} \cdot \p_y \frac{v}{u_s} \\ \n
& = \lim_{M \rightarrow \infty}  \Big[ - \int_0^M u_y y^{2m} \p_{yy} \frac{v}{u_s} + \int_{y = M} u_y y^{2m} \p_y \frac{v}{u_s} \Big] \\
& = - \int  u_y y^{2m} \p_{yy} \frac{v}{u_s}, 
\end{align}

where the limit of the boundary contribution vanishes as $||u_y y^m||_{L^2_x}^2$ and $||v_y y^m ||_{L^2_x}$ are $H^1_y$ functions, according to the definition of $H^2_m$. We similarly have:
\begin{align} \n
\int -2\alpha &\p_y \{ \chi_1 y^{2m} v_y \} \cdot \eps \p_x \{ \frac{v}{u_s} \} \\
& = \int 2\alpha \chi y^{2m}  \eps v_y \Big( \frac{v_{xy}}{u_s} + v_y \p_x \frac{1}{u_s} + v_x \p_y \frac{1}{u_s} + v \p_{xy} \frac{1}{u_s} \Big).
\end{align}

Finally: 
\begin{align} \n
- \int \alpha \eps &\p_x \{ \chi y^{2m} \{u_y + \eps v_x \} \} \cdot \p_x \{ \frac{v}{u_s} \}  \\ \n
& = - \int \alpha \eps \chi y^{2m} u_{xy} \p_x \{ \frac{v}{u_s} \} - \int \alpha \eps^2 v_{xx}v_x \chi y^{2m} \frac{1}{u_s} \\ \label{st.2}
& \hspace{10 mm} -\int \eps^2 \alpha v_{xx}v \p_x \frac{1}{u_s} \chi y^{2m}.
\end{align}

The latter two terms in (\ref{st.2}) are estimated according to standard calculations. For the first term: 
\begin{align} \n
-\int \alpha \eps &\chi y^{2m} u_{xy} \p_x \{ \frac{v}{u_s} \} = \int \alpha \eps u_x \p_y \{ \chi y^{2m} \p_x \{ \frac{v}{u_s} \} \} \\ \n
& = \int \eps \alpha u_x \chi y^{2m} \Big[ \frac{v_{xy}}{u_s} + v_x \p_y \frac{1}{u_s} + v_y \p_x \{ \frac{1}{u_s} \} + v \p_{xy} \frac{1}{u_s}\Big] \\
& \hspace{10 mm} + \int \eps \alpha u_x \p_x \frac{v}{u_s} \p_y \{ \chi y^{2m} \}.
\end{align}

For the final term above, we must use that $m = 1$: 
\begin{align}
 \int \eps \alpha u_x \p_x \frac{v}{u_s} \p_y \{ \chi y^{2m} \} \le \alpha ||\sqrt{\eps} u_x y||_{L^2} ||\sqrt{\eps}v_x||_{L^2}.
\end{align}

The remaining terms can all be estimated similarly to estimate (\ref{pos.est.1}). 

\end{proof}

We first recall the definition of $\mathcal{R}_1$ given in (\ref{R.1}).
\begin{lemma}[Weighted Estimate] Solutions $[u,v] \in H^2_m$ to the system (\ref{stokes.weak}) satisfy the following estimate: 
\begin{align} \n
||\Big\{u_{yy}, &\sqrt{\eps}u_{xy}, \eps u_{xx} \Big\} \cdot y||_{L^2}^2 + ||\{u_y, \sqrt{\eps}u_x \} y||_{L^2}^2 + ||\{u_y, \sqrt{\eps}u_x \} y||_{L^2(x = L)}^2 \\ \n
& + \alpha ||\Big\{u_{yy}, \sqrt{\eps} u_{xy}, \eps u_{xx} \Big\}\cdot y^{m+1}||_{L^2}^2 \lesssim \alpha ||\{u_y, \sqrt{\eps}v_y, \eps v_x \} \cdot \chi y^m||_{L^2}^2 \\ \n
& + ||v_y, \sqrt{\eps}v_x||_{L^2}^2 + ||u_y||_{L^2}^2 + ||\{ a_L, \p_y a_L, b_L, \p_y b_L \} \langle y \rangle^2||_{L^2(x = L)}^2 \\  \label{lax.weight.1}
& + ||\sqrt{\eps}u_x||_{L^2(x = L)}^2 + \mathcal{R}_1.
\end{align}
\end{lemma}

\begin{proof}
For this step, we can apply a cut-off $\chi_N(y) = \chi(\frac{y}{N})$, and take $N \rightarrow \infty$. Due to the cut-off, there is no need to justify contributions from $y = \infty$. Consider the new term: 
\begin{align} \n
- \alpha \int \p_{yy} &\{ y^{2m} \chi(y) u_y \} \cdot \p_y \{u y^2 w(x) \} \chi_N(y) \\ \n
& = + \alpha \int \p_y \{ y^{2m} \chi(y) u_y \} \cdot \p_{y} \{ u y^2 w \} \frac{1}{N} \chi_N'(y) \\
& = +\alpha \int \chi \chi_N y^{2m+2} u_{yy}^2  + \alpha \bigO(||u_y||_{H^1_m}^2).
\end{align}

Analogous calculations can be performed for the remaining $\alpha-$ terms from (\ref{stokes.weak}). For the remaining terms from (\ref{fred.1}), one can repeat the proof of Proposition \ref{weight.prop} with the additional cutoff term $\chi_N(y)$. We omit repeating those details. 
 
\end{proof}

Putting the above estimates, (\ref{alpha.est.1}), (\ref{lax.pos.1}), (\ref{lax.weight.1}) together gives the following uniform in $\alpha$ estimate:
\begin{align} \n
||u,v||_{\X}^2 &+ \alpha ||\Big\{u_{yy}, \sqrt{\eps} u_{xy}, \eps u_{xx} \Big\}\cdot y^{m+1}||_{L^2}^2 +  \alpha ||\{u_y, \sqrt{\eps}v_y, \eps v_x \} \cdot \chi y^m||_{L^2}^2 \\ \label{u.X}
&  \lesssim \mathcal{R}_1 + \mathcal{R}_2 + \mathcal{R}_3 + C(a_L, b_L). 
\end{align}

Taking the forcing $f = g = a_L = b_L = 0$ (thus $\mathcal{R}_i = 0$), we can apply the Fredholm Alternative to conclude that there exists an $H^1$ solution $[u,v]$ to the problem (\ref{fred.1}). Such a solution is automatically $H^2$ by (\ref{h2.lax}), and so is a strong solution. The final task is to establish a solution to our original system (\ref{NSR.1}) - (\ref{NSR.3}), which can be achieved as a weak limit in $\X$ as $\alpha \rightarrow 0$ using the uniform in $\alpha$ estimate (\ref{u.X}). This then proves Proposition \ref{lin.exist}. Proposition \ref{nl.exist} then follows upon applying the Contraction Mapping Theorem when coupled with the main $\X$-estimate in Theorem \ref{th.z}.

\vspace{5 mm}

\textbf{Acknowledgements:} Supported in part by NSF Grant DMS-1611695. The author thanks Toan Nguyen for sharing an unpublished note regarding the contents of Lemmas \ref{lemma.a1} through Lemma \ref{lemma.a4} of Appendix \ref{app.construct}, and also for informing him of the reference \cite{CS}.

\end{document}